\documentclass[11pt]{amsart}
\usepackage[dvipsnames]{xcolor}

\usepackage{amsfonts,amsthm,amsmath,amssymb,etoolbox,comment,graphicx,array,tabu,commath,latexsym,environ,tikz,tikz-cd,stackengine,diagbox,caption}
\usepackage[ twoside,top=1in, bottom=1in, left=1in, right=1in]{geometry}

\usepackage{mathrsfs}
\usepackage{soul}
\usepackage{eso-pic}
\usepackage[toc,page]{appendix}
\usepackage{enumitem}
\usepackage{mathtools}

\usepackage[colorlinks, linkcolor = black, anchorcolor = black, citecolor = black, urlcolor=black]{hyperref}

\newcommand{\fg}{\mathfrak{g}}

\newcommand{\cE}{\mathcal{E}}

\newcommand{\cH}{\mathcal{H}}

\newcommand{\cK}{\mathcal{K}}
\newcommand{\cL}{\mathcal{L}}
\newcommand{\cM}{\mathcal{M}}
\newcommand{\cN}{\mathcal{N}}
\newcommand{\cO}{\mathcal{O}}
\newcommand{\cP}{\mathcal{P}}

\newcommand{\cR}{\mathcal{R}}
\newcommand{\cS}{\mathcal{S}}
\newcommand{\cT}{\mathcal{T}}
\newcommand{\cU}{\mathcal{U}}

\newcommand{\cX}{\mathcal{X}}
\newcommand{\cY}{\mathcal{Y}}
\newcommand{\cZ}{\mathcal{Z}}

\newcommand{\bB}{\mathbf{B}}

\newcommand{\bD}{\mathbf{D}}

\newcommand{\bF}{\mathbf{F}}

\newcommand{\bL}{\mathbf{L}}

\newcommand{\bS}{\mathbf{S}}

\newcommand{\bV}{\mathbf{V}}

\newcommand{\bg}{\mathbf{g}}
\newcommand{\bbS}{\mathbf{S}}

\newcommand{\RP}{{\mathbb{RP}}}
\newcommand{\R}{\mathbb R}
\newcommand{\C}{\mathbb C}
\newcommand{\Z}{\mathbb Z}

\newcommand{\D}{\mathbb D}

\newcommand{\id}{\mathrm {id}}

\newcommand{\Hom}{\mathrm {Hom}}

\newcommand{\area}{\operatorname{area}}

\newcommand{\dist}{\operatorname{dist}}
\newcommand{\dmn}{\mathrm{dmn}}
\newcommand{\UConf}{\mathrm{UConf}}

\newcommand{\hConf}{\widehat\UConf}

\renewcommand{\Re}{\mathrm{Re}}

\renewcommand{\tilde}{\widetilde}

\newcommand{\link}{\operatorname{link}}
\newcommand{\spt}{\operatorname{spt}}

\newcommand{\sing}{\operatorname{sing}}
\newcommand{\x}{\times}

\newcommand{\ins}{\mathrm{in}}
\newcommand{\out}{\mathrm{out}}

\newcommand{\Ker}{\mathrm{ker}}

\newcommand{\Diff}{\operatorname{Diff}}

\newcommand{\GS}{\mathcal{S}}

\newcommand{\inte}{\operatorname{int}}
\newcommand{\Int}{\operatorname{int}}
\newcommand{\Imag}{\operatorname{Im}}
\newcommand{\Span}{\operatorname{span}}
\newcommand{\ord}{\operatorname{ord}}
\newcommand{\Neg}{\mathfrak{n}}

\newcommand{\bq}{\mathbf{q}}

\newcommand{\bx}{\mathbf{x}}

\usepackage[style=alphabetic, backend=biber, sorting=nyt, url=false ]{biblatex}
\title{An enumerative min-max theorem for minimal surfaces}

\author{Adrian Chun-Pong Chu, Yangyang Li, Zhihan Wang} 

\address{Cornell University, Ithaca, NY 14853, USA}
\email{cc2938@cornell.edu}
\address{University of Notre Dame,
Notre Dame, IN 46556, USA}
\email{yangyangli@nd.edu}
\address{Cornell University, Ithaca, NY 14853, USA}
\email{zw782@cornell.edu}

\date{\today}

\numberwithin{equation}{section}
\newtheorem{thm}{Theorem}[section]
\newtheorem{cor}[thm]{Corollary}

\newtheorem{prop}[thm]{Proposition}
\newtheorem{lem}[thm]{Lemma}

\theoremstyle{definition}
\newtheorem{defn}[thm]{Definition}

\newtheorem{rmk}[thm]{Remark}
\newtheorem{nota}[thm]{Notation}

\newtheorem{cond}{Condition}

\newtheorem*{question*}{Question}

\begin{document}

\begin{abstract} We prove an enumerative min-max theorem that relates the number of genus $g$ minimal surfaces in  3-manifolds of positive Ricci curvature to topological properties of the set of embedded surfaces of genus $\leq g$, possibly with finitely many singularities. This completes a central component of our program of using topological methods to enumerating minimal surfaces with prescribed genus \cite{chuLi2024fiveTori,chu2025arbitraryGenus,chuLiWang2025optimalFamily}.

As an application, we show that every  3-sphere of positive Ricci curvature contains at least 4 embedded minimal surfaces of genus 2.
\end{abstract}
\maketitle

\setcounter{tocdepth}{1}
\tableofcontents
  
\section{Introduction}
Enumerative problems in differential geometry have a long history. Here are some well-known conjectures. 
\begin{itemize}
    \item Every Riemannian $2$-sphere has at least $3$ simple closed geodesics (Poincar\'e, 1905).  
    \item Every Riemannian $3$-sphere has at least $4$ embedded minimal spheres (S.-T. Yau \cite{Yau82}). 
    \item Every Riemannian $3$-sphere has at least $5$ embedded minimal tori (B. White \cite{White89}). 
\end{itemize}
Among these three conjectures, only the first one was fully solved: Building on the work of Lyusternik-Schnirelmann \cite{LS47}, M. Grayson gave a proof \cite{Grayson89}  using  curve shortening flow (see also \cite{Lusternik1947,Klingenberg1977,Ballmann3geodesic1978,Jost3geodesic1989,Taimanov3geodesic1992}). Yau's conjecture was solved in the case of bumpy metric and the case of positive Ricci curvature by Wang--Zhou \cite{WangZhou23FourMinimalSpheres}: They developed a multiplicity one theorem for Simon--Smith min-max theory, using  regularity results in \cite{wangZhou2025improvedMembranes}, and   
~\cite{SS23} by Sarnataro--Stryker (see also \cite{ Smith82, Whi91, HK19}). Recently, the first two authors proved  B. White's conjecture for 3-spheres with positive Ricci curvature \cite{chuLi2024fiveTori}. Regarding more works on the construction of minimal submanifolds of controlled topological types under general metrics, see also   \cite{Str84, GJ86,Grayson89, Whi91, Zho16, HK19, bettiolPiccione2023bifurcationsCliffordTorus, HK23, Ko23a, Ko23b,bettiolPiccione2024nonplanarMinSpheres,LiWang2024NineTori,liWangYao2025minimalLensSpace}.

In a series of papers \cite{chuLi2024fiveTori,chu2025arbitraryGenus,chuLiWang2025optimalFamily}\footnote{The current paper and \cite{chuLiWang2025optimalFamily}  together supersede  \cite{chuLiWang2025genus2PartI}.}, the authors proposed a program that constructs minimal surfaces with prescribed genus based on the topological structure of the  set of genus $g$ embedded surfaces, possibly with finitely many singularities. Here is the setting:
Let $M$ be any closed orientable  3-manifold, and $g$ a non-negative integer. Let $\cS(M)$ denote the set of all orientable, separating, finite-area surfaces in $M$ that are smoothly embedded  except at possibly  finitely many points (see \S 2 for the precise definition). Then, we consider the subset $\cS_{\leq g}(M)\subset\cS(M)$ of elements with    genus $\leq g$. In \cite{chuLiWang2025optimalFamily}, we showed that given any closed orientable Riemannian 3-manifold $(M,\bg)$ of positive Ricci curvature, if there exists some map $\Phi:X\to \cS_{\leq g}(M)$ that {\it cannot be deformed via  pinch-off processes} to become a map into $\cS_{\leq g-1}(M)$ (see \S 2 therein), then $(M,\bg)$ contains  {\it at least one}   minimal surface of genus $g$.  

In this paper we   generalize the above result. Namely, by further utilizing the relative topological structure of the pair $(\cS_{\leq g}(M),\cS_{\leq g-1}(M))$, we  produce {\it multiple} minimal surfaces of genus $g$.   This completes a central component
of our program to develop a topological machinery for  enumerating minimal surfaces with  prescribed genus.

Below, we consider  homology groups and  cohomology groups  under $\Z_2$-coefficients. Note, the cap product $\frown$ in item (1) is indeed well-defined, as we show in Lemma \ref{lem:cupProdWellDef}. 
Readers may refer to Figure \ref{fig:enuMinMax} for a schematic.  
 
\begin{thm}[Enumerative min-max theorem]\label{thm:enumerMinMax}
    Let $(M,\bg)$ be a closed orientable Riemannian $3$-manifold with positive Ricci curvature, and let $g$ be a positive integer. Let $$\Phi:(X,Z)\to (\cS_{\leq g}(M),\cS_{\leq g-1}(M))$$ be a Simon--Smith family, where $X$ is a finite simplicial complex and $Z\subset X$ a subcomplex. Suppose there exist some homology class $w\in H_k(X,Z)$ and $p$ cohomology classes $\lambda_i\in H^{k_i}(X\backslash Z)$, $i=1,...,p$, with the following properties.
    \begin{enumerate}
        \item\label{item:enuMinMaxY} Let $(Y,\partial Y)\subset (X,Z)$ be any relative $(k-k_1-\dots-k_p)$-subcomplex such that $$[Y]=w\frown(\lambda_1\smile\dots\smile\lambda_p)\in H_{k-k_1-\dots-k_p}(X,Z).$$ Then the subfamily $\Phi|_{Y}$ cannot be deformed via pinch-off processes to a map into $\cS_{\leq g-1}(M)$.
        \item\label{item:enuMinMaxC} For any embedded smooth genus $g$ surface $S\subset M$, there exists some $r>0$ such that: For each $i=1,\dots,p$,  if $C\subset X\backslash Z$ is any $k_i$-cycle such that   $\lambda_i([C])\ne 0$, then $\Phi|_C$ cannot be deformed via pinch-off processes to a family $r$-close to $S$ under the $\bF$-distance.
    \end{enumerate}
    Then $(M,\bg)$ contains at least $p+1$ orientable embedded minimal surfaces with genus $g$ and area at most $\displaystyle\max_{x\in X}\area(\Phi(x))$.
\end{thm}

\begin{figure}[h] 
    \centering
\includegraphics[width=0.6\linewidth]{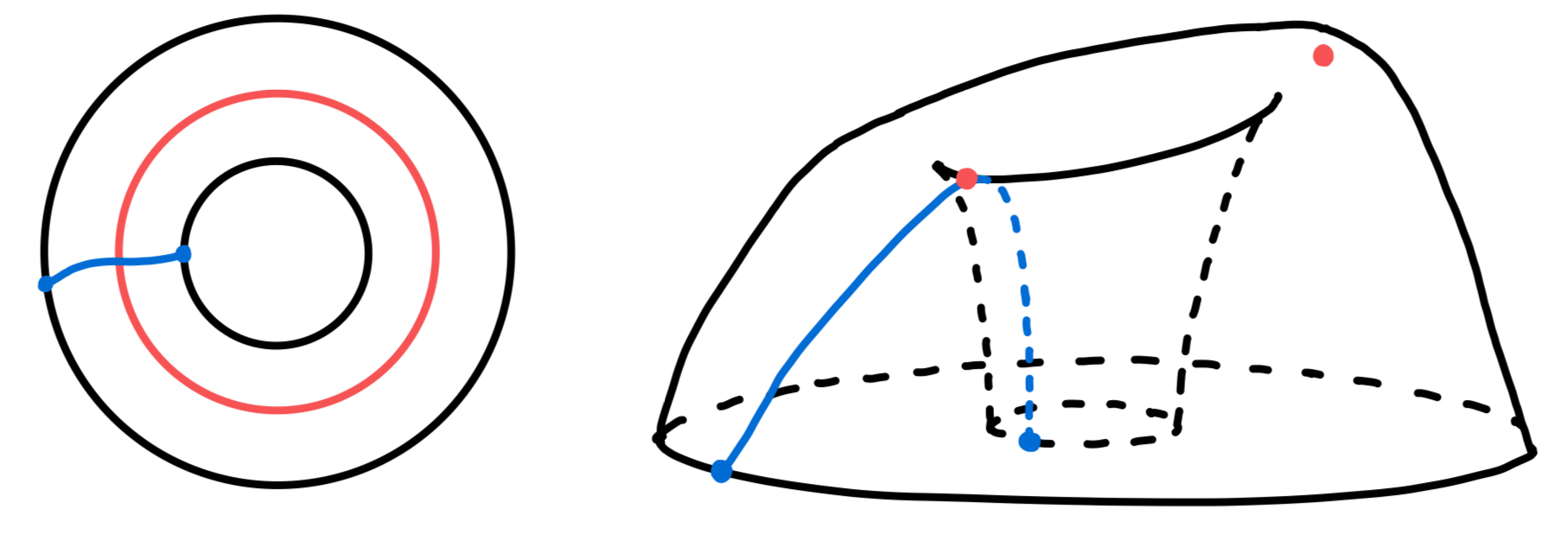}
\caption{On the left is an example where $X$ is the annulus, $Z$ is the two boundary circles, $Y$ is the blue path, $w$ is the fundamental class $[X]$, $p=1$, and $\lambda_1$ is the Poincar\'e dual of  $[Y]$,  meaning $[Y]=[X]\frown \lambda_1$. Note the pairing of $\lambda_1$ with the red loop is 1. On the right, we have a schematic where the graphical surface represents $\cS_{\leq g}(M)$, with the two boundary circles being $\cS_{\leq g-1}(M)$. The blue path is the family $\Phi|_Y$, and the two red critical points are the minimal surfaces we are searching for.}
\label{fig:enuMinMax}
\end{figure}

Importantly, it can be readily checked  that assumptions \eqref{item:enuMinMaxY} and \eqref{item:enuMinMaxC}   above  are purely topological: No geometric information about the ambient metric $\bg$ is needed. Note, the $\bF$-distance is the borrowed from the space $\cZ_2(M;\bF;\Z_2)$ of mod 2 cycles, which is the sum of the flat distance and the varifold distance. 
The proof of this theorem crucially relies on \cite[Theorem 1.4]{chuLiWang2025optimalFamily}, which states that in a generic metric with positive Ricci curvature, every Simon--Smith family can be deformed via pinch-off processes into  a certain ``topologically optimal family" \cite[Theorem 1.4]{chuLiWang2025optimalFamily}. This result is based on numerous previous foundational works on min-max theory \cite{Smith82, CD03, DP10,   Ket19, Zho20, MN21,SarnataroStryker2023Optimal,WangZhou23FourMinimalSpheres}. We note that in \cite{LiWang2024NineTori} Xingzhe Li and Zhichao Wang also utilized certain relative cohomology groups to construct minimal tori.

\begin{rmk} Let us compare Theorem \ref{thm:enumerMinMax} with the   previous work \cite{chuLi2024fiveTori}, in which  we constructed 5  minimal tori in any 3-spheres with positive Ricci curvature. There we relied crucially on the non-existence of a 5-sweepout on $S^3$ that consists of only genus 0 elements. However, this line of argument is unavailable in   the higher genus case. 
\end{rmk}

Next, we apply  the above enumerative min-max theorem to the case of 3-spheres of positive Ricci curvature to construct multiple minimal surfaces of genus 2. The main challenge is to find a suitable map $\Phi:(X,Z)\to (\cS_{\leq 2}(S^3),\cS_{\leq 1}(S^3))$, together with  homology and cohomology classes, that satisfy the conditions of Theorem \ref{thm:enumerMinMax}.

\begin{thm}\label{thm:main}
    Every Riemannian $3$-sphere of positive Ricci curvature contains at least $4$ embedded, orientable minimal surfaces of genus $2$.
\end{thm}

Let us record the known results in this direction. Denote by ${\frak n}_g(S^3)$ the optimal number such that every Riemannian 3-sphere of positive Ricci curvature must contain at least ${\frak n}_g(S^3)$ orientable, embedded  minimal surfaces of genus $g.$ Then we have:

\medskip

\begin{center}
\begin{tabular}{ |c| c| c| c| }
 \hline$\frak n_0(S^3)$ & $\frak n_1(S^3)$ & $\frak n_2(S^3)$ & $\frak n_g(S^3)$ for $g\geq 3$ \\ 
 \hline 4 \cite{WangZhou23FourMinimalSpheres} & $\geq 5$ \cite{chuLi2024fiveTori} & $\geq 4$ (Theorem \ref{thm:main}) & $\geq 1$ \cite{chu2025arbitraryGenus}\\
 \hline
\end{tabular} 
\end{center}

\medskip

Conjecturally $\frak n_1(S^3)=5$. It will be very interesting to know the exact value of   ${\frak n_g}(S^3)$ for any particular $g\geq 2$. We believe that Theorem \ref{thm:enumerMinMax} is general enough to be used as a ``black box"   for exploring the higher genus cases: It reduces the problem of lower bounding  ${\frak n_g}(S^3)$ to finding suitable maps into $\cS_{\leq g}(S^3)$. 

The family  we use to prove Theorem \ref{thm:main} is inspired by the symmetry of the Lawson surfaces $\xi_{2,1}$ \cite{Law70}, which is the only known genus 2 embedded minimal surface in the standard 3-sphere $\mathbb S^3$. To speculate the $g=3$ case, we   recall that there are two known genus 3 minimal surfaces in $\mathbb S^3$: The Lawson surface $\xi_{3,1}$  and the Karcher-Pinkall-Sterling genus 3 surface \cite{karcherPinkallSterling1988new}. We expect that to get a good lower bound for ${\frak n_3}(S^3)$, the family for which we apply Theorem \ref{thm:enumerMinMax} to should encompass surfaces that resemble both of these two minimal surfaces, and also neck-pinch phenomenon from them.

\subsection{Sketch of proof.} 
\subsubsection{ Theorem \ref{thm:enumerMinMax}} We first recall that, when the ambient metric is generic with positive Ricci curvature,  \cite[Theorem 1.4]{chuLiWang2025optimalFamily} shows that every family of surfaces in $\cS(M)$ can be deformed into some topologically optimal family. 
We first perturb the ambient metric such that     \cite[Theorem 1.4]{chuLiWang2025optimalFamily} is applicable,  {\it while fixing the number of orientable genus $g$ minimal surfaces}. Applying this theorem to the given family $\Phi$, we obtain some minimal surfaces  $\Gamma_i$ of genus $\leq g$ and subcomplexes $D_i\subset X$ satisfying the following properties. 

Suppose that  $m$ orientable genus $g$ were produced. We can divide the subcomplexes $D_i$ into two collections: Those that correspond to  orientable genus $g$ minimal surfaces, and the other ones. We denote the first collection by $D_1,...,D_m$, and take the union of the second collection to obtain some subcomplex $D$. Hence, we have
\begin{itemize}
\item $m$ orientable minimal surfaces of genus $g$, $\Gamma_1,\dots,\Gamma_m$ with area $\leq \max_{x}\area(\Phi(x)),$
\item and associated subcomplexes $D_1,D_2,\dots,D_m\subset X$ that, together with $D$, cover $X$,

\end{itemize}
such that:
\begin{enumerate} 
    \item\label{item:sketchItemD} $\Phi|_D$ can be deformed via pinch-off processes to become some map into $\cS_{\leq g-1}(M)$.
    \item\label{item:sketchItemDi} For each $j=1, \dots, m$, $\Phi|_{D_j}$ can be deformed via pinch-off processes to become  some  family $\Xi_j:D_j\to\cS_{g}(M)$  such that each member of $\Xi_j$ is close  to $\Gamma_j$ and has genus $g$.
\end{enumerate}

Then to prove Theorem \ref{thm:enumerMinMax} it suffices to show $m\geq p+1$.
Suppose by contradiction $m\leq p$. By \eqref{item:sketchItemDi}, we know that for every $j=1,...,m$ we have $\lambda_i|_{D_j}=0$, for otherwise  item \eqref{item:enuMinMaxC} of Theorem \ref{thm:enumerMinMax} would be contradicted. Thus, together with $m\le p$, we know $$(\lambda_1\smile...\smile \lambda_p)|_{D_1\cup...\cup D_m}=0.$$  Then by some elementary algebraic topology one can deduce that  $D$ must detect the class
\begin{equation} \label{eq:capProd}   w\frown(\lambda_1\smile\dots\smile\lambda_p)\in H_{k-k_1-\dots-k_p}(X,Z).
\end{equation} More precisely, we can find some $[Y]\in H_{k-k_1-...-k_p}(D,D\cap Z)$ that gets pushforwarded to the class \eqref{eq:capProd}   under the inclusion $(D,D\cap Z)\hookrightarrow (X,Z)$. Then  by item \eqref{item:sketchItemD} above, $\Phi|_D$, and thus $\Phi|_Y$ in particular, can be deformed via pinch-off processes to become a map into $\cS_{\leq g-1}(M)$. This contradicts Theorem \ref{thm:enumerMinMax} \eqref{item:enuMinMaxY}.

\subsubsection{ Theorem \ref{thm:main}}\label{sect:sketchFourGenus2}  In order to construct 4 minimal surfaces of genus 2 using Theorem \ref{thm:enumerMinMax}, we need a suitable map into  $\cS_{\leq 2}(S^3)$.  We first define  in the unit $3$-sphere $S^3\subset\R^4$ a 7-parameter family as follows. 
We introduce  a coordinate system $(x_1,x_2,\alpha)$ for $S^3$ (on the complement of a great circle $\{x_3=x_4=0\}$), defined by the relation 
$$(x_1,x_2,x_3,x_4)=(x_1,x_2,\sqrt{1-x_1^2-x_2^2}\cos\alpha,\sqrt{1-x_1^2-x_2^2}\sin\alpha)\,,$$
where $(x_1,x_2)$ lies in the open unit disc in $\R^2$ and $\alpha\in\R/2\pi\Z$. Note that $(0,0,\alpha)$ parametrizes the great circle $C = \{x_1 = x_2 = 0\}$. Now, we define a family $\Psi:\RP^5\x\D\to\cS_{\leq 2}(S^3)$ by sending each pair $a=[a_0:a_1:\dots:a_5]\in\RP^5$ and $(r,\theta)\in \D$ to the zero set  
\begin{align*}
\{a_0+a_1x_1+a_2x_2+a_3x_3+a_4x_4+a_5\left[x_1x_2+\epsilon(r\cos(\theta+2\alpha)+(1-r)\cos 3\alpha) \right]=0\}\subset S^3,
\end{align*}
where $\epsilon>0$ is some small number (in our actual argument, this should be some small function). Geometrically, this family should be viewed as a desingularization of the $\RP^5$-family $\Phi_5$ given by 
$$\Phi_5(a):=\{a_0+a_1x_1+a_2x_2+a_3x_3+a_4x_4+a_5 x_1x_2=0\}\subset S^3,$$
such that all members have at most finitely many singularities.
Notably, the surface $\Psi([0:...:0:1],0)$ has the same symmetry group as the genus 2 Lawson  minimal surface $\xi_{2,1}$, which is the dihedral group $D_{24}$ of 24 elements.

Then, we define a  $13$-parameter  family $\bar \Xi:\RP^5\x\D\x SO(4)\to\cS_{\leq 2}(S^3)$ by rotating all members of $\Psi$: We set $ \bar \Xi(a,z,R):=R\cdot\Psi(a,z)$ where $R \cdot \Psi(a,z)$ denotes the image of  $\Psi(a,z)\subset S^3$ under the rotation  $R \in SO(4)$. This family has a lot of symmetry, allowing  us to take   quotient. Namely, we will define a certain $D_{24}$-action on $\RP^5\x\D\x SO(4)$ such that $\bar \Xi((a,z,R)\cdot g)=\bar \Xi(a,z,R)$ for all $g\in D_{24}$ and $(a,z,R)$. Hence, letting $\cX:=(\RP^5\x\D\x SO(4))/D_{24}$, $\bar\Xi$ induces a family  $\Xi:\cX\to\cS_{\leq 2}(S^3).$ Note, by the definition of $\Psi$, $\Xi|_{\partial\cX}$ maps into $\cS_{\leq 1}(S^3)$.

In order to   apply Theorem \ref{thm:enumerMinMax}, let us understand the topology of $\cX$. It can be checked that the natural projection map $\pi$ from $\cX$ onto $B:=SO(4)/D_{24}$ is a fiber bundle. By some standard topological  argument, $B\cong S^3/Q_{48}\x S^3$ where $Q_{48}$ is the dicyclic group of order 48. 
We then  introduce two cohomology classes in $\cX$:
\begin{itemize}
    \item It is known that there exists some  element $\hat\alpha\in H^1(S^3/Q_{48})$ for which $\hat\alpha^3\ne 0$. Let $\alpha\in H^1(\cX)$ be the pullback of $\hat\alpha$ under $\Pi_1\circ\pi:\cX\to S^3/Q_{48}$ where $\Pi_1:B\to S^3/Q_{48}$ is the projection onto the first factor.
    \item Fix some $p \in S^3$, and let $B_0:= S^3/Q_{48}\x \{p\}\subset  S^3/Q_{48}\x   S^3 \cong B$. Let $\beta\in H^3(\cX)$ be the pullback $\pi^*(\beta_B)$, where $\beta_B\in H^3(B)$ is the Poincar\'e dual of $[B_0]\in H_3(B)$.
\end{itemize}
Then, it is easy to see that the relative homology class
$[\cX]\frown(\alpha^3\smile\beta)\in H_7(\cX,\partial\cX)$ is induced by the fibers of $\pi:\cX\to B.$

Finally, let $\cX_0:=\pi^{-1}(B_0)\subset\cX$. Note $\partial\cX_0=\cX_0\cap\partial\cX$. We apply Theorem \ref{thm:enumerMinMax} to the subfamily $$\Xi|_{\cX_0}:(\cX_0, \partial \cX_0)\to (\cS_{\leq 2}(S^3),\cS_{\leq 1}(S^3)),$$ with $p=3$ and the three cohomology classes all being $\alpha|_{\cX_0}\in H^1(\cX_0).$ To finish the proof, it suffices to verify condition \eqref{item:enuMinMaxY} and \eqref{item:enuMinMaxC} in Theorem \ref{thm:enumerMinMax}. 

For \eqref{item:enuMinMaxY}, let $[Y]=[\cX_0]\frown(\alpha_0)^3\in H_7(\cX_0,\partial \cX_0)$ and $\Phi'$ be obtained from $\Xi|_Y$ by deformation via pinch-off processes. We need to show that $\Phi'$ must still contain some genus $2$ element. Observe that,  when restricted onto any $\RP^5$-factor, $\Xi$ gives an Almgren-Pitts 5-sweepout, which thus cannot be deformed via pinch-off processes into $\cS_0(S^3)$ (see \cite[\S 3.5]{chuLi2024fiveTori}).
From this, together with some algebraic topology, it can be shown that there exists some relative 2-cycle $(D,\partial D)\subset (Y,Y\cap\partial \cX_0)$ such that $D$ is homologous to the $\D$ factor in $\cX$ and {\it every member of $\Phi'|_D$ has genus $1$ or $2$}. Finally, we analyze the boundary family $\Phi'|_{\partial D}$, and show that it is, in certain sense, homologically non-trivial in the set $\cS_1(S^3)$ of all genus 1 surfaces in $S^3$, and thus does not admit a fill in. Hence,  $\Phi'|_{ D}$ must contain some genus 2 element, as desired. (This step requires Hatcher's resolution of the Smale conjecture and certain  descriptions of the space of all Hopf links in $S^3$.)

As for item \eqref{item:enuMinMaxC}, we shall show that any such 1-cycle $C$ must correspond to a 1-sweepout under $\Xi$ in the Almgren-Pitts sense, so that $\Xi|_C$ cannot be deformed into any $r$-ball if $r$ is small enough.

\subsection{Organization} In \S \ref{sect:prelim}, we recall some preliminary results from our previous works \cite{chuLi2024fiveTori,chuLiWang2025optimalFamily}. A crucial result needed is    \cite[Theorem 1.4]{chuLiWang2025optimalFamily}. Based on that,  in \S \ref{sect:enuMinMax} we prove the enumerative min-max theorem, Theorem \ref{thm:enumerMinMax}. In \S \ref{sect:4Genus2}, we apply this theorem to construct 4 minimal surfaces of genus 2 in 3-spheres of positive Ricci curvature, proving Theorem \ref{thm:main}. Then the three remaining sections contain the proof of some results used in \S \ref{sect:4Genus2}. Namely, in \S \ref{sect:4Genus2}, the 4 genus 2 minimal surfaces are obtained by applying the enumerative min-max theorem, using a certain 13-parameter family $\Xi:\cX\to\cS_{\leq 2}(S^3)$. The construction and the properties of this family will be studied in detail in \S \ref{sect:proofThmPsi} and \S \ref{sect:XiProper}. In \S \ref{sect:topoXi}, we examine the topology of the parameter space of $\Xi$.

\subsection*{Acknowledgment} 
We would like to thank Chi Cheuk Tsang  and Boyu Zhang for numerous  discussions on topology, and Giada Franz for the discussions on min-max theory. We also want to thank Ryan Budney, Andr\'e Neves, Daniel Stern, Shmuel Weinberger, and Xin Zhou for their interests. 
The first author was supported by the NSF grant  DMS-1928930 while in  SLMath in Fall   2024. The first and the second author was partially supported by the AMS-Simons travel grant.

\section{Preliminaries}\label{sect:prelim}
\subsection{Notations}
Throughout this paper, all simplicial complexes are assumed to be finite. Moreover,  a {\it pure} simplicial $k$-complex   is a simplicial complex $X$ in which every simplex of dimension less than $k$ is a face of some simplex in $X$ of dimension exactly $k$.

Let $(M,\bg)$ be a closed orientable Riemannian 3-manifold. Below we briefly recall  the terminologies from \cite[\S 2]{chuLiWang2025optimalFamily}.

\begin{defn}[Punctate surface]\label{def:punctate_surf} Let $M$ be a closed orientable 3-manifold.
    A closed subset $S \subset M$ is a {\it punctate surface} in $M$ provided that: 
    \begin{enumerate}
        \item The 2-dimensional Hausdorff measure $\cH^2(S)$ is finite.
        \item\label{item:punctateOrientable} There exists a (possibly empty) finite set $P \subset S$ such that $S \setminus P$ is a smooth, orientable, embedded surface.
        \item\label{item:punctateSeparate} (Separating) Let $S_{\mathrm{iso}}$ be the set of isolated points of $S$. Then the complement of $S\backslash S_{\mathrm{iso}}$ is a disjoint union of two open regions of $M$, each having $S\backslash S_{\mathrm{iso}}$ as its topological boundary.
    \end{enumerate} 
    Any such set $P$ is called a {\it punctate set} for $S$. We denote by $\cS(M)$ the set of all punctate surfaces in $M$.
\end{defn}

Recall that $\cS(M)$ can be viewed as a subset of the cycle space $\cZ_2(M; \bF; \Z_2)$. In particular, we can put on $\cS_g(M)$ the $\bF$-distance, which is the sum of the flat distance and the varifold distance.
Moreover, we can define the notion of {\it genus} for punctate surfaces as in \cite[Definition 2.5]{chuLiWang2025optimalFamily}. We denote by $\cS_g(M)\subset\cS(M)$ (resp. $\cS_{\leq g}(M)$) the subset  of elements with genus $g$ (resp. $\leq g$).

\begin{defn}[Simon--Smith family]\label{def:Simon_Smith_family}
    Let $X$ be a cubical or simplicial complex. 
    A map $\Phi: X \to \GS(M)$ is called a {\em Simon--Smith family} with parameter space $X$ if the following hold.
    \begin{enumerate}[label=\normalfont(\arabic*)]
        \item \label{item:Hausdorff_cts} The  map $x\mapsto\cH^2(\Phi(x))$ is continuous. 
        \item \label{item:closedFamily} (Closedness) The family $\Phi$ is ``closed" in the sense that 
        $\{(x,p)\in X\x M: p\in \Phi(x)\}$
        is a closed subset of $X \x M$. 
        \item \label{item:SingPointsUpperBound} ($C^\infty$-continuity away from singularities) For each $x\in X$, we can choose a finite set $P_\Phi(x)\subset \Phi(x)$ such that:
        \begin{itemize}
            \item $\Phi(x)\backslash P_\Phi(x)$ is a smooth embedded orientable surface, i.e.  $P_\Phi(x)$ is a   punctate set for $\Phi(x)$.
            \item For any $x_0\in X$  and within any open set $U\subset\subset M \backslash P_\Phi(x_0)$, $\Phi(x)\to \Phi(x_0)$ smoothly whenever $x\to x_0$.
            \item $\displaystyle N(P_\Phi) := \sup_{x\in X} |P_\Phi(x)|<\infty.$
        \end{itemize}
    \end{enumerate}
\end{defn}

It follows from \cite[\S 2]{chuLi2024fiveTori} that every Simon--Smith family is $\bF$-continuous. 
 
The notion of {\it pinch-off process} will also be important. Roughly speaking, it refers to a family $\{\Gamma(t)\}_{t\in[a,b]}$ of elements in $\cS(M)$ given by any combination of isotopy, neck-pinches, and the shrinking of some connected components into points. The rigorous definition is stated in \cite[\S 2]{chuLiWang2025optimalFamily}. Now, let $\Phi,\Phi':X\to\cS(M)$ be Simon--Smith families. Suppose we have a Simon--Smith family $H:[0,1]\x X\to \cS (M)$ such that:
    \begin{itemize}
        \item  $H(0, \cdot)=\Phi$ and $H(1, \cdot)=\Phi'$.
        \item For each $x\in X$, $t\mapsto H(t,x)$ is a pinch-off process.
    \end{itemize}
Then   $H$ is called a {\it deformation via pinch-off processes} from $\Phi$ to $\Phi'$.

\subsection{Preliminary results} First, for readers' convenience, we recall Theorem 1.4 in \cite{chuLiWang2025optimalFamily}.

\begin{thm}[Existence of a topologically optimal family]\label{thm:repetitiveMinMax}
    Let $(M,\mathbf g)$ be a closed orientable Riemannian $3$-manifold of positive Ricci curvature, and $\Phi:X \to \mathcal{S}(M)$ be a Simon--Smith family of genus $\leq g$ with $g\geq 1$, and the min-max width $L:=\bL(\Lambda(\Phi))>0$. Assume Condition \ref{cond:A}  (which automatically holds if $\bg$ is a generic metric with positive Ricci curvature).
    Then for any $r > 0$, there exist:
    \begin{itemize}
        \item a deformation via pinch-off processes, $H:[0,1]\x X\to\cS_{\leq g}(M)$,
        \item  distinct, orientable, multiplicity one, embedded  minimal surfaces $\Gamma_1,\Gamma_2, \dots,\Gamma_l$ of genus $\leq g$ and area  $\leq L$,
        \item subcomplexes $D_0,D_1, \dots,D_l$ of $X$ (after refinement) that cover $X$,
    \end{itemize}
such that: 
    \begin{enumerate}
        \item \label{item:pinchOffH}$H(0,\cdot)=\Phi$.
        \item \label{item:optimal}  The family $\Phi':= H(1,\cdot)$ is such  that $\Phi'|_{D_i}$  is of genus $\leq \fg(\Gamma_i)$ for every $i=1, \dots,l$, and $\Phi'|_{D_0}$ is of genus $0$.
        \item \label{item:ti} For each $i=1,\dots,l$, there is some smooth function $f_i:D_i\to[0,1]$ such that for every $x\in D_i$, $$H(f_i(x),x)\in \cS_{\fg(\Gamma_i)}(M)\cap \bB^\bF_r(\Gamma_i).$$ 
    \end{enumerate}
\end{thm}

Below is the definition of Condition \ref{cond:A}. 
Let $\cO_{g,\leq L}(M,\bg)$ (resp. $\cN_{g,\leq L}(M,\bg)$) be the set of all orientable (resp. non-orientable) embedded minimal surfaces in $(M,\bg)$ with genus $g$ and area $\leq L$. Analogously, we define the sets $\cO_{\leq g,\leq L}(M,\bg)$ and $\cN_{\leq g,\leq L}(M,\bg)$ in which the genus of the minimal surfaces concerned is $\leq g$. 

\begin{cond}\label{cond:A}
Given a closed Riemannian 3-manifold $(M,\bg)$,  a positive integer $g$, and a real number $L>0$, we say that {\it Condition \ref{cond:A}} holds if the following holds:

    \begin{itemize}
    \item The set $\cO_{\leq g,\leq L}(M,\bg)$ is finite.
    \item Every element in $\cO_{\leq g-1,\leq L}(M,\bg)$ is non-degenerate.
    \item For  every  $\Gamma\in \cO_{g,\leq L}(M,\bg)$, numbers in the set 
        \begin{align*}
            &\{\area(\Gamma'):\Gamma'\in\cO_{\leq g-1,\leq L}(M,\bg)\cup\cN_{\leq g+1,\leq L}(M,\bg)
            \}
            \cup\{\area(\Gamma)\}
        \end{align*}
        are linearly independent over $\Z$. 
    \end{itemize}
\end{cond}
We remark that, for a generic Ricci-positive metric $\bg$, Condition \ref{cond:A} must hold for all $g$ and $L$. 
 
To guarantee that a metric satisfies the nondegeneracy assumption and the linear independence assumption of the previous theorem, the perturbation theorem below will be instrumental.  
\begin{prop} \label{prop:perturbMetric}
    Given $g_0 \in \mathbb{N}$ and $L \in \mathbb{R}^+$, suppose that $(M, \bg)$ is a  Riemannian $3$-manifold with $\# \cO_{g_0,\leq L}(M,\bg) < \infty$ and it admits no degenerate stable embedded minimal surface. Then for any $\varepsilon \in \mathbb{R}^+$ and $g_1 > g_0$, there exists $\bg' \leq \bg$ satisfying:
    \begin{enumerate}[label=\normalfont(\arabic*)]
        \item \label{item:sameNumber_inproof}  $\# \cO_{g_0, \leq L}(M, \bg') = \# \cO_{g_0, \leq L}(M, \bg)$.
        \item $\# \cO_{\leq g_1, \leq L}(M, \bg') < \infty $, $\# \cN_{\leq g_1, \leq L}(M, \bg') < \infty$, and the set 
        \begin{equation}\label{eqn:def_K_g'}
          \cK_{\bg'}:= \left(\cO_{\leq g_1, \leq L}(M, \bg') \cup   \cN_{\leq g_1, \leq L}(M, \bg')\right) \setminus \cO_{g_0,\leq L}(M,\bg')
        \end{equation}
        consists of strongly non-degenerate minimal surfaces.
        \item \label{item:linearIndep_inproof} For any $\Sigma\in\cO_{g_0,\leq L}(M,\bg')$,
        \[
            \left\{\area_{\bg'}(\Sigma'): \Sigma' \in \cK_{\bg'} \right\}\cup\{\area_{\bg'}(\Sigma)\}
        \]
        forms a $\mathbb{Z}$-linearly independent set.
    \end{enumerate}
\end{prop}

This proposition generalizes \cite[Proposition 6.1]{chuLi2024fiveTori} to arbitrary genus, and we will prove it in Appendix \ref{sect:proofPerturbMetric}. Here, a minimal surface is called \textit{strongly non-degenerate} if its double cover is non-degenerate, and $\bg'\leq\bg$ means that $\bg-\bg'$ is non-negative definite pointwise. Note, the conclusion of ``strong" non-degeneracy is not needed in this paper (mere non-degeneracy suffices), but we record it nonetheless for possible future usage.
    
\section{Enumerative min-max theorem}\label{sect:enuMinMax}
    In this section, we prove Theorem \ref{thm:enumerMinMax}. We use $\Z_2$-coefficients throughout this section unless specified otherwise.

\subsection{Preparation}
We first state a lemma explaining why the cap product in Theorem \ref{thm:enumerMinMax} is well-defined.

\begin{lem}\label{lem:cupProdWellDef}
    Given a simplicial  complex $X$, a subcomplex $Z\subset X$, and integers $k\geq l\geq 0$, there is a well-defined cap product 
    $$H_k(X,Z)\x H^l(X\backslash Z)\to H_{k-l}(X,Z).$$
\end{lem}
\begin{proof} By the definition of simplicial subcomplex, we can take some  open neighborhood $U(Z)$ of $Z$ such that 
the pairs $(X,Z)$ and $(X,  \overline{U(Z)} )$ are homotopy equivalent. So   $H_k(X,Z)=H_k(X,  \overline{U(Z)} ),$ which is equal to $H_k(X\backslash U(Z),\partial U(Z))$ by excision. Thus, since there is a well-defined cap product
    $$H_k(X\backslash U(Z),\partial U(Z))\x H^l(X\backslash U(Z))\to H_{k-l}(X\backslash U(Z),\partial U(Z))$$
    (see for example Hatcher \cite[p.240]{hatcher2002book}),
    there is also a well-define cap product 
    $$H_k(X,Z)\x H^l(X\backslash U(Z))\to H_{k-l}(X,Z).$$
    But since $X\backslash U(Z)$ and $X\backslash Z$ are homotopy equivalent, the lemma follows.
\end{proof}

Now, let $(M,\bg)$ be the closed orientable 3-manifold with positive Ricci curvature given in  Theorem \ref{thm:enumerMinMax}, and we begin to prove that $(M,\bg)$ has at least $p+1$ orientable minimal surfaces of genus $g$ and area  $\leq L_0 := \displaystyle\max_{x\in X}\area_\bg\Phi(x) $. Without loss of generality, we can assume that $(M,\bg)$ has only finitely many such minimal surfaces, i.e. $\# \cO_{g,\leq L_0}(M,\bg)<\infty$.

\subsection{Perturbation of metric} 
Since $(M,\bg)$ is Ricci-positive, it contains no degenerate stable minimal surfaces. Hence, we apply Proposition \ref{prop:perturbMetric} to $(M,\bg)$ with $g_0:=g,g_1:=g+1,\epsilon:=1$, and $L:=L_0$, and obtain a  metric   $\bg'\leq \bg$ satisfying the properties therein. In particular, $$L_0':=\sup_x\area_{\bg'}\Phi(x)\leq \sup_x\area_{\bg}\Phi(x)\leq L_0$$
as $\bg'\leq \bg$, and $\#\cO_{g, L_0}(M,\bg')=\#\cO_{g, L_0}(M,\bg)$. As a result, to prove   Theorem \ref{thm:repetitiveMinMax},  it suffices to show $\# \cO_{g,\leq L'_0}(M,\bg') \geq p + 1$.

\subsection{Applying the optimal family}\label{sect:applyOptimalFamily}  
Applying Theorem \ref{thm:repetitiveMinMax} to  $(M,\bg')$,  $g$, $\Phi$, and some $r>0$ to be specified, we obtain the optimal family $\Phi'$,  subcomplexes $D_0,D_1,...,D_l$, and minimal surfaces $\Gamma_1,...,\Gamma_l$ as stated. Suppose for the sake of contradiction there are only $\leq p$ minimal surfaces of genus $g$ and area $\leq L_0$. 
We relabel $\Gamma_1,...,\Gamma_l$ such that $\Gamma_1,\dots,\Gamma_m$ have genus $g$ and $\Gamma_{m+1},\dots,\Gamma_l$ have genus $<g$, where $m\leq p$ by assumption. Of course, we should relabel the subcomplexes $D_1,...,D_l$ accordingly.

Recall that we are given some homology class $w\in H_k(X,Z)$ and $p$ cohomology classes 
    $\lambda_i\in H^{k_i}(X\backslash Z),i=1,\dots,p,$
    such that the statements (1) and (2) in Theorem \ref{thm:enumerMinMax} hold: Note these two statements are  independent of the ambient  metric.

\begin{prop}\label{lem:cupProdZero}
    For each $i=1,\dots,p$ and $j=1,\dots,m$, we have  $\lambda_i|_{D_j\backslash Z} =0.$ Thus, $$(\lambda_1\smile\dots\smile\lambda_p)|_{(D_1\cup\dots\cup D_m)\backslash Z}=0.$$
\end{prop}
\begin{proof}
Suppose by contradiction $\lambda_i|_{D_j\backslash Z}\ne 0$ for some $i=1,\dots,p$ and $j=1,\dots,m$. Then there exists some $k_i$-cycle $C\subset D_i\backslash Z$ such that $\lambda_i([C])\ne 0$.  Recall, from item (3) of Theorem \ref{thm:repetitiveMinMax}, that $\Phi|_{D_i}$ can be deformed via pinch-off processes to become some map into  $\cS_{g}(M)\cap \bB^{\bF}_{r}(\Gamma_j)$, and thus so can    $\Phi|_{C}$. Hence, if we take $r$ to be sufficiently small (depending on $\Gamma_1,...,\Gamma_m$), then this contradicts with item (2) in the assumption of Theorem \ref{thm:enumerMinMax}, as desired. 
The second statement of Proposition \ref{lem:cupProdZero}   then follows  by elementary algebraic topology.
\end{proof}

We can now easily finish the proof of Theorem \ref{thm:enumerMinMax}. Denote $D:= D_0\cup D_{m+1}\cup D_{m+2}\cup\dots\cup D_l$. To obtain a contradiction, it suffices to prove:
\begin{prop}\label{prop:splitDomain}
 Consider the inclusion map $i:(D,D\cap Z)\to (X,Z)$, and the induced map $$i_*:H_{k-k_1-\dots-k_p}(D,D\cap Z)\to H_{k-k_1-\dots-k_p}(X,Z).$$ There exists some relative subcomplex $(Y,\partial Y)\subset ( D,D\cap Z)$
        such that 
        $$i_*([  Y])= w\frown(\lambda_1\smile\dots\smile\lambda_p)\in  H_{k-k_1-\dots-k_p}(X, Z).$$
\end{prop}
Indeed, suppose we know this proposition is true. Hence,  $\Phi|_Y$ cannot be deformed via pinch-off processes to become a map into   $\cS_{\leq g-1}(M)$ due to Condition \ref{item:enuMinMaxY} of Theorem \ref{thm:enumerMinMax}. However, since $Y\subset D$, we know using the defining properties of $D_i$ (namely Theorem \ref{thm:repetitiveMinMax} \eqref{item:optimal}) that each member of $\Phi'|_{ Y}$  has  genus at most $\max\{\fg(\Gamma_{m+1}),\dots,\fg(\Gamma_l)\}$. And since   $\Gamma_{m+1},\dots,\Gamma_l$ all have genus $<g$, we know that  $\Phi'|_Y$  maps into $\cS_{\leq g-1}(M)$. Contradiction arises.

So it suffices to prove the above proposition.

\subsection{Proof of Proposition \ref{prop:splitDomain}} In the lemma below,  we fix  an arbitrary commutative ring for the coefficients of homology and cohomology.  Readers may refer to \cite[p.240]{hatcher2002book} for the following relative version of cup product.
\begin{lem}\label{lem:existsThetaAlgTop} 
    Let $X'$ be a (not necessarily finite) simplicial complex, and $A,B\subset X'$ be open subsets. Let $p\geq q\geq 0$ be natural numbers and $p > 0$. 
    Let $C$ be a union of $p$-subcomplexes in $X'$ with $\partial C\subset A\cup B$ (so that
     $[C]\in H_{p}(X', A\cup B)$), and $\omega\in H^q(X',A)$. 
    Suppose that: 
    \begin{itemize}
        \item $C$ is a union of two open subsets $W_1,W_2\subset C$.
        \item $W_2\cap A=\emptyset.$
        \item The pullback of $\omega$ under the inclusion $(W_1,W_1\cap A)\hookrightarrow (X',A)$ is zero.
    \end{itemize}
    Then there exists some  $\theta\in H_{p-q}(W_2,W_2\cap B)$  such that the pushforward of $\theta$ under the inclusion $(W_2,W_2\cap B)\hookrightarrow (X',B)$ is equal to
    $[C]\frown\omega\in H_{p-q}(X',B).$ 
\end{lem}
We postpone the proof of Lemma \ref{lem:existsThetaAlgTop} to Appendix \ref{sect:capProd}: It just uses standard algebraic topology.
In this paper we will only use the case $A=\emptyset$, but we choose to state this general version for possible future usage.

In addition, it is easy to check that the above lemma is also valid if $A,B\subset X'$ and $W_1,W_2,\subset C$ are subcomplexes instead. Hence, we apply this lemma to our situation as follows.

\begin{itemize}
    \item  Take some  open neighborhood $U(Z)\subset X=\dmn(\Phi)$ of $Z$ such that the inclusion map $(X,Z)\hookrightarrow(X,  \overline{U(Z)} )$ and is a  homotopy equivalence between the pairs, and $X':= X\backslash U(Z)$ is homotopy equivalent to $X\backslash Z$. We shall use these homotopy equivalence to identify homology or cohomology groups.
\item  Put $A:=\emptyset$ and $B:=X'\cap \overline{U(Z)}$. Note, by excision, we can identify the homology groups of $(X', B)$ and $(X,Z)$.  Without loss of generality we can assume $B$ is a subcomplex. 
\item Let $r:=k=\deg(w)$ and   $ s:=k_1+...+k_p=\deg(\lambda_1\smile...\smile \lambda_p).$
    \item Let $C\subset X'$ be some  $k$-subcomplex, with $\partial C\subset B$, that  represents  $w$, viewed as an element of $H_k(X',B)$.  
    \item  Let  $W_1:=((D_1\cup...\cup D_m)\backslash U(Z)) \cap C$ and  $W_2:=( D\cup \overline{U(Z)})\cap C$. Note $W_1$ and $W_2$  cover $C$, and can be assumed to be subcomplexes.
    \item Let $\omega:=\lambda_1\smile...\smile \lambda_p$, viewed as an element of  $H_{k_1+...+k_p}(X')$. Note, since the pullback of $\lambda_1\smile...\smile \lambda_p$ under $(D_1\cup...\cup D_m)\backslash Z\hookrightarrow X\backslash Z$ is zero by Proposition \ref{lem:cupProdZero}, we can ensure the pullback of $\omega$ under $W_1\hookrightarrow X'$ is zero  by choosing $U(Z)$ to be a sufficiently thin neighborhood of $Z.$
\end{itemize}
Then, by the above lemma, there exists some $\theta\in H_{k-k_1-...-k_p}(W_2,W_2\cap B)$ such that the pushforward of $\theta$ under the inclusion $(W_2,W_2\cap B)\hookrightarrow (X',B)$ is equal to $$[C]\frown(\lambda_1\smile...\smile\lambda_p)\in H_{k-k_1-...-k_p}(X',B).$$ Then it is easy to see, there exists some $ \theta_1\in H_{k-k_1-...-k_p}(D,D\cap Z)$ such that the pushforward of $ \theta_1$ under the inclusion $(D,D \cap Z)\hookrightarrow (X,Z)$ is equal to $$w\frown(\lambda_1\smile...\smile\lambda_p)\in H_{k-k_1-...-k_p}(X,Z).$$ Letting $Y$ be a representative of $\theta_1$, we finish the proof of Proposition \ref{prop:splitDomain}, and also Theorem \ref{thm:enumerMinMax}.  

\section{Existence of four   genus 2 minimal surfaces}\label{sect:4Genus2}

In this section, we prove   Theorem \ref{thm:main} in the following steps:  
\begin{enumerate}
    \item Construct a 7-parameter family $\Psi:\RP^5\x\D\to\cS_{\leq 2}(S^3)$. 
    \item Rotate $\Psi$ by $SO(4)$, and then the symmetry of $\Psi$ would allow us to define a family
    $$\Xi:(\RP^5\x \D \x SO(4))/D_{24}\to\cS_{\leq 2}(S^3).$$
    \item Verify the applicability of Theorem \ref{thm:enumerMinMax} to $\Xi$, and produce four genus 2 minimal surfaces.
\end{enumerate} 

Let $S^3 := \{x \in \mathbb{R}^4 \mid \sum^4_{i = 1} x^2_i = 1\}$, let $C$ denote the great circle $ \{x \in S^3 \mid x_3 = x_4 = 0\}$. For each $(x_1,x_2,x_3,x_4) \in S^3\backslash C$, we can write 
\begin{equation}\label{eq:coor}
(x_1,x_2,x_3,x_4) = (x_1,x_2,\sqrt{1-x_1^2-x_2^2}\ \cos\alpha,\sqrt{1-x_1^2-x_2^2}\ \sin\alpha)\,,    
\end{equation}
with $(x_1,x_2)$ in the open unit disc $\inte(\bD)$ and $\alpha\in \bS^1:=\R/2\pi\Z$. This yields a parametrization $(x_1, x_2, \alpha)$ of $S^3$ by $\bD\times \bS^1$, although the parametrization is not injective along the locus $x^2_1 + x^2_2 = 1$, where distinct $\alpha$ are mapped to the same point.

\subsection{A $7$-parameter family $\Psi$}\label{sect:2g+3}

Let $\D \subset \R^2$ be the closed unit $2$-disc centered at the origin. We define a $7$-parameter Simon--Smith family $\Psi:\RP^5\x\D\to\cS_{\leq 2}(S^3)$ as follows. 

For each $a=[a_0:a_1:\cdots:a_5]\in\RP^5$ and $(r,\theta)\in \D$ (in polar coordinates), $\Psi$ sends $(a, (r, \theta))$ to the zero set in $S^3$
\begin{equation}\label{eq:PsiDef}
    \begin{aligned}
        \Big\{(x_1, x_2, \alpha) \in S^3 \mid &a_0+a_1x_1+a_2x_2+a_3x_3+a_4x_4\\
        &+a_5\left[x_1x_2+ \rho(a,x)\left(r\cos(\theta+2\alpha)+(1-r)\cos3\alpha\right)\right]=0\Big\}\,.
    \end{aligned}
\end{equation}
Here, $\rho:\RP^5\x S^3\to [0,1]$ is a cut-off function determined below. We fix a non-increasing cut-off function $\zeta\in C^\infty(\R)$ such that $\zeta|_{\R_{\leq 1/2}}=1$, $\zeta|_{\R_{\geq 1}}=0$, and then define  $\rho$  as follows:  
\begin{enumerate} [label={\normalfont(\Roman*)}]
    \item For every $[a_0:a_1:...:a_4:1]\in\RP^5$ with $a_1^2+a_2^2<1$, let 
    \begin{align}
     \begin{split} 
        \rho(a, x) := &\ \zeta \left(64\cdot\frac{(x_1+a_2)^2+(x_2 + a_1)^2}{\left(1-a_1^2-a_2^2\right)^2}\right)\ \zeta\left(\frac{a_3^2+a_4^2+(a_0-a_1a_2)^2}{\delta\left(a_1^2 +  a_2^2\right)}\right)\delta\left({a_1^2 +  a_2^2}\right)
     \end{split}  
    \end{align}
    where $\delta\in C^\infty([0, \infty))$ is $0$ on $[1, +\infty)$, and positive but very small on $[0,1)$. Equivalently, for every $a\in\RP^5$ with $a_5\ne 0$  and $a_1^2+a_2^2<a_5^2$, 
        \begin{align}
     \begin{split} 
        \rho(a, x) := &\ \zeta \left(64a_5^2\cdot\frac{(a_5 x_1+a_2)^2+(a_5 x_2 + a_1)^2}{\left(a^2_5-a_1^2-a_2^2\right)^2}\right) \zeta\left(\frac{(a_3^2+a_4^2)a_5^2+(a_0 a_5-a_1a_2)^2}{a^4_5\ \delta\left((a_1^2 +  a_2^2)/a^2_5\right)}\right)\delta\left(\frac{a_1^2 +  a_2^2}{a^2_5}\right)
     \end{split} \label{Equ_Def_rho(a,x)}
    \end{align}
    \item Otherwise, set $\rho(a, x)=0$.
\end{enumerate}
 
\begin{rmk}
As mentioned in \S \ref{sect:sketchFourGenus2}, the family  $\Psi$ should be viewed as a desingularization of the $\RP^5$-family $\Phi_5$ given by 
$$\Phi_5(a):=\{a_0+a_1x_1+a_2x_2+a_3x_3+a_4x_4+a_5 x_1x_2=0\}\subset S^3,$$
such that only finitely many desingularities are left.
Namely, for any $a\in\RP^5$ in the set
$$A_{\sing}:=\{[a_1a_2:a_1:a_2:0:0:1]\in\RP^5:   a_1^2+a_2^2<1\},$$
$\Phi_5(a)=\{(x_1+a_2)(x_2+a_1)=0\}$ contains the singular circle $C(a_1,a_2):=\{x_1=-a_2,x_2=-a_1\}$. The feature of the above $\rho$ is that:
\begin{itemize}
    \item When $a\in\RP^5$ is near $A_{\sing}$, the function $\rho(a,\cdot):S^3\to\R$ is a bump function supported near the circle $C(a_1,a_2)$.
    \item When $a$ is away from $A_{\sing}$, the function $\rho(a,\cdot)\equiv 0$.
\end{itemize}
\end{rmk}

\begin{thm}\label{prop:Psi1}
    With the functions $\delta,\zeta$ appropriately chosen, \eqref{eq:PsiDef} gives a Simon--Smith family
    \[
        \Psi:\RP^5\x \D\to\cS_{\leq 2}(S^3)
    \] with the following properties.  
    \begin{enumerate}
        \item\label{Item:Family_Genus0} For every  $(a,z)\in \RP^5\x\D$, if $\rho(a,\cdot):S^3\to\R$ is  the constant zero function, then $\Psi(a,z)$ has genus $0$.
         \item\label{Item:Family_Genus1}  $\Psi|_{\RP^5\x \partial\D}$ is of genus $\leq 1$.
        
    \end{enumerate}
\end{thm}
We shall prove this theorem in \S \ref{sect:proofThmPsi}.

\subsection{Symmetry group of genus $2$ Lawson surface}\label{subsec:symm_gp_lawson_surf}
While the Lawson surface $\xi_{2,1}$ does not actually appear in the family $\Psi$, we remark that the element $\Psi([0:...:0:1],0)$ in fact has the same symmetry as $\xi_{2,1}$, and it will be helpful to understand its symmetry.

By Kapouleas--Wiygul~\cite[Lemma~3.10, Lemma~3.11]{KW22Lawson_symmetry} (see also Lawson~\cite{Law70}), the genus $2$ Lawson surface $\xi_{2,1}$ associated to a $4$-segment partition of the great circle $C$ and a $6$-segment partition of the complementary orthogonal great circle $C^\perp = \{x_1=x_2=0\}\cap S^3$ has an orientation-preserving symmetry group $G$ of order $24$. By~\cite[Lemma~3.11 (ii)]{KW22Lawson_symmetry}, $G = \mathcal{G}^\mathbf{C} = D_{24} \subset SO(4)$, generated by $2$ group elements $g_1$ and $g_2$ satisfying the relations
\begin{equation}\label{eq:D24relation}
    (g_1)^{12} = (g_2)^2 = (g_2g_1)^2=\id\,.    
\end{equation}
Hence, the space of genus $2$ Lawson surfaces is 
\[
    \{R(\xi_{2,1}): R\in SO(4)\} \cong SO(4)/\langle g_1, g_2\rangle \,.
\]
The fact that $G$ is dihedral is not explicitly stated in \cite{KW22Lawson_symmetry}. For completeness and convenience, we therefore rewrite $G$ in an equivalent form below.

For our purpose, in the construction of the Lawson surface $\xi_{2,1}$, we shall use on $C^\perp$ (resp. $C$) the partition whose endpoints contain $(0,0,\cos(\pi/6),\sin(\pi/6))$ (resp. $(0,1,0,0)$).

Since $\mathbb{R}^4$ can be identified with the quaternion space $\mathbb{H}$ via $(x_1, x_2, x_3, x_4) \mapsto x_1 + i x_2 + j x_3 + k x_3$, we will henceforth express the group $G$ in quaternion coordinates. Let $\mathbb S^3\subset \mathbb{H}$ denote the unit quaternion group, and each element $(q_1, q_2) \in \mathbb S^3\times \mathbb S^3$ acts on $S^3 \subset \mathbb{R}^4$ isometrically by 
\begin{equation}\label{eq:doubleCover}
    (q_1, q_2): \mathbb{R}^4 \to \mathbb{R}^4, \quad p \mapsto q_1pq_2^{-1}\,. \footnote{The expression $q_1pq_2^{-1}$ (as opposed to $q_1^{-1}pq_2$) is compatible with the left $SO(4)$-action on $\R^4$.}
\end{equation}
The induced group homomorphism $\mathbb{S}^3 \times \mathbb{S}^3 \to SO(4)$ is a double cover.
Let $\hat G$ denote the subgroup of $\mathbb S^3\times \mathbb S^3$ generated by
\[
  \hat g_1 := \left( e^{i5\pi/12}, e^{-i\pi/12} \right)\,, \quad \hat g_2 := (j, -j) \,.
\] 
It can be checked that $\hat G$ is the dicyclic group $Q_{48}$ of order $48$. Note that the kernel of the above double cover map is $\{\pm(1,1)\}\subset\hat G$.

Under the group homomorphism induced by~\eqref{eq:doubleCover}, $\hat g_1$ and $\hat g_2$ are mapped to some elements $g_1$ and $g_2$ in $SO(4)$, and we let $G < SO(4)$ be the subgroup generated by $g_1$ and $g_2$, which must be of order 24 by the previous paragraph. Then $g_1$ corresponds to a rotation in $\R^4$ by $\pi/2$ in the $x_1x_2$-plane and by $\pi/3$ in the $x_3x_4$-plane, as
\begin{equation}\label{eqn:g_1_act1}
    \begin{aligned}
        & g_1(x_1 + ix_2 + jx_3 + kx_4) \\
        &\qquad = e^{i5\pi/12}\big((x_1+ix_2) + (x_3+ix_4)j \big) e^{i\pi/12} \\
        &\qquad = e^{i\pi/2}(x_1+ix_2) + e^{i\pi/3}(x_3 + ix_4)j\,, \\
        &\qquad = (-x_2 + i x_1) + \left( \left(x_3 \cos\frac{\pi}{3} - x_4 \sin\frac{\pi}{3}\right) + \left(x_3 \sin\frac{\pi}{3} + x_4 \cos\frac{\pi}{3}\right)i\right)j\,;
    \end{aligned}
\end{equation}
$g_2$ corresponds to a reflection about the 2-plane $\{x_1=x_3=0\}$, as
\begin{equation}\label{eqn:g_2_act1}
    \begin{aligned}
      g_2(x_1 + ix_2 + jx_3 + kx_4) & = j\big((x_1+ix_2) + (x_3+ix_4)j \big)(-j)^{-1} \\
      & = (-x_1+ix_2) + (-x_3+ix_4)j \,.
    \end{aligned}
\end{equation}
In particular, on $S^3\subset \R^4$ and under  the $(x_1, x_2, \alpha)$ coordinate, the above actions are:
\begin{equation}\label{eqn:g_act2}
    g_1(x_1, x_2, \alpha)=(-x_2, x_1, \alpha + \pi/3), \quad g_2(x_1, x_2, \alpha) =(-x_1, x_2, \pi - \alpha)\,,
\end{equation}
which preserve the orientation of $S^3$. 

Using \eqref{eqn:g_act2}, we can check that $g_1$ and $g_2$ satisfy the relations~\eqref{eq:D24relation}, so $G$ is the dihedral group $D_{24}$ of order $24$. Since $g_1$ and $g_2$ preserve $\xi_{2, 1}$, and the orientation preserving symmetry group of $\xi_{2,1}$ is of order $24$~\cite[Lemma~3.11 (ii)]{KW22Lawson_symmetry}, $G$ is indeed this group.

\subsection{A 13-parameter family $\Xi$} 
  First, we define a $13$-parameter Simon--Smith family $\bar \Xi:\RP^5\x\D\x SO(4)\to\cS_{\leq 2}(S^3)$ by
  \begin{equation}\label{eq:barXi}
    \bar \Xi(a,z,R):=R\cdot\Psi(a,z),
  \end{equation}
  where $\Psi$ is the 7-parameter family defined in the previous subsection, and $R \cdot \Psi(a,z)$ denotes the image of  $\Psi(a,z)\subset S^3$ under the rotation  $R \in SO(4)$ acting from the left.

  Then, we define a {\it left} $G$-action $\sigma$ on $\RP^5\x\D$ by\footnote{We show the form of $\sigma(g^{-1}_i)$ instead of $\sigma(g_i)$, for convenience in later calculation.}:
  \begin{equation}\label{Equ_D_24-action_Y}
    \begin{aligned}
      &\sigma(g_1^{-1}):& ([a_0:\dots:a_5],z) &\mapsto ([a_0:a_2:-a_1:a_3\cos\frac\pi 3+a_4\sin\frac \pi 3,-a_3\sin\frac\pi 3+a_4\cos\frac\pi 3,-a_5],e^{-i\pi/3}z),\\
      &\sigma(g_2^{-1}):& ([a_0:\dots:a_5],z) &\mapsto  ([a_0:-a_1:a_2:-a_3:a_4:-a_5],-\bar z).
    \end{aligned}
  \end{equation}
  Since $G < SO(4)$, this induces a {\it right} $G$-action on $\RP^5\x\D\times SO(4)$ by:
  \begin{equation}\label{eqn:g_act3}
    (a, z, R) \cdot g = (\sigma(g^{-1})(a, z), R\cdot g)\,.
  \end{equation}

  \begin{rmk}
      The right quotient $\cX:=(\RP^5\x \D\x SO(4))/G$ is the same as the total space $SO(4)\x_{\sigma} \cY$ of the {\it associated bundle to $SO(4)\to SO(4)/G$ with fiber $\cY$}; See Notation \ref{nota:fiberBundle}.
  \end{rmk}

Note, a direct calculation shows that $\rho$ has the following symmetry:
\begin{equation}\label{eq:g1ActRP5}
 \begin{split}
  & \rho([a_0:a_1:\dots:a_5],g_1x) \\
  & \qquad =\rho([a_0:a_2:-a_1:a_3\cos\frac\pi 3+a_4\sin\frac \pi 3,-a_3\sin\frac\pi 3+a_4\cos\frac\pi 3:-a_5],x)\,, 
 \end{split}
\end{equation}
and 
\begin{equation}\label{eq:g2ActRP5}
\rho([a_0:a_1:\dots:a_5],g_2x)=\rho([a_0:-a_1:a_2:-a_3:a_4:-a_5],x)\,. 
\end{equation}

  \begin{prop}\label{prop:tildeXiInvariant} 
    For each $g\in G$ and each $(a,z,R)\in \RP^5\x\D\x SO(4)$,   
    \[
      \bar \Xi ((a,z,R)\cdot g)=\bar \Xi(a,z,R)\,.
    \]
  \end{prop}
  
  In other words, the family $\bar \Xi$ also has certain $G$-symmetry. It follows straightforwardly from the definition, and we postpone the proof to \S \ref{sect:tildeXiInvariantProof}. 
  Note that the facts that $G$ acts on $SO(4)$ from the right while $SO(4)$ acts on $\R^4$ from the left would be important for the proof.

Hence, letting
\[
    \cX:=(\RP^5\x\D\x SO(4))/G\,,
\]
 $\bar \Xi$ induces a well-defined $13$-parameter Simon–Smith family 
\[
    \Xi:\cX\to \cS_{\leq 2}, (a, z, R)\cdot G \mapsto \bar \Xi(a, z, R)\,. 
\]

\subsection{Topology of $\cX$}\label{sect:TopoX}
Let $B:=SO(4)/G$, which is a closed smooth manifold. Since by definition, the natural projection onto the third factor $ \RP^5\times \D\times SO(4)\to SO(4)$ is $G$-equivarient, it naturally descends to fiber bundles
$$\RP^5 \times \D \hookrightarrow \cX \xrightarrow{\pi} B\quad\textrm{ and }\quad   \RP^5 \times \partial\D \hookrightarrow \partial\cX \xrightarrow{\pi} B.$$

By \S~\ref{subsec:symm_gp_lawson_surf}, we can also write $B = (\mathbb{S}^3\times \mathbb{S}^3)/\hat G$. Let $\hat\Pi_1: \mathbb{S}^3\times \mathbb{S}^3 \to \mathbb{S}^3 \subset \mathbb{H}$ be the projection map onto the first factor. Hence, $\hat\Pi_1(\hat g_1) = e^{i5\pi/12}$ and $\hat\Pi_1(\hat g_2) = j$, so $\hat\Pi_1(\hat G)$ is also isomorphic to $Q_{48}$.

\begin{lem}\label{lem:BisProduct}
    The projection map $\hat\Pi_1$ induces a well-defined map \[
       \Pi_1: B\to \mathbb{S}^3/\hat \Pi_1(\hat G), \quad (q_1, q_2)\cdot \hat G\mapsto q_1 \cdot\hat \Pi_1(\hat G)\,.
    \] 
    Moreover, this yields a \emph{trivial} $\mathbb{S}^3$-bundle $\mathbb S^3 \hookrightarrow B \xrightarrow{\Pi_1} \mathbb{S}^3/\hat\Pi_1(\hat G)$.
\end{lem}

We postpone this topological proof to \S \ref{sect:spaceLawsonProof}.

By the previous lemma, $B\cong \mathbb S^3/Q_{48}\x\mathbb S^3$. Now, note the universal cover $\tilde \cX$ of $\cX$ is $S^5\x\D\x \mathbb S^3\x \mathbb S^3$. Let $\tilde G:=\Z_2\oplus \hat G\cong \Z_2\oplus Q_{48}$, which is generated by
\[
    \tilde g_0:=(1,0)\,,\quad \tilde g_1:=(0,\hat g_1)\,,\quad \tilde g_2:=(0,\hat g_2)\,.
\]
We define a {\it right} $\tilde G$-action on $\tilde \cX$ as follows:
 For every $(a, z, {\bf q})\in S^5\times \D\times (\mathbb{S}^3\times \mathbb{S}^3)$,
\begin{equation}\label{eqn:action_tilde_g}
\begin{aligned} 
  (a,z,{\bf q})\cdot\tilde g_0 & := ((-a_0,-a_1,-a_2,-a_3,-a_4,-a_5),z, {\bf q}),\\
     (a,z, {\bf q})\cdot\tilde g_1 & := ( (a_0, a_2, -a_1,\, a_3 \cos \frac{\pi}{3} + a_4 \sin \frac{\pi}{3},\, -a_3 \sin \frac{\pi}{3} + a_4 \cos \frac{\pi}{3}, -a_5 ), e^{-i\frac{\pi}{3}}z, {\bf q} \hat g_1),\\
   (a,z, {\bf q})\cdot\tilde g_2 & :=
  (\left( a_0, -a_1, a_2, -a_3, a_4, -a_5 \right), -\bar z,
  {\bf q} \hat g_2),
\end{aligned}
\end{equation}
where for any ${\bf q} = (q_1,q_2)$, the product $(q_1,q_2)\hat g_i$ refers to a componentwise quaternionic multiplication.
It is straightforward to check that this $\tilde G$-action on $\tilde \cX$ descends to the $G$-action on $\RP^5\x \D\x SO(4)$. Hence, $\tilde G$ acts on $\tilde X$ freely, and $\cX = \tilde \cX / \tilde G$. This implies that $\pi_1(\cX)= \tilde G = \Z_2\oplus Q_{48}$ and $H_1(\cX;\Z_2)=\Z_2\oplus \Z_2\oplus \Z_2$.

In summary, we have:
\[\begin{tikzcd}[cramped]
	{\tilde \cX} && {\textrm{dmn}(\bar \Xi)} & \cX \\
	& {\mathbb S^3\times\mathbb S^3} && {B=(\mathbb S^3\times\mathbb S^3)/Q_{48}} \\
	& {\mathbb S^3} && {\mathbb S^3/Q_{48}}
	\arrow[from=1-1, to=1-3]
\arrow["{\tilde G=\mathbb Z_2\times Q_{48}}", bend left=20, from=1-1, to=1-4]
	\arrow["{G=D_{24}}", from=1-3, to=1-4]
	\arrow["{\pi}", from=1-4, to=2-4]
	\arrow["{\hat G=Q_{48}}", from=2-2, to=2-4]
	\arrow["{\hat \Pi}_1"', from=2-2, to=3-2]
	\arrow["{\Pi_1}", from=2-4, to=3-4]
	\arrow["{\hat\Pi_1( \hat G)=Q_{48}}", from=3-2, to=3-4]
\end{tikzcd}\]

We now fix a base point $\tilde x_0\in \tilde \cX$:
\begin{equation}\label{eq:x0Tilde}
\tilde x_0:=((0,0,0,0,0,1),0,(1,0,0,0),(1,0,0,0))\,.    
\end{equation}
Then the generators $\tilde g_0,\tilde g_1$, and $\tilde g_2$ can be viewed as elements in $\pi_1(\cX)$, and thus they induce three elements $c_0,c_1$, and $c_2$ in $H_1(\cX;\Z_2)$ respectively. {\it From now on, we will always use $\Z_2$-coefficients, unless specified otherwise.}
We define three cohomology classes $\lambda,\alpha\in H^1(\cX)$ and $\beta \in H^3(\cX)$:
\begin{itemize}
    \item Let $\lambda:= c_0^*+c_1^*+c_2^*\in H^1(\cX)$, where $\{c_0^*,c_1^*,c_2^*\} \subset H^1(\cX)\cong\Hom(H_1(\cX);\Z_2)$ is the dual basis of $\{c_0,c_1,c_2\} \subset H_1(\cX)$.
    \item Let $\alpha:=c_1^*+c_2^*\in H^1(\cX)$.

    \item Fix some $p \in \mathbb{S}^3$, and let $B_0:=\mathbb S^3/Q_{48}\x \{p\}\subset \mathbb S^3/Q_{48}\x \mathbb S^3 \cong B$. Let $\beta$ be the pullback $\pi^*(\beta_B)$, where $\beta_B\in H^3(B)$ is the Poincar\'e dual of $[B_0]\in H_3(B)$.
\end{itemize}

Denote the $2$-dimensional disc
\begin{equation}\label{eq:tildeD0}
    \tilde D_0:=(0,0,0,0,0,1)\x\D\x(1,0,0,0)\x (1,0,0,0)\subset \tilde \cX,
\end{equation}
and let 
$D_0\subset \cX$ be its image under the quotient map $\tilde \cX\to \cX$. Below is an important topological feature of $\cX.$

\begin{thm}\label{thm:capProduct}
    In $H_{2}(\cX,\partial \cX)$, $[\cX]\frown(\lambda^5\smile\alpha^3\smile\beta)=[D_0].$
\end{thm}

Here, $\smile$ is the cup product in $H^*(\cX)$, and $\frown$ is the relative cap product for $(\cX,\partial \cX)$. The proof is postponed to \S \ref{sect:topologyX}.

Now, let us record some crucial features of the family $\Xi.$

\begin{thm}\label{thm:XiProper} By appropriately choosing the functions $\delta,\zeta$ in the definition \eqref{Equ_Def_rho(a,x)} for the cut-off function $\rho$, the Simon--Smith family $\Xi:\cX\to \cS_{\leq 2}(S^3)$ satisfies the following.
    \begin{enumerate}
        \item\label{Item:FamilyXi_Genus0} For any $x \in \left(\{\rho=0\}\x\D\x SO(4)\right)/G\subset \cX$, $\Xi(x)$ has  genus $0$.
        \item\label{Item:FamilyXi_Genus1} For any $x \in \partial \cX$, $\Xi(x)$  has genus $\leq 1$.
        \item\label{Item:FamilyXi_1-sweepout} For any $1$-cycle $C\subset \cX$   such that $\alpha([C])\ne 0$ and $\Xi|_C$ maps into $\cS_2(S^3)$,   $\Xi|_C$ is a $1$-sweepout in the Almgren-Pitts sense. 
        \item\label{Item:FamilyXi_5-sweepout} For any $5$-cycle $C\subset \cX$ such that $\lambda^5([C])\ne 0$, $\Xi|_{C}$ is a $5$-sweepout.
        \item\label{Item:FamilyXi_2chain}   For any closed $1$-subcomplex $C$ homologous to $ \partial D_0$ in $\partial\cX$, if $\Xi|_C$ maps into $\cS_1(S^3)$ then   
        $\Xi|_C$ is ``homologically non-trivial in $\cS_1(S^3)$" in the following sense.  There does not exist a Simon--Smith family $\Phi:W\to\cS_1(S^3)$, where $W$ is a {\it pure} simplicial $2$-complex, such that  (i) $\partial W=C$ and (ii) $\Phi|_{\partial W}=\Psi|_{C}$.
    \end{enumerate}
\end{thm}

We postpone the proof to \S \ref{sect:XiProper}. It will be checked in the proof that  in item \eqref{Item:FamilyXi_Genus0} the set defined by $\rho=0$ is indeed $G$-equivariant and thus we can take the quotient by $G$.  

\subsection{Proof of Theorem \ref{thm:main} via Theorem \ref{thm:enumerMinMax}} 

Let $(S^3,\bg)$  be the 3-sphere of positive Ricci curvature given in Theorem \ref{thm:main}. Let $\cX_{0} \subset \cX$ be the preimage of $B_0:=\mathbb S^3/Q_{48}\x \{p\}\subset B$ under the bundle projection $\pi: \cX\to B$, and $\alpha_0\in H^1(\cX_0)$ be the restriction $\alpha|_{\cX_0}$. Clearly, $\partial \cX_0 = \cX_0 \cap \partial \cX$.

Theorem \ref{thm:main} would follow immediately from applying Theorem \ref{thm:enumerMinMax}, to the 10-dimensional family  $\Xi_0:=\Xi|_{\cX_0}:(\cX_0,\partial\cX_0)\to\cS_{\leq 2}(S^3)$, with $p = 3$ and $\lambda_1 = \lambda_2 = \lambda_3 = \alpha_0$. Thus, it now suffices to verify the two  assumptions in Theorem \ref{thm:enumerMinMax}. 

We first verify   assumption \eqref{item:enuMinMaxC} of Theorem \ref{thm:enumerMinMax}. Let $S\subset M$ be any smooth embedded genus $g$ surface. Pick any 1-cycle $C\subset \cX_0$ such that $\alpha_0([C])\ne 0$. We  let $r>0$  be small enough such that every $S'\in\cS(M)$ that is $r$-close to $S$ in the $\bF$-distance must have genus at least $g$: See \cite[Lemma 2.6]{chuLiWang2025optimalFamily}. Thus, to verify Theorem \ref{thm:enumerMinMax} \eqref{item:enuMinMaxC}, we can assume without loss of generality that $\Xi|_C$ maps into $\cS_2(S^3)$. Then by Theorem \ref{thm:XiProper} \eqref{Item:FamilyXi_1-sweepout}, $\Xi|_C$ must be a 1-sweepout. In particular, it cannot be deformed via pinch-off processes to become $r$-close to $S$, as desired.

To verify   Theorem \ref{thm:enumerMinMax} \eqref{item:enuMinMaxY}, we let $(Y,\partial Y)\subset (\cX_0,\partial \cX_0)$ be an arbitrary relative 7-subcomplex such that
\begin{equation}\label{eq:YDef}
[Y]=[\cX]\frown(\alpha_0)^3\in H_7(\cX_0,\partial \cX_0).    
\end{equation}
By \eqref{eq:YDef} and Theorem \ref{thm:capProduct}, we have 
\[
    [Y] \frown \lambda^5=[\cX] \frown \lambda^5 \frown \alpha^3 \frown \beta = [D_0] \in H_2(\cX,\partial\cX)\,.
\] 
Suppose by contradiction that there exists a deformation via pinch-off processes, 
\begin{equation}
    H:[0,1]\x Y\to\cS_{\leq 2}(S^3)\,,
\end{equation}
where $H(0,\cdot)=\Xi|_{Y}$, and $\Xi':=H(1,\cdot)$ maps into $\cS_{\leq 1}(S^3)$. For $k=0,1$, let $Z_k:=\{x\in Y:\fg(\Xi'(y))=k\}$.   We can derive a contradiction by finding a relative $2$-subcomplex $(D,\partial D)\subset (Z_1,Z_1\cap \partial \cX)$ such that $[\partial D]=[\partial D_0] \in H_1(\partial \cX)$. Indeed, with such a $D$, we can concatenate $H|_{[0,1]\x\partial D}$ and $\Xi'|_D$ along $\Xi'|_{\partial D}$ to obtain a family $\Phi$. In this construction, the boundary of the domain of $\Phi$ is $\{0\}\x\partial D$, on which $\Phi$ agrees with $\Xi|_{\partial D}$, contradicting Theorem \ref{thm:XiProper} \eqref{Item:FamilyXi_2chain}.

Using \cite[Proposition 2.23]{chuLiWang2025optimalFamily}, we take a {\it subcomplex} $\tilde Z_0\subset Y$ slightly larger than  $Z_0$ such that $\Xi|_{\tilde Z_0}$ can also be deformed via pinch-off processes to become a map into $\cS_0(S^3)$. Now, recall that there does not exist a Simon--Smith family in $\cS_0(S^3)$ that is 5-sweepout in the Almgren-Pitts sense: This was proven in \cite[\S 3.5]{chuLi2024fiveTori}. Hence, $\Xi|_{\tilde Z_0}$ cannot be a 5-sweepout, and therefore, by Theorem \ref{thm:XiProper} \eqref{Item:FamilyXi_5-sweepout}, $\lambda^5|_{\tilde Z_0}=0
$. Now, we apply Lemma \ref{lem:existsThetaAlgTop}  with
\begin{itemize}
    \item Let $X:=\cX$, $A:=\emptyset$. Slightly thickened $\partial\cX$ to make it open, and let $B$ be that.
    \item $p:=7$, $q:=5$, $C:=Y$. $\omega:=\lambda^5$,
    \item $W_1:=\textrm{int}(\tilde Z_{0}) $,
    $W_2:=Z_1$.
\end{itemize}
Note that the assumptions Lemma \ref{lem:existsThetaAlgTop} are   satisfied, and so  we immediately obtain some relative 2-subcomplex $(D,\partial D)\subset(Z_1,Z_1\cap \partial \cX)$ for which, when viewed in $H_2(\cX,\partial\cX)$, $[D]=[D_0]$. In particular, $[\partial D]=[\partial D_0]$ in $H_1(\partial \cX)$, as desired. This finishes the proof of Theorem \ref{thm:main}.

\section{Properties of  the family $\Psi$}\label{sect:proofThmPsi}
  In this section, we specify the cut-off functions $\rho, \zeta, \delta$ in the definition of family $\Psi$ in \eqref{eq:PsiDef} and \eqref{Equ_Def_rho(a,x)} and prove Theorem \ref{prop:Psi1}. This section is organized as follows. We first study in \S \ref{Subsec_5-family} zero sets of the form \[
    \{x\in S^3: a_0+ a_1x_1+ a_2x_2+ a_3x_3+ a_4x_4+ a_5x_1x_2 = 0\} \,.
  \]
  This corresponds to an $\RP^5$-family of cycles which (if nonempty) are either punctate surfaces with genus $0$, or a union of two transversely intersecting round spheres.  \S \ref{Subsec_AbstractLem} contains some abstract lemmas discussing zeros of functions on $\bbS^1$, which help to understand in Lemma \ref{Lem_Constru Psi_GenusCal of (F=0)} the genus and other geometric behavior of the zero sets of a special class of functions of the form \eqref{eq:PsiDef}. 
  Using them, we specify the choice of the cut-off functions $\rho, \zeta, \delta$ in \S \ref{Subsec_Constru Psi} and discuss the basic properties of the family $\Psi$ in Lemma \ref{Lem_SmallDelta}. Theorem \ref{prop:Psi1} will proved at the end of \S \ref{Subsec_Constru Psi} based on Lemma \ref{Lem_SmallDelta}. Finally, in \S \ref{sect:tildeXiInvariantProof}  we prove Proposition \ref{prop:tildeXiInvariant}.
  
  Let us recall the coordinate system we  use on $S^3$. First, we view $S^3$ as the set of points $x\in\R^4$ of unit distance to the origin.  Let $C\subset S^3$ be the great circle $\{x\in S^3:x_3=x_4=0\}$. We parametrize $S^3$ by 
  \begin{align}
    \bD\times \bbS^1 \to S^3, \quad (x, \alpha)\mapsto (x_1, x_2, \varsigma(x)\cos\alpha, \varsigma(x)\sin\alpha)\,, \label{Equ_Constru Psi_Param S^3}
  \end{align}
  where $\varsigma(x):= \sqrt{1-x_1^2-x_2^2}$, $\bD$ is the closed unit disk in $\R^2$, and $\bbS^1:= \R/2\pi\Z$. For later reference in this section, we also let $\bD_\kappa(b)$ be the closed disc of radius $\kappa$ centered at $b$. Recall that \eqref{eq:coor} is a diffeomorphism from $\Int(\bD)\times \bbS^1$ to $S^3\setminus C$, and it is a smooth $\bbS^1$ fiber bundle when restricted to $\partial\bD\times \bbS^1$, mapping $\{z\}\times \bbS^1$ to a point on $C$. Note also that here we use the symbol $\bf D$ in order to distinguish it from the disc $\D$ in $\dmn(\Psi)=\RP^5\x\D$.

 \subsection{The canonical $\RP^5$ family} \label{Subsec_5-family}
   We first consider the $\RP^5$ family: for every $a=[a_0:a_1:a_2:a_3:a_4:a_5]$, define
   \begin{align}
     \tilde\Phi_5(a) & := \{x\in \R^4: a_0+ a_1x_1+ a_2x_2+ a_3x_3+ a_4x_4+ a_5x_1x_2 = 0\} \subset \R^4\,; \\
     \Phi_5(a) & := S^3\cap \tilde\Phi_5(a)\subset S^3 \,. \label{Equ_Def_Phi_5(a)}
   \end{align}
   Note that $\tilde\Phi_5(a)$ is either the union of two orthogonal hyperplanes (in which case $a = [a_1a_2 :a_1 :a_2 :0 :0 :1]$), or an empty set (in which case $a = [1:0:0:0:0:0]$) or a smooth hypersurface in $\R^4$. And $\tilde\Phi_5(a)$ varies smoothly in $a$ in their regular parts. We say that  $\tilde\Phi_5(a)$ is \textit{transverse} to $S^3$ at $x\in S^3\cap \tilde\Phi_5(a)$ if $x$ is a regular point of $\tilde\Phi_5(a)$ and the tangent spaces $T_xS^3$ and $ T_x\tilde\Phi_5(a)$ span $\R^4$.

   It is easy to check that $\Phi_5$ gives a 5-sweepout in the Almgren-Pitts setting, but it is not a Simon--Smith family. However, we still have genus bound when $\Phi_5(a)$ does not contain a singular curve:
   \begin{lem}\label{Lem_5-family_SingParam}
     Let \[
       A_{\sing}:= \{[a_1a_2 :a_1 :a_2 :0 :0 :1] \in\RP^5: a_1^2 + a_2^2 < 1\}\,,
     \]
     Then, 
     \begin{enumerate} [label={\normalfont(\roman*)}]
      \item\label{Item_5-family_2plane} for every $a\in A_{\sing}$, $\Phi_5(a)$ is the union of two round (but not necessarily equatorial) spheres that intersect transversely along the circle $\{x\in S^3: x_1=-a_2, x_2=-a_1\}$;
      \item\label{Item_5-family_reg_a.e.} for a dense subset of $a\in \RP^5\setminus A_{\sing}$, $\tilde\Phi_5(a)$ is transverse to $S^3$, and hence $\Phi_5(a)$ is smooth;
      \item\label{Item_5-family_finiteSing} for every $a\in \RP^5\setminus A_{\sing}$, $\tilde\Phi_5(a)$ is transverse to $S^3$ at all but at most $9$ points; and $\mathfrak{g}(\Phi_5(a)) = 0$.
     \end{enumerate}
   \end{lem}
   \begin{proof}
     \ref{Item_5-family_2plane} follows by noticing that when $a\in A_{\sing}$, $\tilde\Phi_5(a) = \{(x_1+a_2)(x_2+a_1)=0\}$ is the union of two orthogonal hyperplanes $\{x_1=-a_2\}$ and $\{x_2=-a_1\}$ whose intersection transversely intersects $S^3$ along the circle described in \ref{Item_5-family_2plane}.
     
     To prove \ref{Item_5-family_reg_a.e.} and \ref{Item_5-family_finiteSing}, consider an arbitrary $a\in \RP^5\setminus A_{\sing}$. If $a_5 = 0$, then $\tilde\Phi_5(a)$ is either empty or an affine hyperplane in $\R^4$, hence \ref{Item_5-family_finiteSing} holds for such $a$; if not, then we write $a=[a_0:\dots:a_4:1]$. If $a_0=a_1a_2$ and $a_3=a_4=0$, then by $a\notin A_{\sing}$, we have $a_1^2+a_2^2\geq 1$, hence $\Phi_5(a)$ is the union of at most two 2-spheres (possibly degenerated to points), either disjoint or tangent to each other, which also satisfies the assertion in \ref{Item_5-family_finiteSing}. For the remaining case, \[
       \tilde\Phi_5(a)=\{x\in \R^4: (x_1+a_2)(x_2+a_1) + (a_3x_3+a_4x_4) + a_0-a_1a_2 = 0\}\,.
     \] 
     where $(a_3, a_4, a_0-a_1a_2)\neq (0,0,0)$. Hence there exists $\cR\in SO(4)$, $b\in \R^4$ and $\lambda>0$ such that 
     \begin{align*}
        \eta(\tilde\Phi_5(a)) = \begin{cases}
          \{x\in \R^4: x_1x_2 = 1\}, & \ \text{ if }(a_3, a_4)= (0,0)\,, \\
          \{x\in \R^4: x_1x_2 = x_3\}, & \ \text{ if }(a_3, a_4)\neq (0,0)\,;
        \end{cases}
     \end{align*}
     where $\eta(x) = \eta_{\cR, b, \lambda}(x):= \lambda\cdot\cR x + b$ is a diffeomorphism of $\R^4$. Also notice that $\eta(S^3) = \partial B^4_\lambda(b)$, hence $\eta(\Phi_5(a)) = (P\times \R)\cap \partial B^4_\lambda(b)$, where $P$ is the smooth surface described in  Lemma \ref{Lem_5-family_topo of surf} below. 
     
     Write $b=(\breve b, b_4)\in \R^3\times \R$, $\varrho_{\breve b}(\breve x):= |\breve x - \breve b|^2$ is a function on $\R^3$. Notice that $P\times \R$ intersects non-transversely with $\partial B^4_\lambda(b)$ at $x=(\breve x, x_4)$ if and only if $x_4=b_4$ and $\breve b - \breve x \perp T_{\breve x} P$, i.e. $\varrho_{\breve b}(\breve x)=\lambda^2$ and $\breve x$ is a critical point of $\varrho_{\breve b}|_P$. Also notice that if $\lambda^2$ is a regular value of $\varrho_{\breve b}|_P$, then the smooth surface $(P\times \R)\cap \partial B_\lambda^4(b)$ is homeomorphic to the following. Consider the smooth surface with boundary $P\cap \overline{B^3_\lambda(\breve b)}$, take two copies of it, and glue them together along their boundaries: Let us call the resulting surface (without boundary) the {\it double of $P\cap \overline{B^3_\lambda(\breve b)}$} . 
     
     Therefore, \ref{Item_5-family_reg_a.e.} follows by Sard Theorem, as the regular value of $\varrho_{\breve b}|_P$ is dense; \ref{Item_5-family_finiteSing} follows from Lemma \ref{Lem_5-family_topo of surf} below, where the genus control is immediate when the intersection of $\tilde\Phi_5(a)$ and $S^3$  is transverse, since the double of finitely many disks is finitely many spheres; when the intersection is not transverse, the genus control follows from the approximation of regular $\Phi_5(a_j)$ guaranteed by \ref{Item_5-family_reg_a.e.} and the lower-semi-continuity of genus $\mathfrak{g}$ under $C^\infty_{loc}$ convergence away from finitely many points. 
   \end{proof}

   \begin{lem} \label{Lem_5-family_topo of surf}
     Let $P$ be one of the following two smooth surfaces in $\R^3$ \[
       \{x\in \R^3: x_1x_2 = 1\}\,, \qquad \{x\in \R^3: x_1x_2 = x_3\}\,.
     \]
     Then for every $b\in \R^3$, let $\varrho_b(x):= |x-b|^2$, we have,
     \begin{enumerate} [label={\normalfont(\alph*)}]
      \item\label{Item_5-family_Topo_9Crit} $\varrho_b|_P$ has at most $9$ critical points;
      \item\label{Item_5-family_Topo_Disk} if $r^2$ is a regular value of $\varrho_b|_P$, then $\overline{B_r^3(b)}\cap P$ is diffeomorphic to a finite collection (possibly empty) of disjoint disks. 
     \end{enumerate}
   \end{lem} 
   \begin{proof}
     We only focus on the case when $P=\{x\in \R^3:x_1x_2=x_3\}$, the other surface can be treated in a similar (actually even simpler) way.

     For every $x\in P$, $x$ is a critical point of $\varrho_b|_P$ if and only if $2(x_1-b_1, x_2-b_2, x_3-b_3) =\nabla\varrho_b(x)\perp T_x P$, which is equivalent to 
     \begin{align}
      \begin{cases}
        x_1-b_1 = -x_2(x_3-b_3)\,, \\
        x_2-b_2 = -x_1(x_3-b_3)\,.
      \end{cases} \label{Equ_5-family_quadr equ}
     \end{align}

     When $(x_3-b_3)^2\neq 1$, \eqref{Equ_5-family_quadr equ} has solution 
     \begin{align*}
        x_1 = \frac{b_1 - b_2(x_3-b_3)}{1-(x_3-b_3)^2}\,, & &
        x_2 = \frac{b_2 - b_1(x_3-b_3)}{1-(x_3-b_3)^2}
     \end{align*}
     Plugging back to the defining equation $x_1x_2=x_3$ for $P$ gives a $5$-th order polynomial equation in $x_3$, so there are  at most $5$ solutions for $(x_1, x_2, x_3)$. 

     When $x_3-b_3 = \pm 1$, by \eqref{Equ_5-family_quadr equ}, $
       x_1 \pm x_2 = b_1 = \pm b_2\,.$
     Combining with $x_1x_2=x_3 = b_3 \pm 1$ gives at most $4$ solutions for $(x_1, x_2, x_3)$. In summary, there are at most $5+4=9$ solutions $x$ to $x\in P$ and $\nabla\varrho_b\perp T_xP$, which proves \ref{Item_5-family_Topo_9Crit}.

     To prove \ref{Item_5-family_Topo_Disk}, we can first take a smooth perturbation of $\varrho_b|_P$ to get a Morse function $\bar\varrho\in C^\infty(P)$ such that $\spt(\bar\varrho-\varrho_b|_P)$ is compact and that $\bar\varrho^{-1}(\R_{\leq r^2})$ is diffeomorphic to $\varrho_b^{-1}(\R_{\leq r^2})\cap P = \overline{B_r^3(b)}\cap P$. Also, since the second fundamental form $A_P$ of $P$ is everywhere indefinite, \[
       \nabla^2(\varrho_b|_P) = \nabla^2\varrho_b + (\nabla^\perp\varrho_b)\cdot A_P = 2\,\id_{T_*P} + (\nabla^\perp\varrho_b)\cdot A_P\, 
     \]
     has at least $1$ strictly positive eigenvalue everywhere on $P$. Thus we can further require the smooth perturbation $\bar\varrho$ to have no local maximum, i.e. every critical point of $\bar\varrho$ has Morse index $0$ or $1$. Therefore, the first betti number function $\lambda\mapsto \dim H_1(\bar\varrho^{-1}(\R_{\leq \lambda}))$
     is monotone non-decreasing in regular value $\lambda$ of $\bar\varrho$. 
     When $\lambda\gg 1$, $\bar\varrho^{-1}(\R_{\leq \lambda}) = \varrho_b^{-1}(\R_{\leq \lambda})\cap P = \overline{B^3_{\sqrt\lambda}(b)}\cap P$ is a single disk, so we conclude that \[
       \dim H_1(\overline{B_r^3(b)}\cap P) = \dim H_1(\bar\varrho^{-1}(\R_{\leq r^2})) = 0\,.
     \]
     Also, since $P$ is a graph over the $x_1x_2$-plane and hence diffeomorphic to $\R^2$, as a compact subdomain with first betti number $0$, $\overline{B_r^3(b)}\cap P$ is diffeomorphic to a finite disjoint union of disks.
   \end{proof}

 \subsection{Some abstract lemmas} \label{Subsec_AbstractLem}
    We list some useful abstract lemmas in this subsection for later application. Readers may skip the proofs in this subsection in a first reading.
   
   For $f\in C^\infty(\bbS^1; \R)$, let $\alpha\in f^{-1}(0)$. Define the order of $f$ at $\alpha$ to be
   \begin{align*}
      \ord_f(\alpha):= \sup\{l\in \Z_{\geq 1}: f(\alpha)=f^{(1)}(\alpha) = \dots = f^{(l-1)}(\alpha)=0\}\,.
   \end{align*}
   Also for $E\subset \bbS^1$, let 
   \begin{align*}
      \ord(f; E) := \sum_{\alpha\in f^{-1}(0)\cap E} \ord_f(\alpha)\,; 
   \end{align*}
   and we may omit $E$ if $E=\bbS^1$. 
   \begin{lem} \label{Lem_Constru Psi_local ord usc}
     For $N \in \mathbb{N}^+$, $f\in C^\infty(\bbS^1)$ and $\alpha\in f^{-1}(0)$ with $\ord_f(\alpha)\leq N$, suppose that a sequence $f_j\in C^\infty(\bbS^1)$ $C^N$-converges to $f$ as $j\to \infty$, and $U_1\supset U_2\supset...$ are closed intervals containing $\alpha$ that shrink to $\alpha$ as $j\to \infty$. Then for all $j\gg 1$, $\ord(f_j; U_j) \leq \ord_f(\alpha)$.
   \end{lem}
   \begin{proof}
     Notice that by the intermediate value theorem, any closed interval $U\subset \bbS^1$ ended by two zeros of $f$ contains a zero of $f'$ in its interior. Hence, inductively, we have, \[
       \ord(f; U) \leq \ord(f^{(1)}; U) + 1 \leq \dots \leq \ord(f^{(N)}; U) + N\,.
     \]
     
     Suppose now for contradiction $\ord(f_j; U_j) > \ord_f(\alpha)=:N'$ for infinitely many $j$. Then after passing to a subsequence, for every $j$, and every $k \leq N'$,
     \[
       \ord(f_j^{(k)}, U_j) \geq \ord(f_j, U_j) - k>0\,,
     \] 
     i.e. there exists $\alpha_{j,k}\in U_j$ such that $f_j^{(k)}(\alpha_{j, k})=0$. For $j\to \infty$, we get $\alpha_{j, k}\to \alpha$ and $f^{(k)}(\alpha)=0$ for all $0 \leq k \leq N'$. This forces $\ord_f(\alpha)\geq N'+1$, which gives a contradiction.
   \end{proof}

   Associated to each $f\in C^\infty(\bS^1; \R)$ with
   $0<\ord(f)<+\infty$, we define:
   \begin{align}
      \cZ(f)   & := \sum_{\alpha\in f^{-1}(0)} \ord_f(\alpha)\cdot \alpha\,;  \\
      \cZ^-(f) & :=  \sum_{\alpha  \in \mu(\{f<0\})}\alpha + \sum_{\alpha\in f^{-1}(0)} \frac12 \big(\ord_f(\alpha)-\Neg_f(\alpha)\big)\cdot \alpha \,. \label{Equ_Def_cZ^-}
   \end{align}
   Here, when given a disjoint union $J$ of intervals on $\bS^1$, $\mu(J)$ denotes the set consisting of the midpoint of each connected component of $J$; 
   and  $\Neg_f(\alpha)\in \{0,1,2\}$ denotes   the number of connected components of $\{f<0\}$ whose closure contains $\alpha$. 
   Note that since $\alpha,\mu(I)\in \bbS^1=\R/2\pi\Z$ and $\ord_f(\alpha)-\Neg_f(\alpha)$ is always even, we have $\cZ(f), \cZ^-(f)$ also take value in $\bS^1=\R/2\pi\Z$. 
   When $\ord(f)=0$, we set as a convention that $\cZ(f)=\cZ^-(f):= 0 \text{ mod}\,2\pi$. 

   \begin{lem} \label{Lem_cZ=2cZ^-}
     For every $f\in C^\infty(\bS^1; \R)$ with $\ord(f)<+\infty$, we have $\cZ(f) = 2\cZ^-(f)$ in $\bS^1$. Moreover, if $f_j\in C^\infty(\bS^1, \R)$ such that $f_j\to f_\infty$ in $C^n$ and $\ord(f_j)=\ord(f_\infty)=n<+\infty$ for all $j\geq 1$, then $\cZ^-(f_j)\to \cZ^-(f_\infty)$.
   \end{lem}
   \begin{proof}
     The first assertion follows directly from the definition and that for an open interval $I\subset \bS^1$ with end points $\alpha^-, \alpha^+$, we have $2\mu(I)=\alpha^-+\alpha^+$ in $\bS^1$.

     To prove the second assertion, notice that we can view each $\alpha\in f^{-1}(0)$ with $\ord_f(\alpha)-\Neg_f(\alpha)>0$ as $(\ord_f(\alpha)-\Neg_f(\alpha))/2$ degenerate intervals in $\{f\leq 0\}$, so that \eqref{Equ_Def_cZ^-} is the sum of midpoints of all (possibly degenerate) $n/2$ intervals in $\{f\leq 0\}$. By Lemma \ref{Lem_Constru Psi_local ord usc}, if the total order is preserved under taking the  limit $f_j\to f_\infty$, then the  zeros of $f_\infty$ must be the limits of the zeros of $f_j$. In particular, the corresponding (possibly degenerate) intervals in $\{f_j\leq 0\}$ vary continuously  under convergence. $\cZ^-(f_j)\to \cZ^-(f_\infty)$ thus follows from this.
   \end{proof}

   For every integer $k>0$, a \textit{trigonometric polynomial} of degree $k$ is a function on $\bbS^1$ of the form
   \begin{align}
      T(\alpha) := b_0 + b_1\cos(\alpha+\theta_1) + \dots + b_k\cos(k\alpha +\theta_k)\,, \label{Equ_Constru Psi_TrigonPolyn}
   \end{align}
   where $b_j\in \R$, $b_k\neq 0$, and $\theta_j\in \R/2\pi\Z$. (We use the convention that the  zero function has degree $-1$.) We let \[
      \|T\|:= \left(b_0^2+b_1^2+\dots + b_k^2\right)^{1/2}\,,
   \]
   and let $\cT_k$ be the space of all trigonometric polynomial of degree $\leq k$, which is a vector space of  dimension $2k+1$.
   
   \begin{lem} \label{Lem_Constru Psi_Ord(Trig polyn)}
     For every nonzero $T\in \cT_k$, $\ord(T)\leq 2k$. Furthermore, if $T$ has the form \eqref{Equ_Constru Psi_TrigonPolyn} and $\ord(T) = 2k$, then $b_k\neq 0$, $\cZ(T)=-2\theta_k$ and
     \begin{align*}
        \cZ^-(T) = \begin{cases}
            -\theta_k + k\pi,     & \text{ if }b_k>0\,; \\
            -\theta_k+ (k+1)\pi, & \text{ if }b_k<0\,.
        \end{cases} 
     \end{align*}
   \end{lem}
   \begin{proof}
     By induction, without loss of generality, we can assume that $T$ has the form \eqref{Equ_Constru Psi_TrigonPolyn} where $b_k\neq 0$. By possibly replacing $(b_k, \theta_k)$ by $(-b_k, \theta_k-\pi)$ and multiplying $T$ by a positive constant (which does not change $\cZ^-(T)$), we can even assume that $b_k=1$.
     
     Notice that $\alpha\in T^{-1}(0)$ if and only if $z:= e^{i\alpha}$ is a root, with multiplicity $\ord_T(\alpha)$,  of the following polynomial:
     \begin{align*}
        q_T(z):= z^k\left(b_0 + \frac{b_1}{2}(e^{i\theta_1}z + e^{-i\theta_1}z^{-1})+ \dots + \frac{1}{2}(e^{i\theta_k}z^k + e^{-i\theta_k}z^{-k}) \right). 
     \end{align*}
     Hence, $\ord(T)$ is no more than the total number of roots of $q_T$, counting multiplicity, which is $2k$ by the fundamental theorem of algebra. 
     
     Moreover, if $\ord(T) = 2k$, then all roots of $q_T$ are sitting on the unit circle, and $e^{i\cZ(T)}$ is the product of all roots of $q_T$, which is equal to $q_T(0)$ divided by the leading coefficient of $q_T$. Hence, $e^{i\cZ(T)} = e^{-2i\theta_k}$, and then $\cZ(T)=-2\theta_k$, which yields by Lemma \ref{Lem_cZ=2cZ^-} that $\cZ^-(T)\in \{-\theta_k, -\theta_k+\pi\}$.

     To determine $\cZ^-(T)$ precisely, notice that  any trigonometric polynomial of form 
     \begin{align}
       T(\alpha) = b_0 + b_1\cos(\alpha+\theta_1) + \dots + b_{k-1}\cos((k-1)\alpha+\theta_{k-1}) \ + \cos(k\alpha+\theta_k) \label{Equ_Trigon b_k=1}
     \end{align}
     with order $2k$ is uniquely determined by   its zeros and $\theta_k$. More precisely, if $\alpha_1, \dots, \alpha_{2k}\in \R/2\pi\Z$ are the zeros of the above $T$ (counting the orders) such that $\sum\alpha_j = -2\theta_k$, then \[
       q_T(z) = \frac{e^{i\theta_k}}2(z-e^{i\alpha_1})(z-e^{i\alpha_2})\cdots (z-e^{i\alpha_{2k}}) 
     \] 
     and therefore \[
       T(\alpha) = e^{-ik\alpha}\, q_T(e^{i\alpha}) = \frac12 e^{i(-k\alpha+\theta_k)}(e^{i\alpha}-e^{i\alpha_1})(e^{i\alpha}-e^{i\alpha_2})\cdots (e^{i\alpha}-e^{i\alpha_{2k}}) \,.
     \] 
     Also, any $\alpha_j$ and $\theta_k$ as above give a trigonometric polynomial of the form \eqref{Equ_Trigon b_k=1} with order $2k$. Therefore, the space $\cT_k^{top}$ of trigonometric polynomial of form \eqref{Equ_Trigon b_k=1} with order $2k$ is path connected (and clearly contains $\cos(k\alpha+\theta_k)$). Moreover by Lemma \ref{Lem_cZ=2cZ^-}, $T\mapsto \cZ^-(T)$ is continuous on $\cT_k^{top}$, hence combined with the fact that $\cZ^-(T)\in \{-\theta_k, -\theta_k+\pi\}$, we see that \[
       \cZ^-(T) = \cZ^-(\cos(k\alpha+\theta_k)) = -\theta_k+k\pi\,.
     \]
   \end{proof}

   \begin{cor} \label{Cor_alm trig polyn has zero bd}
     For every integer $k>0$ and $\varepsilon_1 > 0$, there exists $\delta_1(k, \varepsilon_1)>0$ with the following property. For any $T\in \cT_k$ in the form \eqref{Equ_Constru Psi_TrigonPolyn}, and any  $f\in C^\infty(\bbS^1)$   such that \[
       \|f-T\|_{C^{2k}}\leq \delta_1\|T\|\,,
     \]
     we have $\ord(f)\leq 2k$; and if in fact $\ord(f)= 2k$ and $b_k>0$, then $
        \dist_{\bbS^1}(\cZ^-(f), -\theta_k+k\pi)\leq \varepsilon_1$.
   \end{cor}
   \begin{proof}
     We may argue by contradiction and apply Lemma \ref{Lem_cZ=2cZ^-} and \ref{Lem_Constru Psi_Ord(Trig polyn)}. 
   \end{proof}

   \begin{lem} \label{Lem_Constru Psi_GenusCal of (F=0)}
     For every $b=(b_1, b_2)\in \bD$ and $\kappa>0$ such that $|b|+2\kappa<1$, there exists $\delta_2(\kappa)\in (0, \kappa/4)$ with the following property.

     Suppose $F\in C^\infty(\bD\times \bbS^1; \R)$ and $\bx\in C^\infty(\bbS^1; \bD_{\kappa/4}(b))$ satisfy the following:
     \begin{enumerate} [label={\normalfont(\alph*)}]
       \item\label{Item_GenusCalc_asump1} $\Sigma:= F^{-1}(0)$ agrees with $\Phi_5(a)$ (see section \ref{Subsec_5-family}) in $(\bD\setminus \bD_\kappa(b))\times \bbS^1$ for some $a=[a_0:a_1:a_2:a_3:a_4:1]\in \RP^5$ with \[
         |a_0-a_1a_2|+|a_1+b_2|+|a_2+b_1|+|a_3|+|a_4| < \delta_2\,;
       \]
       \item\label{Item_GenusCalc_asump2} $f(\alpha):= F(\bx(\alpha), \alpha)$ has $\ord(f)<+\infty$;
       \item\label{Item_GenusCalc_asump3} in $\bD_\kappa(b)\times \bbS^1$, $\nabla_xF(x, \alpha) = 0$ if and only if $x=\bx(\alpha)$; and \[
         \left|\nabla^2_x F - \begin{bmatrix}
           0 & 1 \\ 1 & 0
         \end{bmatrix} \right| < \delta_2\,.
       \]
     \end{enumerate}
     Then 
     \begin{enumerate}
      \item\label{Item_GenusCalc_concl_reg} Under the parametrization \eqref{Equ_Constru Psi_Param S^3}, $\Sigma$ is a smooth surface in $S^3$ with isolated singularities in $\{(\bx(\alpha), \alpha): \alpha\in f^{-1}(0), \ord_f(\alpha)\geq 2\}$.
      \item\label{Item_GenusCalc_concl_innerLoop} For every $\alpha\in \bbS^1$,
      \begin{align*}
        F(\cdot, \alpha) \begin{cases}
          >0\ \text{ on } \{\bx(\alpha)+(t,t)\in \bD: 0\neq t\in \R\}\,, &\ \text{ if }f(\alpha)\geq 0\,; \\
          <0\ \text{ on } \{\bx(\alpha)+(t,-t)\in \bD: 0\neq t\in \R\}\,, &\ \text{ if }f(\alpha)\leq 0\,;
        \end{cases}
      \end{align*}
      \item\label{Item_GenusCalc_concl_genus} the genus of $\Sigma$ satisfies \[
        \mathfrak{g}(\Sigma) \leq \max\left\{0,\ \frac12\#\{\alpha\in f^{-1}(0): \ord_f(\alpha) \text{ is odd}\}-1\right\}\,.
      \]  
      and equality holds if every zero of $f$ has order $1$.
     \end{enumerate}
   \end{lem}

\begin{figure}[!ht]
    \centering
\includegraphics[width=1.5in]{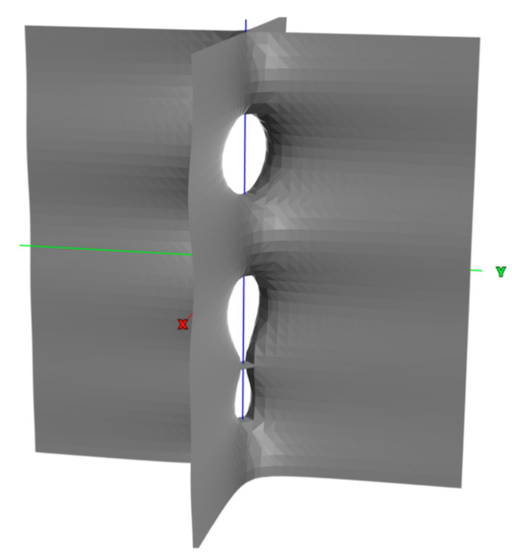}
\caption{An example of the surface $\Sigma$: It shows a portion of a surface $\Sigma$ near the circle $\{x_1=-b_2,x_2=-b_1\}\subset S^3$, which is represented by the $z$-axis in the figure, parametrized by $\alpha$. Those $\alpha$ satisfying $F (\bx(\alpha),\alpha)=0$ correspond to the points on the surface with horizontal tangent planes. 
}
\label{fig:xy}
\end{figure}
   \begin{proof}
     To prove \eqref{Item_GenusCalc_concl_reg}, by \ref{Item_GenusCalc_asump1} and assuming $\delta_2\ll 1$, $\Sigma = \Phi_5(a)$ is smooth outside $\bD_\kappa(b)\times \bbS^1$. Hence it suffices to show that for every $(x, \alpha)\in \Sigma\cap \bD_\kappa(b)\times\bbS^1$, 
     \begin{align}
        \nabla F(x,\alpha) = 0 \quad \Leftrightarrow \quad
        x=\bx(\alpha)\ \text{ and }\ \ord_f(\alpha)\geq 2 \,. \label{Equ_GenusCalc_crit pt}
     \end{align}
     To see this, note that by \ref{Item_GenusCalc_asump3}, $\nabla_xF(x,\alpha)=0$ if and only if $x = \bx(\alpha)$; while by definition of $f$, \[
       f'(\alpha) = \nabla_xF(\bx(\alpha), \alpha)\cdot \bx'(\alpha) + \partial_\alpha F(\bx(\alpha), \alpha) = \partial_\alpha F(\bx(\alpha), \alpha) \,.
     \] 
     Combining these proves \eqref{Equ_GenusCalc_crit pt}.

     To prove \eqref{Item_GenusCalc_concl_innerLoop}, for $0\neq v=(v_1, v_2)\in \R^2$ such that $\bx(\alpha)+v\in \bD_\kappa(b)$, let $h(t):= F(\bx(\alpha)+tv, \alpha)$. Note that by \ref{Item_GenusCalc_asump2} and \ref{Item_GenusCalc_asump3}, 
     \begin{align*}
       h(0) = f(\alpha)\,, \quad h'(0) = 0\,, \quad
       |h''(t) - 2v_1v_2| \leq \delta_2|v|^2\,.
     \end{align*}
     Hence $\pm h''(t)>0$ if $v=(t, \pm t)$. Combined this with the fact \[
       F(\bx(\alpha) + v) = f(\alpha) + \int_0^1(1-t)h''(t)\ dt
     \]
     proves \eqref{Item_GenusCalc_concl_innerLoop} in $\bD_\kappa(b)$. Moreover,  by \ref{Item_GenusCalc_asump1} (assuming $\delta_2\ll 1$) and  a compactness argument, in $(\bD\setminus \bD_\kappa(b))\times \bbS^1$ we know $\Sigma = \Phi_5(a)$ does not intersect $\{(\bx(\alpha)+(t,\pm t), \alpha)\in \bD\times \bbS^1: 0\neq t\in \R, \alpha\in \bbS^1\}$, so the sign of $F(\cdot, \alpha)$ on the corresponding line segments are not changing outside $\bD_\kappa(b)$.

     To prove \eqref{Item_GenusCalc_concl_genus}, we first further assume that every zero of $f$ has order $1$ (hence $\Sigma$ is smooth everywhere by \eqref{Item_GenusCalc_concl_reg}). Consider the Killing vector field $\partial_\alpha$ on $\bD\times \bbS^1$ (which is also a Killing vector field on the round $S^3$ under parametrization \eqref{Equ_Constru Psi_Param S^3}). For $(x_\circ, \alpha_\circ)\in \bD_\kappa(b)\times \bbS^1\cap \Sigma$, the orthogonal projection $\partial_\alpha^\Sigma$ of $\partial_\alpha$ onto $T\Sigma$ is zero at $(x_\circ, \alpha_\circ)$ if and only if $\nabla F(x_\circ, \alpha_\circ)$ and $\partial_\alpha$ are parallel, or equivalently, $\nabla_x F(x_\circ, \alpha_\circ) = 0$, which is equivalent to $x_\circ=\bx(\alpha_\circ)$ (and $\alpha_\circ\in f^{-1}(0)$) by \ref{Item_GenusCalc_asump3}. Moreover, near such $(x_\circ, \alpha_\circ)$, viewing $\alpha$ as a single valued function on $\bD\times \bbS^1$, we have $\partial_\alpha^\Sigma = \nabla(\alpha|_\Sigma)$. And if we set $\partial_\alpha|_{(x_\circ, \alpha_\circ)} = c_\circ\nabla F(x_\circ, \alpha_\circ)$ for some nonzero constant $c_\circ\in \R$, then at $(x_\circ, \alpha_\circ)$,
     \begin{align*}
        \nabla^2(\alpha|_\Sigma) = \nabla^2\alpha|_{T\Sigma\otimes T\Sigma} + \vec{A}_\Sigma\cdot \partial_\alpha 
        & = c_\circ \vec A_\Sigma\cdot \nabla F(x_\circ, \alpha_\circ) \\
        & = c_\circ(\nabla^2(F|_\Sigma) - \nabla^2F(x_\circ, \alpha_\circ)|_{T\Sigma\otimes T\Sigma}) \\
        & = -c_\circ \nabla_x^2 F(x_\circ, \alpha_\circ)
     \end{align*}
     is indefinite and nondegenerate by \ref{Item_GenusCalc_asump3} if $\delta_2$ is chosen small enough. Hence, the Poincar\'e-Hopf index of the vector field $\partial_\alpha^\Sigma$ at its zero $(x_\circ, \alpha_\circ)$ is $-1$.

     On the other hand, set $\hat a:= [a_1a_2:a_1:a_2:0:0:1]\in A_{\sing}$. By \ref{Item_GenusCalc_asump1} with $\delta_2$ chosen even smaller, $\Sigma$ is isotopic to some smooth surface $\Sigma'$ which agrees with $\Sigma$ in $\bD_\kappa(b)\times \bbS^1$, agrees with $\Phi_5(\hat a)$ outside $\bD_{2\kappa}(b)\times \bbS^1$, and such that $\partial_\alpha^{\Sigma'}$ is nonzero on $(\bD_{2\kappa}(b)\setminus \bD_{\kappa/2}(b))\times \bbS^1$. In particular, each zero of $\partial_\alpha^{\Sigma'}$ is either
     \begin{itemize}
     \item one of $\{(\bx(\alpha_\circ), \alpha_\circ):\alpha_\circ\in f^{-1}(0)\}$, each of which has Poincar\'e-Hopf index $-1$; or
     \item  one of zeros of $\partial_\alpha^{\Phi_5(\hat a)}$, whose index sum is equal to $2\chi(S^2) = 4$ by Poincar\'e-Hopf theorem.
     \end{itemize}
     Therefore,  applying Poincar\'e-Hopf theorem to $\partial_\alpha^{\Sigma'}$ we find \[
       \chi(\Sigma) = \chi(\Sigma') = 4 + (-1)\cdot\# f^{-1}(0) \,.
     \]
     Then \eqref{Item_GenusCalc_concl_genus}  follows from the fact that $\Sigma'$ has at most two connected components, and is connected when $f^{-1}(0)\neq \emptyset$. 

     Finally, when $f$ has zeros of higher order, let $Z_{ev}$ and $Z_{od}$ be the set of zeros of $f$ with even order and odd order respectively. We can construct another smooth function $h\in C^\infty(\bbS^1)$ such that
     \begin{itemize}
        \item $h^{-1}(0) = Z_{od}$ and each zero of $h$ has order $1$;
        \item $fh\geq 0$ everywhere on $\bbS^1$.
     \end{itemize}
     Then, let $\xi\in C^\infty_c(\bD_{\kappa}(b), [0,1])$ be a cut-off such that $\xi|_{\bD_{\kappa/2}(b)} = 1$, we can consider \[
       F_\epsilon (x, \alpha) := F(x, \alpha) + \epsilon \xi(x)h(\alpha) \,.
     \] 
     It is easy to check that when $0<\epsilon\ll 1$,  $F_\epsilon$ also satisfies \ref{Item_GenusCalc_asump1}-\ref{Item_GenusCalc_asump3} with $f$ replaced by $f_\epsilon:= f + \epsilon h$. By the construction above, $f_\epsilon$ only has zeros of order $1$, and $f_\epsilon^{-1}(0) = Z_{od}$. Therefore, as $F_\epsilon^{-1}(0)$ locally smoothly converges to $\Sigma$ in the regular part, we have \[
       \mathfrak{g}(\Sigma) \leq \liminf_{\epsilon \searrow 0} \mathfrak{g}(F_\epsilon^{-1}(0)) = \max\left\{0,\ \frac12\#Z_{od}-1\right\}\,.
     \]   \end{proof}

 \subsection{Construction of $\Psi$} \label{Subsec_Constru Psi}
   Recall that $\D$ denotes the closed 2-dimensional unit disc in the second factor of  $\dmn(\Psi)=\RP^5\x \D$. We will often use polar coordinate $z=re^{i\theta}$ for $\D$, and we parametrize $S^3$ by \eqref{eq:coor}. Under this parametrization, the family $\Psi: \RP^5\times \D\to \cS(S^3)$ takes the form  
   \begin{align}
     \Psi(a, z):= \begin{cases}
        \Phi_5(a)\,, & \text{ if } a_5= 0\,; \\
        \{(x, \alpha)\in \bD\times \bbS^1 : F_{a,z}(x, \alpha) = 0\}\,, & \text{ if } a_5\neq 0
     \end{cases} \label{Equ_Def_Psi(a,z)}
   \end{align}
   where for $(a,z)\in (\RP^5\times \D)\cap \{a_5\neq 0\}$ (recall $\varsigma(x):=\sqrt{1-x_1^2-x_2^2}$), 
   \begin{align}
    \begin{split}
     F_{a,z}(x, \alpha) := a_5^{-1} \Big( a_0 & + a_1x_1 + a_2x_2 + \varsigma(x)(a_3\cos\alpha + a_4\sin\alpha) \\
     & + a_5\Big[x_1x_2+\rho(a, x)\big(r\cos(\theta+2\alpha)+(1-r)\cos(3\alpha)\big)\Big] \Big)\,. 
    \end{split} \label{Equ_Def_F_a,z}
   \end{align}
   Recall that the cut-off function $\rho:\RP^5\x S^3\to[0,1]$ is chosen as in \eqref{Equ_Def_rho(a,x)}: 
   \begin{enumerate} [label={\normalfont(\Roman*)}]
     \item When $a=[a_0: a_1: a_2: a_3: a_4: a_5]\in \RP^5$ with $a_1^2+a_2^2<a_5^2$, let 
     \begin{align*}
        \rho(a, x) := &\ \zeta \left(64a_5^2\cdot\frac{(a_5 x_1+a_2)^2+(a_5 x_2 + a_1)^2}{\left(a^2_5-a_1^2-a_2^2\right)^2}\right) \zeta\left(\frac{(a_3^2+a_4^2)a_5^2+(a_0 a_5-a_1a_2)^2}{a^4_5\ \delta\left((a_1^2 +  a_2^2)/a^2_5\right)}\right)\delta\left(\frac{a_1^2 +  a_2^2}{a^2_5}\right)
     \end{align*}
     where $\zeta\in C^\infty(\R)$ is a fixed non-increasing cut-off such that $\zeta|_{\R_{\leq 1/2}}=1$, $\zeta|_{\R_{\geq 1}}=0$; $\delta\in C^\infty(\R; [0,1])$ is $0$ outside $(-1, 1)$, positive and very small in $(-1, 1)$, to be specified in Lemma \ref{Lem_SmallDelta} below.
     \item When $a\in \RP^5$ does not take the form above, we set $\rho(a, x)=0$.
   \end{enumerate}
   \begin{rmk} \label{Rem_Prop_rho(a,x)}
     The following facts can be verified directly from the definition:
     \begin{enumerate} [label={\normalfont(\roman*)}]
       \item $\rho$ is a well-defined continuous function on $\RP^5\times S^3$ and is $C^\infty$ in $x$ variable.
       \item\label{Item_rho(a,x)_genus0} For $a\in \RP^5$ that is not in a small neighborhood of $A_{\sing}$, $\rho(a, \cdot)\equiv 0$, hence (under the parametrization \eqref{eq:coor}) by Lemma \ref{Lem_5-family_SingParam}, $\Psi(a, z) = \Phi_5(a)$ is a surface with at most $9$ singular points and genus $0$.
       \item When $\rho(a, x)\neq 0$ and $a_5=1$, \[
       (x_1+a_2)^2+(x_2+a_1)^2\leq \frac1{64}(1-a_1^2-a_2^2)^2 < \frac1{16}\left(1-\sqrt{a_1^2+a_2^2} \right)^2\,, 
       \]
       hence $|x|\leq (1+3\sqrt{a_1^2+a_2^2}\ )/4<1$. As a consequence, $\Psi(a, z)$ agrees with $\Phi_5(a)$ in a small neighborhood of $\partial \bD\times \bbS^1$ (whose image under parametrzation \eqref{eq:coor} is a neighborhood of $C$ in $S^3$).  We shall establish the smoothness of $\Psi(a, z)$ away from at most $3$ points in the interior of $\bD\times \bbS^1$, which then implies the regularity of $\Psi(a, z)$ in the whole $S^3$.
       \item When $a\in A_{\sing}$, $\rho(a,x)=\delta(a_1^2+a_2^2)$ is a nonzero constant for $(x, \alpha)$ close to the singular circle of $\Phi_5(a)$. Heuristically, this means $\Psi(a, z)$ is a desingularization of $\Phi_5(a)$ using some \textit{trigonometric polynomials} in $\alpha$ variable, whose number of zeros bounds the genus of $\Psi(a, z)$ as in Lemma \ref{Lem_Constru Psi_GenusCal of (F=0)}.
     \end{enumerate}
   \end{rmk}

   \begin{lem} \label{Lem_SmallDelta}
     There exists some function $\delta\in C^\infty(\R)$, which vanishes outside $(-1, 1)$ and is positive and small enough in $(-1,1)$, such that $\Psi(a, z)$ and $\ F_{a,z}$ (defined at the beginning of this subsection) satisfy the following \footnote{$a_5$ is set to be $1$, so the domains $\{\pm F_{a,z}>0\}$ are well-defined for such $a$.} . For every $a=[a_0: a_1:\dots:a_4:1]\in \RP^5$ with $\rho(a,\cdot)$ not constantly zero, if we let $\kappa(a):=(1-\sqrt{a_1^2+a_2^2})/4$, then we have:
     \begin{enumerate} [label={\normalfont(\alph*)}]
       \item\label{Item_Smalldelta_bx error} There is a unique $\bx_a \in C^\infty(\bbS^1; \bD_{\kappa(a)/4}(-a_2, -a_1))$ so that for $(x, \alpha) \in \bD_{\kappa(a)}(-a_2, -a_1)\times \bS^1$, 
       \begin{align*}
         \nabla_x F_{a,z}(x, \alpha) = 0\,, \quad \Leftrightarrow \quad
         x = \bx_a(\alpha)\,.
       \end{align*}
       And for some constant $C=C(a^2_1+a^2_2)>0$ we have,
       \[
         |\bx_a(\alpha) - (-a_2, -a_1)| \leq C\cdot (|a_3| + |a_4|)\,
       \] 
       \item\label{Item_Smalldelta_ZerosBd} Let $f_{a,z}\in C^\infty(\bbS^1)$ be given by $f_{a, z}(\alpha):= F_{a,z}(\bx_a(\alpha), \alpha)$. Then 
       \begin{align*}
         \# f_{a,z}^{-1}(0) \leq \begin{cases}
            4\,, &\ \text{ if } |z| = 1\,; \\
            6\,, &\ \text{ if } |z| < 1\,.
         \end{cases}
       \end{align*}
       Moreover, if the equality holds when $z=e^{i\theta}$ (hence $|z|=1$), then each zero of $f_{a,z}$ has order $1$, and \[
          \dist_{\bbS^1}(\cZ^-(f_{a,z}), -\theta) \leq 1000^{-1} \,.  
       \]
       \item\label{Item_Smalldelta_GenusBd} For every $z\in \D$, $\Psi(a, z)$ is a smooth surface in $S^3$ away from at most $3$ points, with the genus estimate
       \begin{align*}
          \mathfrak{g}(\Psi(a,z)) \leq \begin{cases}
            1\,, &\ \text{ if } |z| = 1\,; \\
            2\,, &\ \text{ if } |z| < 1\,.
          \end{cases}
       \end{align*}
       And equality holds when $|z|=1$ if and only if $\# f_{a,z}^{-1}(0)=4$.
     \end{enumerate}
   \end{lem}

   \begin{proof}[Proof of Lemma \ref{Lem_SmallDelta}]
     To prove \ref{Item_Smalldelta_bx error}, in the following $\Lambda_i>1$  will be absolute constants changing from line to line, and $C(a^2_1+a^2_2)>0$ also changes from line to line. Recall that 
     \begin{align*}
        F_{a, z}(x,\alpha) = (x_1+a_2)(x_2+a_1) & + (a_0-a_1a_2) + \varsigma(x)(a_3\cos\alpha + a_4\sin\alpha) \\
         & + \rho(a, x)(r\cos(\theta+2\alpha) + (1-r)\cos(3\alpha))
     \end{align*}
     where $\rho(a, \cdot)$ is constant for $x\in \bD_{\kappa(a)/4}(-a_2, -a_1)$, $\rho(a,\cdot)=0$ for $x\notin \bD_{\kappa(a)}(-a_2,-a_1)$ and in general, 
     \begin{align*}
       |\nabla_x\rho|\leq \Lambda_1(1-a_1^2 - a_2^2)^{-2}\cdot \delta(a_1^2+a_2^2)\,. 
     \end{align*}
     Also since $\rho(a, \cdot)\neq 0$, we have 
     \begin{align}
       a_3^2 + a_4^2 + (a_0-a_1a_2)^2 \leq \delta(a_1^2+a_2^2)\,. \label{Equ_Smalldelta_|a_3|+|a_4| bd by delta}
     \end{align}
     And in $\bD_{\kappa(a)}(-a_2, -a_1)$, \[
       |\varsigma| + |\nabla \varsigma| + |\nabla^2\varsigma| \leq \Lambda_2\ \kappa(a)^{-4}\,.
     \]
     Hence, in $\bD_{\kappa(a)}(-a_2, -a_1)\setminus\bD_{\kappa(a)/4}(-a_2, -a_1)$, 
     \begin{align*}
       \nabla_x F_{a,z}(x, \alpha) & = (x_2+a_1, x_1+a_2) + O\left(|a_3|+|a_4| + \delta(a_1^2+a_2^2) \right)\kappa(a)^{-4} \\
       & = (x_2+a_1, x_1+a_2) + O\left(\delta(a_1^2+a_2^2)^{1/2} \right)\kappa(a)^{-4} \\
       & \neq 0
     \end{align*}
     as long as $\delta(a_1^2+a_2^2)\leq \Lambda_3^{-1}\kappa(a)^{10}$. While in $\bD_{\kappa(a)/4}(-a_2, -a_1)$, 
     \begin{align*}
       \nabla_x F_{a,z}(x, \alpha) & = (x_2+a_1,\ x_1+a_2) + (a_3\cos\alpha + a_4\sin\alpha)\nabla\varsigma(x)\,; \\
       \nabla_x F_{a,z}(-a_2, -a_1, \alpha) & = O(|a_3|+|a_4|)\kappa(a)^{-4} = O(\delta(a_1^2+a_2^2)^{1/2})\kappa(a)^{-4}\,;
     \end{align*}
     \begin{align}
       \nabla^2_x F_{a,z}(x, \alpha) = \begin{bmatrix}
         0 & 1 \\ 1 & 0 
       \end{bmatrix} + O(|a_3|+|a_4|)\kappa(a)^{-4} \,. \label{Equ_Smalldelta_nabla^2_x F bd}
     \end{align}
     Hence consider \[
       \varpi: (y_1, y_2) \mapsto -\begin{bmatrix}
        0 & 1 \\ 1 & 0
       \end{bmatrix} \nabla_xF_{a,z}(-a_2+y_1, -a_1+y_2, \alpha) + (y_1, y_2)
     \] 
     By taking $\delta(a_1^2+a_2^2) \leq \Lambda_4^{-1}\kappa(a)^{10}$, we can see that $\varpi$ has small Lipschitz constant and is a contraction map from $\bD_{\kappa(a)/4}$ to itself, and by the Banach fixed-point theorem, it has a unique fixed point (which corresponds to zero of $\nabla_x F_{a, z}$) with length no more than \[
       |\varpi(0,0)| + \|\nabla\varpi\|_{C^0(\bD_{\kappa(a)/4})} \leq C(a_1^2+a_2^2)(|a_3| + |a_4|) \,.
     \]
     This finishes the proof of \ref{Item_Smalldelta_bx error}. Moreover, since $\nabla_xF_{a,z}(\bx_a(\alpha), \alpha) = 0$, by chain rule we have, 
     \begin{align}
        \|\bx_a - (-a_2, -a_1)\|_{C^6(\bbS^1)} \leq C(a_1^2+a_2^2)(|a_3|+|a_4|) \,. \label{Equ_Smalldelta_|bx_a|_C^6 bd}
     \end{align}

     To prove \ref{Item_Smalldelta_ZerosBd}, notice that 
     \begin{align*}
       T_{a,z}(\alpha) & := F_{a,z}(-a_2, -a_1, \alpha) \\
       & \ = (a_0-a_1a_2) + \varsigma(-a_2, -a_1)(a_3\cos\alpha +a_4\sin\alpha) \\
       & \ \ + \rho(a,( -a_2, -a_1))(r\cos(\theta+2\alpha)+(1-r)\cos(3\alpha))
     \end{align*}
     is a trigonometric polynomial of degree $\leq 3$ for every $(a, z)$, and it is of degree $2$ if $z\in \partial\D$. And since $\rho(a, \cdot)$ is a constant in $\bD_{\kappa(a)/4}$, 
     \begin{align*}
       f_{a,z}(\alpha) - T_{a,z}(\alpha) & = (\bx_a(\alpha)_1+a_2)(\bx_a(\alpha)_2+a_1) \\
       & + (\varsigma(\bx_a(\alpha))-\varsigma(-a_2, -a_1))(a_3\cos\alpha +a_4\sin\alpha)\,.
     \end{align*}
     Therefore
     \begin{align*}
       \|f_{a,z} - T_{a,z}\|_{C^6(\bbS^1)} & \leq C(a_1^2+a_2^2) \Big(\|\bx_a-(-a_2,-a_1)\|_{C^6(\bbS^1)}^2 \\
       & \;\;\qquad\qquad +\; (|a_3|+|a_4|)\|\bx_a-(-a_2,-a_1)\|_{C^6(\bbS^1)} \Big) \\
       & \leq C(a_1^2+a_2^2)(|a_3|+|a_4|)^2 \\
       & \leq \bar C(a_1^2+a_2^2)(|a_3|+|a_4|)\|T_{a,z}\|\,.
     \end{align*}
     Hence by taking $\delta(a_1^2+a_2^2)\ll 1$ such that $\bar C(a_1^2+a_2^2)(|a_3|+|a_4|) \leq \delta_1(3, 1000^{-1})$ from Corollary \ref{Cor_alm trig polyn has zero bd}, \ref{Item_Smalldelta_ZerosBd} follows directly by Corollary \ref{Cor_alm trig polyn has zero bd}. 

     To prove \ref{Item_Smalldelta_GenusBd}, by \eqref{Equ_Smalldelta_|a_3|+|a_4| bd by delta} and \eqref{Equ_Smalldelta_nabla^2_x F bd}, we can further assume $\delta(a_1^2+a_2^2)$  to be small   such that assumptions \ref{Item_GenusCalc_asump1}-\ref{Item_GenusCalc_asump3} of Lemma \ref{Lem_Constru Psi_GenusCal of (F=0)} hold for $F_{a,z},\ (-a_2, -a_1),\ \kappa(a),\ \bx_a$ in place of $F,\ b,\ \kappa,\ \bx$ therein. Hence, \ref{Item_Smalldelta_GenusBd} follows immediately from \ref{Item_Smalldelta_ZerosBd} proved above together with Lemma \ref{Lem_Constru Psi_GenusCal of (F=0)} \eqref{Item_GenusCalc_concl_reg} and \eqref{Item_GenusCalc_concl_genus}. 
   \end{proof}

   \begin{proof}[Proof of Theorem \ref{prop:Psi1}.]
     The construction of $\Psi$ has been carried out above with the choice of $\delta$ fixed by Lemma \ref{Lem_SmallDelta}. 
     The assertion that $\Psi$ takes value in the set of punctate surfaces of genus $\leq 2$ and \eqref{Item:Family_Genus0}, \eqref{Item:Family_Genus1} of Theorem \ref{prop:Psi1} follow from \ref{Item_Smalldelta_GenusBd} of Lemma \ref{Lem_SmallDelta} and \ref{Item_rho(a,x)_genus0} of Remark \ref{Rem_Prop_rho(a,x)}.       
   \end{proof}

\subsection{Symmetry of $\bar \Xi$ and $\Psi$}\label{sect:tildeXiInvariantProof}
In this section, we prove Proposition \ref{prop:tildeXiInvariant}. Recall that in \eqref{Equ_Def_Psi(a,z)}, $\Psi(a,z)$ is defined using $\Phi_5(a)$ in \eqref{Equ_Def_Phi_5(a)} when $a_5= 0$, and using $F_{a,z}$ in \eqref{Equ_Def_F_a,z} when $a_5\neq 0$: 
  \begin{align}
    \begin{split}
      F_{a,re^{i\theta}}(x, \alpha) := a_5^{-1} \Big( a_0 & + a_1x_1+a_2x_2+a_3x_3+a_4x_4\\
        & +a_5\left[x_1x_2+ \rho(a,x)\left(r\cos(\theta+2\alpha)+(1-r)\cos3\alpha\right)\right]\Big).   
    \end{split}
  \end{align}
  Therefore, Proposition \ref{prop:tildeXiInvariant} is a direct consequence of the following lemma.
  \begin{lem} \label{Lem_D_24Sym_F_a,z(p)}
    Recall the left-$G$ action $\sigma$ on $\RP^5\times \D$ is defined in \eqref{Equ_D_24-action_Y}. For $i\in \{1,2\}$, we have   \[
      F_{\sigma(g_i^{-1})(a,z)}(p) = -F_{a,z}(g_i\cdot p) , \qquad \forall\, (a,z)\in \{a_5\ne 0\}\x\D, \ \  \forall\, p\in S^3 \,,
    \] 
    where the $G$-action on $S^3$ was given in \eqref{eqn:g_act2}.
    In particular, the $7$-parameter family $\Psi: \cY\to \cS_{\leq 2}(S^3)$  is $G$-equivariant, where the $G$-action on $\cS_{\leq 2}(S^3)$ is induced by \eqref{eqn:g_act2}.
  \end{lem}
  \begin{proof}For $g_1$, we have
    \begin{equation*}
      \begin{aligned}
        F_{a,z}(g_1\cdot p) &= a^{-1}_5(a_0 + a_1 (-x_2) + a_2 x_1 + \begin{pmatrix} a_3 & a_4 \end{pmatrix} \begin{pmatrix} \cos\frac{\pi}{3} & -\sin\frac{\pi}{3}  \\ \sin \frac{\pi}{3}  & \cos \frac{\pi}{3}
        \end{pmatrix}\begin{pmatrix} x_3 \\x_4 \end{pmatrix}\\
        & \quad + a_5 \left[ (-x_2) x_1 + \rho(a,g_1\cdot x) \left(r \cos (\theta + 2 (\alpha + \pi/3)) + (1-r) \cos3(\alpha + \pi/3 )\right)\right])\\
        &= a^{-1}_5(a_0 + a_2 x_1 + (-a_1) x_2 + \left(a_3 \cos \frac{\pi}{3} + a_4 \sin \frac {\pi}{3}\right) x_3 + \left(-a_3 \sin \frac{\pi}{3} + a_4 \cos \frac{\pi}{3}\right) x_4\\
        & \quad + (-a_5) \left[ x_1 x_2 + \rho(g_1^{-1}\cdot a,   x) \left(r \cos ( (\theta - \pi/3 ) + 2\alpha  ) + (1-r)\cos3\alpha\right)\right])\\
        &= -F_{\sigma(g_1^{-1})(a,z)}(p)\,,
      \end{aligned}
    \end{equation*}
    where the first equality uses ~\eqref{eqn:g_1_act1} and ~\eqref{eqn:g_act2}, the second uses \eqref{eq:g1ActRP5}, 
     and the third uses ~\eqref{Equ_D_24-action_Y}. Note, the notation $g_i^{-1}\cdot a$ involves an obvious abuse of notation: It refers to the $G$-action on $\RP^5$ induced by \eqref{eq:g1ActRP5} and \eqref{eq:g2ActRP5}. 
    Similarly, by ~\eqref{eqn:g_2_act1}, ~\eqref{eqn:g_act2}, \eqref{eq:g2ActRP5}, and~\eqref{Equ_D_24-action_Y},
    \begin{equation*}
        \begin{aligned}
            F_{a,z}(g_2\cdot p) &= a^{-1}_5(a_0 + a_1 (-x_1) + a_2 x_2 + a_3 (-x_3) + a_4 x_4\\
            & \quad + a_5 \left[ (-x_1) x_2 + \rho(a, g_2 \cdot x) \left( r \cos(\theta + 2(\pi - \alpha)) + (1 - r) \cos 3(\pi - \alpha) \right) \right])\\
            & = a^{-1}_5(a_0 + (-a_1) x_1 + a_2 x_2 + (-a_3) x_3 + a_4 x_4\\
            &\quad  + (-a_5) \left[ x_1 x_2 + \rho(g_2^{-1}\cdot a,  x) \left(r \cos ((\pi-\theta) + 2\alpha) + (1 - r) \cos 3\alpha \right) \right])\\
            &=-F_{\sigma(g_2^{-1})(a,z)}(p)\,.
        \end{aligned}
    \end{equation*}
  \end{proof}

\section{Topology of $\cX$} \label{sect:topoXi}

In this section, we prove Lemma \ref{lem:BisProduct} and Theorem \ref{thm:capProduct}.

\subsection{Proof of Lemma \ref{lem:BisProduct}}\label{sect:spaceLawsonProof} 
    
    First, the map 
    \[
      \Pi_1: B=(\mathbb S^3\times \mathbb S^3)/\hat G \to \mathbb S^3/\hat\Pi_1(\hat G), \quad (q_1, q_2)\cdot \hat G\mapsto q_1\cdot \hat\Pi_1(\hat G)\,,
    \]  
    is clearly well-defined, where $(q_1, q_2)\cdot \hat G$ and $q_1\cdot \hat\Pi_1(\hat G)$ denote the left cosets. This gives a principal $\mathbb S^3$-bundle via the fiberwise right action 
    \begin{align*}
      \tau: ((\mathbb S^3\times \mathbb S^3)/\hat G)\times \mathbb S^3 & \to (\mathbb S^3\times \mathbb S^3)/\hat G, \\
      ((q_1, q_2)\cdot \hat G,\ q') & \mapsto (q_1, q'^{-1}q_2)\cdot \hat G
    \end{align*}
    Since the base of this bundle has dimension $3$, by the obstruction theory, it admits a smooth global section $\mathfrak s: \mathbb S^3/\hat\Pi_1(\hat G)\to (\mathbb S^3\times \mathbb S^3)/\hat G$, i.e., $\pi\circ \mathfrak s = \id_{\mathbb S^3/\hat\Pi_1(\hat G)}$. Alternatively, we extend $(\mathbb S^3\times \mathbb S^3)/\hat G$ to a real vector bundle $\bV$ over $\mathbb S^3/\hat\Pi_1(\hat G)$ of rank $4$, in which $(\mathbb S^3\times \mathbb S^3)/\hat G\subset \bV$ is the unit sphere bundle. By a perturbation of the zero section, there exists a section $\tilde {\mathfrak s}$ of $\bV$ transverse to the zero section $\tilde {\mathfrak s}_0$ in $\bV$, and thus, doesn't intersect $\tilde{\mathfrak s}_0$ by counting dimensions; in other words, $\tilde{\mathfrak s}$ is nonvanishing everywhere. Then $\mathfrak s:= \tilde{\mathfrak s}/|\tilde{\mathfrak s}|$ is a global section of $(\mathbb S^3\times \mathbb S^3)/\hat G$.

    The triviality of this $\mathbb{S}^3$-bundle follows from this global section $\mathfrak s$, which induces a diffeomorphism
    \[
      \mathbb S^3\times (\mathbb S^3/\hat\Pi_1(\hat G)) \to (\mathbb S^3\times \mathbb S^3)/\hat G, \quad (q', q\cdot \hat\Pi_1(\hat G)) \mapsto \tau(\mathfrak s(q\cdot \hat\Pi_1(\hat G)), q')\,.
    \]

\subsection{Proof of Theorem \ref{thm:capProduct}}\label{sect:topologyX}
    Recall that $\cX = \tilde \cX/\tilde G$, where $\tilde \cX:= S^5\times \D \times \mathbb S^3\times \mathbb S^3$, parametrized by $(a, z, z_1+z_2j, w_1+w_2j)$, and $\tilde G = \Z_2\times Q_{48}$, with generators $\tilde{g}_0, \tilde{g}_1,$ and $\tilde{g}_2$ acting on $\tilde\cX$ by \eqref{eqn:action_tilde_g}. Here, we view $\mathbb S^3\subset \C^2$ as the group of unit quaternions, and write any point $(w_1,w_2)\in \C^2$ as the quaternion $w_1+w_2j$.

    We shall apply Appendix \ref{Append_PoincareDual} with $\tilde \cX, \tilde G$ in place of $\tilde X, H$ therein to construct Poincar\'e duals of first cohomology classes. 
    
    We first construct $\tilde G$-equivarient functions on $\tilde \cX$, concentrating on its first $\mathbb{S}^3$ factor.
    \begin{lem} \label{Lem_AssoBundle_PD of c_1*+c_2*}
      Let $\epsilon:=2026^{-1}$. Define $v_1, v_2, v_3\in C^\infty(\tilde\cX, \R)$ (that only depend on $(z_1, z_2)$) by 
      \begin{align*}
        v_1 := \Re(z_1^{12}-z_2^{12})\,, \quad
        v_2 := \Imag(z_1^{12}-z_2^{12})\,, \quad 
        v_3 := \Re \left(e^{i\epsilon\pi}(z_1z_2^{13}+z_1^{13}z_2) \right)\,.
      \end{align*}
      Then they satisfy the following:
      \begin{enumerate} [label={\normalfont(\roman*)}]
       \item\label{Item_AssoBundle_v_ell G-equivar} For each $\ell \in \{1, 2, 3\}$, \[
         -v_\ell(p \tilde g_0) = v_\ell(p \tilde g_1) = v_\ell(p \tilde g_2) = -v_\ell (p), \qquad \forall\, p\in \tilde\cX \,.
       \]
       \item\label{Item_AssoBundle_|v_ell =0|=48*7} $\{v_1=v_2=v_3=0\} = S^5\times \D\times Z_{\alpha^3}\times \mathbb S^3$, where $Z_{\alpha^3}\subset \mathbb S^3$ consists of $48\times 7$ points;
       \item\label{Item_AssoBundle_nabla v_ell linear indep} At each $p\in \{v_1=v_2=v_3=0\}$, $\{\nabla v_1(p), \nabla v_2(p), \nabla v_3(p)\}$ are linear independent in $T_p \tilde\cX^3$.
      \end{enumerate}
    \end{lem}
    \begin{proof}
      Note that by \eqref{eqn:action_tilde_g} and the definitions of $\hat g_1, \hat g_2$ in \S \ref{subsec:symm_gp_lawson_surf}, 
      \begin{align*}
        (z_1+z_2j, w_1+w_2j)\hat g_1 & = (z_1e^{i5\pi/12} + (z_2e^{-i5\pi/12})j, w_1e^{-i\pi/12}+(w_2e^{i\pi/12})j)\,; \\
        (z_1+z_2j, w_1+w_2j)\hat g_2 & = (-z_2+z_1j, w_2-w_1j) \,.
      \end{align*}
      \ref{Item_AssoBundle_v_ell G-equivar} thus follows from a direct calculation. For the rest, it suffices to only work with $(z_1, z_2)$ variables.

      To prove \ref{Item_AssoBundle_|v_ell =0|=48*7}, first notice that $v_1(z_1, z_2)=v_2(z_1, z_2)=0$ is equivalent to $z_2^{12}=z_1^{12}$, and since $|z_1|^2+|z_2|^2=1$, this is to say,  
        $z_1=2^{-1/2}e^{i\theta}$ and $ z_2=2^{-1/2}e^{i(\theta+l\pi/6)} $, 
      for some $l\in \{0, 1, \dots 11\}$. For each such $l$, plugging this into $v_3=0$ yields \[
        0 = 2\Re(e^{i\epsilon\pi}z_1^{13}z_2) = 2^{-6}\, \Re(e^{i(14\theta+\epsilon\pi+l\pi/6)})\,;
      \]
      which has $28$ solutions for $\theta\ \operatorname{mod} 2\pi$ for each $l$. 
      Therefore, $\{v_1=v_2=v_3\}$ consists of $12\times 28=48\times 7$ points in $\mathbb S^3$. We denote them by 
      \begin{align}
        Z_{\alpha^3} = \left\{ 2^{-1/2}(e^{i\theta}, e^{i(\theta+l\pi/6)}): 14\theta+(\epsilon+\frac l6)\pi=(\frac12+m)\pi,\ l\in \{0, \dots, 11\},\, m\in \{0, \dots, 27\} \right\}\,.  \label{Equ_Z_alpha^3}
      \end{align}

      To prove \ref{Item_AssoBundle_nabla v_ell linear indep}, first note that each $v_\ell$ naturally extends to a homogeneous polynomial $V_\ell$ on $\C^2=\R^4$. Hence at every $p\in \{v_1=v_2=v_3=0\}$, $p\cdot\nabla_{\R^4} V_\ell(p) = 0$ and hence $\nabla_{\R^4} V_\ell(p) = \nabla v_\ell(p)$. By Lemma \ref{Lem_AssoBundle_nabla Re(f) vs partial_z f}, it suffices to verify that at $p$, $\partial_z(z_1^{12}-z_2^{12})$ and $\partial_z(z_1z_2^{13}+z_1^{13}z_2)$ are $\C$-linearly independent in $\C^2$. A direct calculation proves this. 
    \end{proof}

    Another class of $\tilde G$-equivarient functions comes from $S^5\times \mathbb S^3$ directions. 
    \begin{lem} \label{Lem_PD of c_0^*+error}
      Define $f_0, \dots, f_4\in C^\infty(\tilde\cX, \R)$ by 
      \begin{align*}
        f_0 = a_0\,, \quad
        f_1+if_2 = (\bar z_1^6+z_2^6)(a_1+ia_2)\,, \quad
        f_3+if_4 = (\bar z_1^4+z_2^4)(a_3+ia_4)\,.
      \end{align*}
      They satisfy the following properties:
      \begin{enumerate}[label={\normalfont(\roman*)}]
        \item\label{Item_f_ell G-equivar} For every $p\in \tilde\cX$, $0\leq \ell \leq 4$ and $0\leq m\leq 2$, $f_\ell(p \tilde g_m) = \phi_\ell(\tilde g_m)f_\ell(p)$ for some $\phi_\ell(\tilde g_m)\in \{\pm 1\}$ determined by the following list,
        \begin{center} 
          \begin{tabular}{c | c | c | c | c | c} 
                       & $\phi_0$ & $\phi_1$ & $\phi_2$ & $\phi_3$ & $\phi_4$ \\  \hline 
          $\tilde g_0$ & $-1$ & $-1$ & $-1$ & $-1$ & $-1$ \\  \hline
          $\tilde g_1$ & $ 1$ & $-1$ & $-1$ & $ 1$ & $ 1$ \\  \hline
          $\tilde g_2$ & $ 1$ & $-1$ & $ 1$ & $-1$ & $ 1$ \\ 
          \end{tabular}
        \end{center} 
        \item\label{Item_nabla f_ell linear indep} For every $(z_1, z_2)\in Z_{\alpha^3}$ from Lemma \ref{Lem_AssoBundle_PD of c_1*+c_2*}, \[
          \{a\in S^5: f_\ell(a, z, z_1+z_2j, w_1+w_2j)=0,\ \forall\, 0\leq \ell\leq 4\} = \{(0,0,0,0,0, \pm 1)\}
        \]
        and at each point in this set, $\{\nabla_{S^5}f_\ell\}_{0\leq \ell\leq 4}$ are linearly independent.
      \end{enumerate}
    \end{lem}
    \begin{proof}
      Similar to Lemma \ref{Lem_AssoBundle_PD of c_1*+c_2*}, \ref{Item_f_ell G-equivar} follows by a straightforward verification.
      To prove \ref{Item_nabla f_ell linear indep}, notice that for fixed $(z_1, z_2)$, $f_\ell$ are restrictions on $S^5$ of linear functions on $\R^6$ that are independent of $a_5$-variable. Thus, it suffices to show that \[
        (\bar z_1^6 + z_2^6) \neq 0, \qquad (\bar z_1^4 + z_2^4) \neq 0
      \]
      for every $(z_1, z_2)\in Z_{\alpha^3}$ from Lemma \ref{Lem_AssoBundle_PD of c_1*+c_2*}. This also follows from a direct calculation using \eqref{Equ_Z_alpha^3}.
    \end{proof}

    \begin{proof}[Proof of Theorem \ref{thm:capProduct}.]
      We apply Corollary \ref{Cor_Poincare dual via line bundles} to $v_1, v_2, v_3, f_0, \dots, f_4\in C^\infty(\tilde\cX)$. Note that Lemma \ref{Lem_AssoBundle_PD of c_1*+c_2*} \ref{Item_AssoBundle_nabla v_ell linear indep} and Lemma \ref{Lem_PD of c_0^*+error} \ref{Item_nabla f_ell linear indep} together imply the linear independence assumptions in Corollary \ref{Cor_Poincare dual via line bundles}. Also, by Lemma \ref{Lem_AssoBundle_PD of c_1*+c_2*} \ref{Item_AssoBundle_v_ell G-equivar} and Lemma \ref{Lem_PD of c_0^*+error} \ref{Item_f_ell G-equivar}, cohomology classes corresponding to $v_1, v_2, v_3$ are all $c_1^*+c_2^* = \alpha$, while the cohomology class corresponding to $f_\ell$ has form $c_0^*+d^*_\ell$ for some $d^*_\ell\in \Span\{c_1^*, c_2^*\}$. Thus, Corollary \ref{Cor_Poincare dual via line bundles} yields that the Poincar\'e dual of $(\smile_{\ell=0}^4(c_0^*+d_\ell^*))\smile\alpha^3\in H^8(\cX, \Z_2)$ is \[
        Z:=\{v_1=v_2=v_3=f_0=\dots=f_4=0\}/\tilde G 
        = (\{(0, \dots, 0, \pm 1)\}\times \D\times Z_{\alpha^3}\times \mathbb S^3)/\tilde G \subset \cX\,.
      \]

      Also recall that $\pi: \cX\to B:= (\mathbb S^3\times \mathbb S^3)/\hat G$, $(a, z, z_1+z_2j, w_1+w_2j)\cdot \tilde G \mapsto (z_1+z_2j, w_1+w_2j)\cdot \hat G$ is a fiber bundle; while by Lemma \ref{lem:BisProduct}, $\Pi_1: B\to \mathbb S^3/\hat\Pi_1(\hat G)$, $(q_1, q_2)\cdot \hat G\mapsto q_1\cdot \hat\Pi(\hat G)$ is a trivial $\mathbb S^3$-bundle, and $\beta=\pi^*\beta_B$ where $\beta_B\in H^3(B, \Z_2)$ is the Poincar\'e dual of a section $B_0$ of the fiber bundle $\Pi_1$. Set $\hat B_0$ be the inverse image of $B_0$ under $\mathbb S^3\times \mathbb S^3\to B$. Thus, the Ponicar\'e dual of $(\smile_{\ell=0}^4(c_0^*+d_\ell^*))\smile\alpha^3\smile\beta\in H^{11}(\cX, \Z_2)$ is \[
        Z\cap \pi^{-1}(B_0) = \left(\{(0, \dots, 0, \pm 1)\}\times \D\times (Z_{\alpha^3}\times \mathbb S^3 \cap \hat B_0)\right)/\tilde G\,,
      \]
      which consists of $(2\times 48\times 7)/(\# \tilde G) = 7$ copies of $\D$ in the fibers of $\pi: \cX\to B$, each homologous to $D_0$ defined above Theorem \ref{thm:capProduct}. Therefore, we have \[
        [\cX]\frown \left((\smile_{\ell=0}^4(c_0^*+d_\ell^*))\smile\alpha^3\smile\beta\right) = 7[D_0] = [D_0]\,, \quad \text{ in } H_2(\cX, \partial \cX; \Z_2) \,.
      \]

      Finally, we claim that $(\smile_{\ell=0}^4(c_0^*+d_\ell^*))\smile\alpha^3\smile\beta = \lambda^5\smile\alpha^3\smile\beta$ in $H^{11}(\cX, \Z_2)$, and this will complete the proof of Theorem \ref{thm:capProduct}. 
      To see this,  notice that by definition   $\lambda=c_0^*+c_1^*+c_2^*$, and there exist $c_{1,B}^*, c_{2, B}^*\in H^1(B, \Z_2)$ such that $\pi^* c_{l, B}^*=c_l^*$, $l=1,2$. Thus, $\alpha^3\smile\beta = \pi^*(\alpha_B^3\smile\beta_B)$ and there are some $d_{\ell, B}\in H^1(B, \Z_2)$ such that \[
        (\smile_{\ell=0}^4(c_0^*+d_\ell^*))\smile\alpha^3\smile\beta 
        = (\smile_{\ell=0}^4(\lambda + \pi^* d_{\ell, B}))\smile\pi^*(\alpha_B^3\smile\beta_B)
        = \lambda^5\smile\pi^*(\alpha_B^3\smile\beta_B)\,,
      \]
      where the last equality follows from that $d_{\ell,B}\smile(\alpha_B^3\smile\beta_B)=0$ for every $\ell$, since $\dim B=6<7$. This finishes the proof of the claim.
    \end{proof}

\section{Properties of the family $\Xi$}\label{sect:XiProper}
In this section, we prove Theorem \ref{thm:XiProper}.

First, Theorem \ref{thm:XiProper} \eqref{Item:FamilyXi_Genus0} and \eqref{Item:FamilyXi_Genus1} follow immediately from the definition and Theorem \ref{prop:Psi1}.

For \eqref{Item:FamilyXi_1-sweepout},
since $\Xi|_C$ maps into $\cS_2(S^3)$, we know by \eqref{Item:FamilyXi_Genus0} that $C$ must lie in the subset $(\{\rho\ne 0\}\x\D\x SO(4))/D_{24}$. Thus, $[C]\in H_1(\cX)$ has no $c_0$ component (recall the notation in \S \ref{sect:TopoX}). Then, since $\alpha:=c_1^*+c_2^*$, we have  $[C]=c_1$ or $c_2.$  
Now, choose an embedded path $\bar\sigma_1(t)$ (resp. $\bar\sigma_2(t)$)  in $SO(4)$, with $t\in [0,1]$, that starts at $\id$ and ends at the point $g_1$ (resp. $g_2)$: See \S \ref{subsec:symm_gp_lawson_surf} for the definitions of $g_i\in SO(4)$. Then the path $t\mapsto ([0:\dots:0:1],0,\bar\sigma_i(t))$ in $\RP^5\x\D\x SO(4)$ descends to a {\it loop} in $\cX$, which we denote as $\sigma_i$. Let $x_0\in\cX$ be the image of $\tilde x_0$ (defined by  \eqref{eq:x0Tilde}) under the map $\tilde\cX\to\cX$. Note, both loops $\sigma_1,\sigma_2$ are based at $x_0$, and $\sigma_i$ induces the class $c_i\in H_1(\cX)$. 

Notice $\Xi(x_0)$  takes the form $\{x_1x_2+\rho([0:...:0:1],x)\cos3\alpha=0\}$, which can be readily shown to be of genus $2$   using  Lemma \ref{Lem_Constru Psi_GenusCal of (F=0)} \eqref{Item_GenusCalc_concl_genus}. Then since all members in the loops $\Xi|_{\sigma_1}$ and $\Xi|_{\sigma_2}$ are just rotations of $\Xi(x_0)$, they all have genus 2. Thus, recalling that the property of a loop in $\cZ_2(S^3;\Z_2)$ being an Almgren-Pitts 1-sweepout depends only on its homology class, without loss of generality we can assume that $C$ is exactly the loop $\sigma_1$ or $\sigma_2$. Hence, it suffices to check that $\Xi|_{\sigma_i}$ is a 1-sweepout for $i=1,2$. That is, if we continuously choose  inside and outside regions for $\Xi(\sigma_i(t))$ for each $t$, then the inside region for $\Xi(\sigma_i(0))$ should coincide  with the {\it outside} region of $\Xi(\sigma_i(1))$. To see this, recall that as in \eqref{eq:barXi}, \eqref{Equ_Def_Psi(a,z)} and \eqref{Equ_Def_F_a,z}, \[
  \Xi(\sigma_i(t)) = \bar\sigma_i(t)\cdot \Psi([0:\dots:1],0) = \partial\left(\bar\sigma_i(t)(\{F_{[0:\dots:1],0}>0\})\right) = \partial\{F_{[0:\dots:1],0}\circ \bar\sigma_i(t)>0\} \,.
\]  
Hence $t\mapsto \Omega_i(t)=\{F_{[0:\dots:1],0}\circ \bar\sigma_i(t)>0\}$ is a continuous family of regions bounded by $\Xi|_{\sigma_i}$. But by Lemma \ref{Lem_D_24Sym_F_a,z(p)} and \eqref{Equ_D_24-action_Y}, \[
  F_{[0:\dots:1],0}\circ \bar\sigma_i(1) = F_{[0:\dots:1],0}\circ g_i = -F_{\sigma(g_i^{-1})([0:\dots:1],0)} = -F_{[0:\dots:1],0}\circ \bar\sigma_i(0) \,.
\]
Therefore, $\Omega_i(1)=\Int (S^3\setminus \Omega_i(0))$. This shows that $\Xi|_{\sigma_i}$ is an Almgren-Pitts 1-sweepout for $i=1,2$.

For \eqref{Item:FamilyXi_5-sweepout}, it suffices to prove that $\lambda:=c^*_0+c^*_1+c^*_2$ is equal to  $\Xi^*(\bar\lambda)$, where $\bar\lambda$ is the generator of the cohomology ring of the cycle space, $H^*(\cZ_2(S^3;\Z_2);\Z_2)$. Hence, since $c_0,c_1,c_2$ generate $H_1(\cX)$, it suffices to show that for each $i$, if  $C\subset \cX$ is some loop representing $c_i$, then $\Xi|_C$ is a 1-sweepout. For $c_0$, we pick the loop that lies in the fiber $\cX|_{[\id]}\cong\RP^5\x\D$ and is given by  $$C=\{[a_0:a_1:0:...:0]:a_0^2+a_1^2=1\}\x \{0\}.$$
Then $\Xi|_C$ is the family $\{x\in S^3:a_0+a_1x_1=0\}$, which  is clearly a 1-sweepout. Moreover, the cases of $c_1$ and $c_2$ are already done above in the proof of item \eqref{Item:FamilyXi_1-sweepout}. This finishes the proof of item \eqref{Item:FamilyXi_5-sweepout}.

  To prove Theorem \ref{thm:XiProper} \eqref{Item:FamilyXi_2chain}, we first need to introduce the {\it Hopf links} in $S^3$.  Let $L:= \beta_+^H\sqcup \beta_-^H\subset S^3$ be the {\it standard Hopf link} in $S^3$, which consists of two great circles \[
    \beta^H_+:= \{(x_1,x_2,0,0): x_1^2+x_2^2=1\}, \qquad
    \beta^H_-:= \{(0,0,x_3,x_4): x^2_3+x^2_4=1\}\,.    
  \]
  Denote
    $\cE(L, S^3) : = \Diff(S^3)/\Diff(S^3, L),$
  where $\Diff(S^3, L)$ denotes the subgroup  of all diffeomorphisms $\phi\in \Diff(S^3)$ such that $\phi(L)=L$ (see Appendix \ref{Append_HopfLink} for a discussion  on the topology of $\cE(L,S^3)$). Without causing confusion, we can viewed each element of $\cE(L,S^3)$ as a link in $S^3$, which we call a {\it smooth Hopf link}. We will also sometimes consider Hopf links made up of piecewise smooth segments: The set of all {\it piecewise smooth Hopf links} will be denoted $\check\cE(L,S^3)$.

  Next, let $\cU\subset\cS_1(S^3)$ be the set of all elements $S\in \cS_1(S^3)$ that are {\it unknotted} in the following sense: Let $\Omega,\Omega'\subset S^3$ be the two open regions with topological boundary $S$ (such that $\Omega\sqcup\Omega'= S^3\backslash S$). Then there exist smooth embedded loops $\alpha\subset \Omega$ and $\beta\subset\Omega'$ that form a Hopf link.  In this case, we say that {\it $S$ bounds the Hopf link $\alpha\cup\beta$.} The following fact will be helpful, and its proof is included in Appendix  \ref{sect:UnknottedTlements}. 
  \begin{prop}\label{prop:wholeFamilyUnknotted}
Let $\Phi:[0,1]\to\cS_1(S^3)$ be a Simon--Smith family such that $\Phi(0)\in \cU$. Then $\Phi(t)\in\cU$ for every $t\in[0,1]$.
\end{prop}

  The key to Theorem \ref{thm:XiProper} \eqref{Item:FamilyXi_2chain} is the following proposition. Let $\partial_1\cX:=\{y\in\partial \cX:\fg(\Xi(y))=1\}$.
  \begin{prop} \label{Lem_NontrivialLoop_Upsilon}
    There exists a continuous map $\Upsilon: \partial_1\cX \to \cE(L, S^3)$ with the following properties.
    \begin{enumerate} [label={\normalfont(\roman*)}]
      \item\label{Item_Upsilon_BdedHopfLinks} For every $y\in \partial_1\cX$, $\Xi(y)\in \cU$ and  $\Upsilon(y)$ is a Hopf link bounded by $\Xi(y)$.
      \item\label{Item_Upsilon_nontrivial-bD_0} $\Upsilon_*[\partial D_0] \neq 0$ in $H_1(\cE(L, S^3), \Z_2)$ (recall \eqref{eq:tildeD0} regarding $D_0$).
      \item\label{Item_Upsilon_Ker-iota_*-to-0} If $\vartheta\in H_1(\partial_1\cX, \Z_2)$ is such that $\iota_*\vartheta = 0$ in $H_1(\partial\cX, \Z_2)$, where $\iota: \partial_1\cX\hookrightarrow \partial\cX$ is the natural inclusion, then $\Upsilon_* \vartheta = 0$ in $H_1(\cE(L, S^3), \Z_2)$.
    \end{enumerate}
  \end{prop}
  Assuming this proposition, we can finish the proof of Theorem \ref{thm:XiProper} \eqref{Item:FamilyXi_2chain} easily:  Suppose $C\subset \partial \cX$ is  a closed $1$-subcomplex homologous to $\partial D_0$ in $\partial\cX$ and (since $C$ is mapped into $\cS_1(S^3)$ by $\Xi$) $C\subset\partial_1\cX$. Then by Proposition  \ref{Lem_NontrivialLoop_Upsilon}, we have 
  \begin{align}
    \Upsilon_*[C] = \Upsilon_*[\partial D_0] \neq 0\,, \quad \text{ in } H_1(\cE(L, S^3), \Z_2) \,.  \label{Equ_Upsilon_*[gamma] neq 0 in HopfLinks}
  \end{align}
  Now, we suppose for contradiction that there exists a Simon--Smith family $\Phi: W\to \cS_1(S^3)$ for some pure simplicial $2$-complex $W$  such that $\partial W=C$ and $\Phi|_{\partial W}=\Xi|_C$. 
  Due to Proposition \ref{prop:wholeFamilyUnknotted}, by possibly subtracting components of $W$ that do not intersect $C$, we may assume without loss of generality that $\Phi(W)$ only consists of unknotted elements, so that $\Phi(W)\subset \cU\subset \cS_1(S^3)$.
  Then, in fact,    $\Upsilon|_{\partial W}$ can be extended to a continuous map $\bar \Upsilon: W\to \cE(L, S^3)$: See Lemma \ref{Lem_Hopflink_ExtLoops} in the appendix, whose proof is an elementary argument using the basic properties of Simon--Smith family and the topology of the space $\cE(L,S^3)$. 
  Thus,  $\Upsilon_*[C] = \Upsilon_*[\partial W] = 0$ in $H_1(\cE(L, S^3), \Z_2)$, which   contradicts  \eqref{Equ_Upsilon_*[gamma] neq 0 in HopfLinks}. So we have obtained  Theorem \ref{thm:XiProper} \eqref{Item:FamilyXi_2chain}. 
  
\medskip
  The rest of this section is devoted to the proof of Proposition  \ref{Lem_NontrivialLoop_Upsilon}.  The map $\Upsilon$ will be constructed first along $\partial_1\cY:= (\RP^5\times \partial \D) \cap \Psi^{-1}(\cS_1(S^3))$, and then extended to $\partial_1\cX$ by rotations (see Lemma \ref{Lem_Upsilon:b_1X-to-HopfLinks} and the discussion above it). Before diving into details, let us sketch the main picture here. We will define certain spaces $\bar\tau$ and $ \hConf_4(\tau)$, and some maps between them:
  \begin{equation} \label{Diag_Upsilon-factorthrough-hConf}
    \begin{tikzcd}[cramped,column sep=scriptsize]
       {\partial\cX} & { {\bar\tau}} &  \\
       {\partial_1\cX} & {\hConf_4(\tau)}   & {\cE(L, S^3)}  
       \arrow[from=1-1, to=1-2, "\Theta_\tau"]
       \arrow[from=2-1, to=2-2, "\cM_\tau"]
       \arrow[from=2-2, to=2-3, "\Upsilon'_\tau"]
       \arrow[from=2-1, to=1-1, "\iota"]
       \arrow[from=2-2, to=1-2, " {\cP_\tau}"]
    \end{tikzcd}
  \end{equation}
  All these spaces, except for $\cE(L, S^3)$, are fiber bundles over $B:=SO(4)/D_{24}$, and the maps are all constructed first along individual fibers, and then extended to the whole bundle via rotation by $SO(4)$. The maps will be constructed such that the following is satisfied.
  \begin{itemize}
    \item  {$(\cP_\tau)_*: H_1(\hConf_4(\tau); \Z_2) \to H_1(\bar\tau; \Z_2)$ is injective   (Lemma \ref{Lem_P_tau:hConf-to-tau^4})}.
    \item $\Theta_\tau\circ \iota \simeq \cP_\tau\circ \cM_\tau$ as maps $\partial_1\cX\to  {\bar\tau}$ (Lemma \ref{Lem_Theta:b_1X-to-tau^4}), where $\simeq$ means the two maps concerned are homotopic.
    \item $\Upsilon'_\tau\circ \cM_\tau \simeq \Upsilon$ as maps $\partial_1\cX \to \cE(L, S^3)$ (Lemma \ref{Lem_Upsilon:b_1X-to-HopfLinks}).
     
  \end{itemize}
  Based on these facts, \ref{Item_Upsilon_Ker-iota_*-to-0} of Proposition \ref{Lem_NontrivialLoop_Upsilon} then follows from a formal chase of diagram, which will be done in \S \ref{sect:ProofLem_NontrivialLoop_Upsilon}.
  
  Throughout this section, we let $G\subset SO(4)$ be the dihedral subgroup of order $24$ generated by $g_1, g_2$ defined in section \ref{subsec:symm_gp_lawson_surf} (see \eqref{eqn:g_1_act1}, \eqref{eqn:g_2_act1} for an explicit description of the action on $\R^4$). \begin{nota}\label{nota:fiberBundle}
 If $\eta: G\to \Diff(F)$ is a {\it left} action of $G$ on a manifold $F$, we write $SO(4)\times_\eta F:= (SO(4) \times F)/G$, where the {\it right} $G$-action on $SO(4)\times F$ is given by, 
  \begin{align*}
    (R, v)\cdot g:= (Rg, \eta(g)^{-1}v), \quad \forall\, g\in G \,.
  \end{align*}
  This gives a fiber bundle over $B:=SO(4)/G$ with fiber $F$ with total space $SO(4)\times_\eta F$: This is called the {\it associated bundle} to $SO(4)\to B$ with fiber $F$. 
  We shall use $(R, v)\cdot G$ to denote the element in $SO(4)\times_\eta F$ represented by $(R, v)\in SO(4) \times F$.
\end{nota}

\subsection{Configuration space}  
  We start by introducing some terminology related to the configuration space of $\bS^1:= \R/2\pi\Z$. 
  Define the unordered configuration space of $4$ points in $\bS^1$ by \[
    \UConf_4(\bS^1):= \{Z\subset \bS^1: \# Z=4\} \,.
  \]
  It is better for us to work with the space  $\hConf_4(\bS^1)$, which would be its double cover, that   consists of all triples $(Z; \mu^+, \mu^-)$ with the following properties:
  \begin{itemize}
    \item $Z\in \UConf_4(\bS^1)$, and $\mu^+$ and $ \mu^- \subset \bS^1$ both consist of $2$ points;
    \item $\mu^+\cap \mu^-=\emptyset$, $\mu^+\cup \mu^-$ contains the midpoint of each of the 4 connected components of $\bS^1\setminus Z$;
    \item for $i\in \{\pm\}$, the two points in $\mu^i$ lie on 2 non-adjacent connected components of $\bS^1\setminus Z$.
  \end{itemize}
  Clearly, from the definition, $(Z; \mu^+, \mu^-)\mapsto Z$ gives a double cover $\hConf_4(\bS^1)\to \UConf_4(\bS^1)$. We endow a topology on $\hConf_4(\bS^1)$ by this cover.

  \begin{rmk}\label{rmk:Uconf}
      It is well-known  that the $\UConf_n(\bS^1)$ is homeomorphic to an open $(n-1)$-ball bundle over $S^1$ \cite{morton1967symmetricProduct}, where the projection map $\UConf_n(\bS^1)\to \bS^1$ is given by summing the $n$ points (recall $\bS^1:=\R/2\pi\Z$).  
  \end{rmk}

  \begin{lem}\label{lem:sumMu+AndMu-}
      For any $(Z;\mu^+,\mu^-)\in\hConf_4(\bS^1)$, $\sum \mu^+=\pi+\sum \mu^-$.
  \end{lem}
  Here $\sum A$ denotes the sum of the elements in $A\subset\bS^1$, which is a well-defined element in $\bS^1.$
  \begin{proof}
    It follows from definition that $2\sum\mu^+=2\sum\mu^-$, so $\sum\mu^+=\sum\mu^-$ or $\pi+\sum\mu^-$. Then, using the path-connected of $\hConf_4(\bS^1)$ explained  in Remark \ref{rmk:Uconf} and a continuity argument, it suffices to check $\sum \mu^+=\pi+\sum \mu^-$ for any one choice of $(Z;\mu^+,\mu^-),$ which is a trivial task. 
  \end{proof}

  We shall consider the left $G$-action $\varrho$ on $\bS^1$ generated by 
  \begin{align} 
    \varrho(g_1): \alpha \mapsto \alpha+\pi/3\,, \qquad 
    \varrho(g_2): \alpha \mapsto \pi - \alpha \,.  \label{Equ_D_24-action on S^1}      
  \end{align}
  It induces a $G$-action on $\hConf_4(\bS^1)$ generated by 
  \begin{align}
    \hat\varrho(g_i): (Z; \mu^+, \mu^-) \mapsto \big(\varrho(g_i)(Z); \varrho(g_i)({\color{red}{\mu^-}}), \varrho(g_i)({\color{red}{\mu^+}}) \big), \quad i\in \{1,2\} \,.  \label{Equ_D_24-action on hConf_4}
  \end{align}
  It is easy to check that this is a well-defined $G$-action on $\hConf_4(\bS^1)$. 
  We shall also let  {$\bar\varrho$} be the left $G$-action on $\bS^1$ generated by 
 
  \begin{align} 
    \bar\varrho(g_1): \alpha \mapsto \alpha-\pi/3\,, \qquad 
    \bar\varrho(g_2): \alpha \mapsto \pi - \alpha \,.  \label{Equ_D_24-action barvarrho on S^1}      
  \end{align}
  Then, define the following bundles over $B$:
  \begin{align*}
    \tau:= SO(4)\times_{\varrho} \bS^1\,, \quad
   {\bar\tau:= SO(4)\times_{\bar\varrho} \bS^1}\,, \quad
    \hConf_4(\tau):= SO(4)\times_{\hat\varrho} \hConf_4(\bS^1) \,.
  \end{align*}

  \begin{lem} \label{Lem_P_tau:hConf-to-tau^4}
    The map 
    \begin{align}
      \cP_\circ: \hConf_4(\bS^1) \to \bS^1, \quad (Z; \mu^+, \mu^-) \mapsto \sum  \mu^- \,,  \label{Equ_Def_P:hConf-to-S^1}
    \end{align}
    is $G$-equivarient (note the sum is well-defined modulo $2\pi$), where $G$ acts on the left hand side by $\hat\varrho$, and on the right hand side by  {$\bar\varrho$}. 

    Thus it induces a well-defined map
    \begin{align}
      \cP_\tau: \hConf_4(\tau) \to \bar\tau, \quad (R, (Z;\mu^+, \mu^-))\cdot G \mapsto (R, \cP_\circ(Z; \mu^+, \mu^-))\cdot G\,.
    \end{align}
    Moreover, $(\cP_\tau)_*: H_1(\hConf_4(\tau), \Z_2) \to H_1(\bar\tau, \Z_2)$ is injective.
  \end{lem}
  \begin{proof}
    The $G$-equivarience of $\cP_\circ$ follows by a direct calculation: Note that
    \begin{align*}
    \cP_\circ\hat\varrho(g_1)(Z;\mu^+,\mu^-)&=\cP_\circ(Z+\pi/3;\mu^-+\pi/3,\mu^++\pi/3)=2\pi/3+\sum\mu^+,\\ 
    \bar\varrho(g_1)\cP_\circ(Z;\mu^+,\mu^-)&=\bar\varrho(g_1)(\sum \mu^-)=-\pi/3+\sum\mu^-.
    \end{align*}
    By Remark \ref{rmk:Uconf}, these two expressions are the same. For $g_2$, the calculation is analogous. Therefore the map $\cP_\tau$ is well-defined, and we are left to show the injectivity of $(\cP_\tau)_*$. 

    Here and below, we work with $\Z_2$-coefficients. First, consider  the   diagram.
    \begin{equation} \label{Diag_BundleMap}
      \begin{tikzcd}[cramped,column sep=scriptsize]
        {H_1(\hConf_4(\bS^1))} & {H_1(\hConf_4(\tau))} & {H_1(B)}  \\
        {H_1(\bS^1)} & {H_1( {\bar\tau})}   & {H_1(B)}  
        \arrow[r,phantom,"\xhookrightarrow{
        \hat i_*}" description, from=1-1, to=1-2]
        \arrow[from=1-2, to=1-3, "\hat\pi_*"]
        \arrow[r,phantom,"\xhookrightarrow{i_*}" description, from=2-1, to=2-2]
        \arrow[from=2-2, to=2-3, "\pi_*"]
        \arrow[from=1-1, to=2-1, "(\cP_\circ)_*"]
        \arrow[from=1-2, to=2-2, "(\cP_\tau)_*"]
        \arrow[from=1-3, to=2-3, "\id"]
      \end{tikzcd}
    \end{equation}
    Note, by construction we see that $\hConf_4(\tau)$ is a  $\hConf_4(\bS^1)$-bundle over $B$, so the map $\hat i$ just denotes an inclusion map into some fixed fiber while $\hat \pi$ the projection map, and similarly for $i$ and $\pi$. By the $G$-equivariance of $\cP_\circ$, it is easy to see that the above diagram Moreover, we claim that 
    \begin{enumerate} [label={\normalfont(\alph*)}]
     \item\label{Item_SpecSeq_P0_*inj} $\hConf_4(\bS^1)$ is path-connected, and $(\cP_\circ)_*: H_1(\hConf_4(\bS^1))\to H_1(\bS^1)$ is   {bijective}.
     \item\label{Item_SpecSeq_i_*inj} $i_*: H_1(\bS^1)\to H_1(\bar\tau)$ is injective;
    \end{enumerate}

    Let us first finish the proof assuming these two claims. Since $B$ is path-connected, the following Lemma \ref{Lem_SpectralSequ} yields that the horizontal sequences are exact. Therefore, if $a\in \Ker (\cP_\tau)_*$, then by $\hat\pi_* a = 0$, there exists $v\in H_1(\hConf_4(\bS^1))$ such that $\hat i_* v = a$.  Since $i_*\circ (\cP_\circ)_* v = (\cP_\tau)_*a = 0$, by the injectivity assertions in the two claims we see that $v=0$ and hence $a=\hat i_* v=0$. This finishes the proof of the lemma.

     \begin{lem} \label{Lem_SpectralSequ}
    Let $F \hookrightarrow E \xrightarrow{\pi} B$ be a fiber bundle such that both $B$ and $F$ are path-connected. Then the inclusion and the projection induce an exact sequence \[
      H_1(F, \Z_2) \to H_1(E, \Z_2) \to H_1(B, \Z_2) \to 0 \,.
    \]
  \end{lem}
  \begin{proof}
    Consider the Serre spectral sequence with $\Z_2$-coefficents for the fibration. By the five-term exact sequence \cite[Corollary~9.14]{JamesPaul2001AlgebraicTopology}, we have the exact sequence:
    \[
        H_0(B; H_1(F)) \to H_1(E) \xrightarrow{\pi_*} H_1(B) \to 0\,. 
    \]
    Moreover, denoting by $F_b$ the fiber at any $b \in B$, 
    \[
        H_0(B; H_1(F)) \cong H_1(F_b)_{\pi_1(B)} := H_1(F_b) /\langle g \cdot x - x | g \in \pi_1(B, b), x \in H_1(F_b)\rangle\,;
    \]
    and the composition of the surjection $H_1(F_b) \to H_0(B; H_1(F))$ with the map $H_0(B; H_1(F)) \to H_1(E)$ is the homomorphism induced by the inclusion $F_b \hookrightarrow E$. This proves the lemma.   
  \end{proof}

Now let us prove the two claims.
    \begin{proof}[Proof of \ref{Item_SpecSeq_P0_*inj}.] The path-connectedness follows easily from  Remark \ref{rmk:Uconf}. Let $ \gamma\subset\hConf_4(\bS^1)$ be the loop that consists of all possible triples $(Z;\mu^+,\mu^-)$ where $Z$ is formed by 4 equidistant points on $\bS^1$. Note, by the discussion in Remark \ref{rmk:Uconf}, it follows easily  that $\hConf_4(\bS^1)$ can be deformation retracted onto $\gamma$. On the other hand, it is straightforward to check from definition that $\cP_\circ\circ\gamma$ generates  $H_1(\bS^1)$. Therefore $(\cP_\circ)_*$ is bijective, as desired.

    \end{proof}
    \begin{proof}[Proof of \ref{Item_SpecSeq_i_*inj}.]
      Consider the subgroup $H$ of $\Diff((\mathbb S^3\times \mathbb S^3)\times\R)$ generated by 
      \begin{align*}
          h_0: (\bq, \tilde\alpha) \mapsto (\bq, \tilde\alpha+2\pi)\,, \quad
          h_1: (\bq, \tilde\alpha) \mapsto (\bq \hat g_1^{-1}, \tilde\alpha-\pi/3)\,, \quad
          h_2: (\bq, \tilde\alpha) \mapsto (\bq \hat g_2^{-1}, \pi - \tilde\alpha) \,.
      \end{align*}
      (Recall the definition of $\hat g_i$ in \S \ref{subsec:symm_gp_lawson_surf}.)  This yields a natural right action of $H$ on $(\mathbb S^3\times \mathbb S^3)\times\R$, and clearly, \[
        \bar\tau = SO(4)\times_{\bar\varrho} \bS^1 := (SO(4)\times \bS^1)/(R, v)\sim (Rg, \bar\varrho(g)^{-1}v) = (\mathbb S^3\times \mathbb S^3\times\R)/H \,.
      \]
      We will check that $H$ acts freely on $(\mathbb S^3\times \mathbb S^3)\times\R$ and that the abelianization $H_{ab}=\Z_2\oplus \Z_2\oplus \Z_2$ with generators induced by $h_0, h_1, h_2$. Once we have this, we know  $\pi_1(\bar\tau)=H$, and since $h_0$ clearly represents a loop in a fiber $\bS^1$ of $\tau$ which generates $\pi_1(\bS^1)$, we would obtain the injectivity of $i_*:H_1(\bS^1, \Z_2)\to H_1(\bar\tau, \Z_2)$. 
      First, from  definitions, its easy to check that 
      \begin{align*}
        h_1^{24}h_0^4=\id\,, \quad
        h_1^{12}h_0^2 = h_2^{2}\,, \quad
        h_1h_2 = h_2h_1^{-1}\,, \quad
        h_0h_1 = h_1h_0\,, \quad
        h_0h_2 = h_2h_0^{-1} \,.
      \end{align*}
      These are all the relations that determine $H$ as the quotient of the free group generated by $h_0, h_1, h_2$, since any other relation can be reduced by switching $h_0, h_1, h_2$ and cancellation to $h_0^ah_1^bh_2^c=1$ for some $a\in \Z$, $0\leq b\leq 23$, $c\in \{0,1\}$. But one can directly check from the definition of $h_i$ that this forces $a=b=c=0$. 
      Now from the relation, its clear that the map \[
        H_{ab} \to \Z_2\oplus\Z_2\oplus\Z_2\,, \quad 
        [h_0]\mapsto (1,0,0),\ \
        [h_1]\mapsto (0,1,0),\ \
        [h_2]\mapsto (0,0,1)\,,
      \] 
      is a well-defined isomorphism. 
      
      Also, a general element in $H$ takes the form $h_0^ah_1^bh_2^c$ for some $a\in \Z$, $0\leq b\leq 23$, $c\in \{0,1\}$. If it has a fix point, then $\hat g_1^b\hat g_2^c$ has a fixed point on $\mathbb S^3\times \mathbb S^3$, which forces $b=c=0$; and then $h_0^a$ has a fixed point, forcing $a=0$. Thus, $H$ acts freely on $(\mathbb S^3\times \mathbb S^3)\times\R$.
    \end{proof}
    This completes the proof of Lemma \ref{Lem_P_tau:hConf-to-tau^4}.
  \end{proof}

  \subsection{Hopf links through midpoints I: the map $\Upsilon'_\tau$}

  We are going to define  a map $\check\Upsilon'_\tau$ from  $\hConf_4(\bS^1)$ to the space $\check\cE(L, S^3)$ of piecewise smooth Hopf links, and then smoothen it to become a map $\Upsilon'_\tau$ into $\cE(L, S^3)$. We will continue to use the parametrization of $S^3$ by \eqref{Equ_Constru Psi_Param S^3}. Notice the $G$-action on $S^3$ induces a $G$-action on $\cE(L,S^3)$ by rotating the Hopf links, and recall $G$ acts on $\hConf_4(\bS^1)$ by \eqref{Equ_D_24-action on hConf_4}.

  \begin{lem} \label{Lem_Upsilon'_tau:hConf-to-HopfLinks}
    The following gives  a well-defined $G$-equivarient map: \begin{align}
     \begin{split}
      \check\Upsilon'_\circ: \hConf_4(\bS^1) & \to \check\cE(L, S^3)\,, \\
      (Z; \mu^+, \mu^-) & \mapsto \{(t,t,\alpha): |t|\leq 1, \alpha\in \mu^+\} \cup \{(t,-t,\alpha): |t|\leq 1, \alpha\in \mu^-\} \,
     \end{split} \label{Equ_Def_Upsilon_0:hConf-to-HopfLinks}
    \end{align}
    (see Figure \ref{fig:link}). 
    As a corollary, it induces a well-defined map
    \begin{align}
      \check\Upsilon'_\tau: \hConf_4(\tau) \to \check \cE(L, S^3), \quad (R, (Z;\mu^+, \mu^-))\cdot G \mapsto R\big(\check\Upsilon'_\circ(Z; \mu^+, \mu^-)\big)\,.
    \end{align}
  \end{lem}
\begin{figure}
    \centering
\includegraphics[width=1.5in]{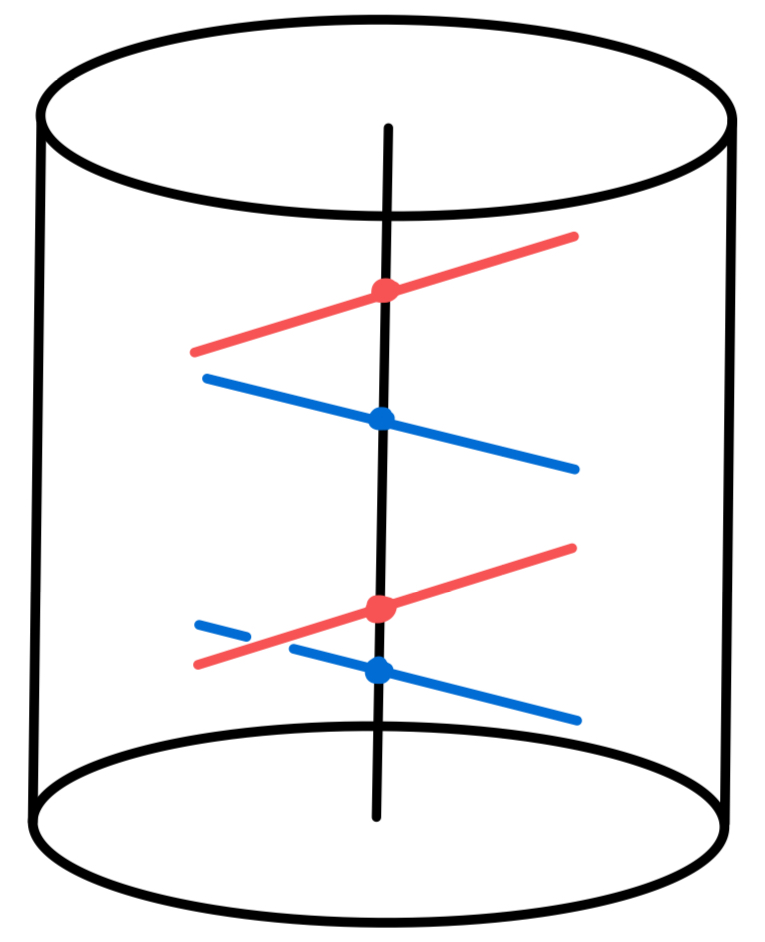}
\caption{The cylinder, after suitable boundary identification, is $S^3$. Then the red pieces form a loop, and so do the blue pieces. Together they form a link of the form $\Upsilon'_\circ(Z;\mu^+,\mu^-).$ }
\label{fig:link}
\end{figure}
  \begin{proof}
 By the definition of $\hat\varrho$ and \eqref{Equ_Def_Upsilon_0:hConf-to-HopfLinks}, 
   $$\check\Upsilon'_\circ(\hat\varrho(g_1)(Z;\mu^+,\mu^-))=\{(t,t,\alpha): |t|\leq 1, \alpha\in \varrho(g_1)(\mu^-)\} \cup \{(t,-t,\alpha): |t|\leq 1, \alpha\in  \varrho(g_1)(\mu^+)\},$$
   while using \eqref{eqn:g_act2} we have $$g_1(\check\Upsilon'_\circ(Z;\mu^+,\mu^-))=\{(-t,t,\alpha+\pi/3): |t|\leq 1, \alpha\in  \mu^+\} \cup \{(t,t,\alpha+\pi/3): |t|\leq 1, \alpha\in  \varrho(g_1)(\mu^-)\},$$
   and it is easy to see these two expressions are the same. For the action by $g_2$, the verification is similar.
  \end{proof}
  
Finally, we smoothen all the Hopf links defined by \eqref{Equ_Def_Upsilon_0:hConf-to-HopfLinks} in a $G$-equivarient way to obtain a map $\Upsilon'_\tau:\hConf_4(\tau)\to \cE(L,S^3).$

\subsection{Configuration determined by zeros}
  Recall the $7$-parameter family $\Psi: \RP^5\times \D\to \cS_{\leq 2}(S^3)$   specified in \S \ref{sect:2g+3}. Denote $\cY:=\RP^5\times \D$. Let  $\partial_1\cY\subset\partial \cY$ be the subset of $y$ such that $\fg(\Psi(y))=1$.
  Recall that there is a left $G$-action on $\cY$ defined in \eqref{Equ_D_24-action_Y} which makes $\Psi$ $G$-equivarient (see Lemma \ref{Lem_D_24Sym_F_a,z(p)}). Note, the sets $\partial\cY,\ \partial_1\cY$ are both preserved  under this action.
    

Now, recall from  \S \ref{sect:tildeXiInvariantProof} that for each $(a,z)=(a,re^{i\theta})\in \cY$ in $\{a_5\ne 0\}\x\D\subset \RP^5\x\D$, we considered the real function $F_{a,z}$ on $S^3$ (parametrized by $(x,\alpha)$) given by 
\begin{align}
   \begin{split}
    F_{a,re^{i\theta}}(x, \alpha) := a_5^{-1} \Big( a_0 & + a_1x_1+a_2x_2+a_3x_3+a_4x_4\\
        & +a_5\left[x_1x_2+ \rho(a,x)\left(r\cos(\theta+2\alpha)+(1-r)\cos3\alpha\right)\right]\Big),   \label{Equ_Def_F_a,z(x)}
   \end{split}
  \end{align}
and the zero set of $F_{a,z}$ defines $\Psi(a, z)$ under \eqref{eq:PsiDef}. Note $\partial_1\cY\subset \{a_5\ne 0\}\x\partial\D$. 
In Lemma \ref{Lem_SmallDelta} ,  
for each $(a,z)\in \cY$ such that function  $\rho(a,\cdot)$ is not constantly zero, we obtained a map ${\bf x}_a\in C^\infty(\bS^1,\bD)$ with the following properties:
     \begin{itemize}
         \item          \begin{equation}\label{sect:bx}
          \nabla_x F_{a,z}(x, \alpha) = 0 \quad \Leftrightarrow \quad  x = \bx_a(\alpha).
          \end{equation}
         \item Consider the map
  \begin{align} 
    f: \cY\cap \{\rho(a,\cdot)\not\equiv 0\} \to C^\infty(\bS^1), \quad (a, z)\mapsto \big(f_{a,z}(\alpha):= F_{a,z}(\bx_a(\alpha), \alpha) \big)\,. \label{Equ_Def_f_a,z}
  \end{align}
  Then, from Lemma \ref{Lem_SmallDelta} \ref{Item_Smalldelta_GenusBd}, 
  \[
    \partial_1\cY = \{(a, z)\in \partial\cY: \rho(a, \cdot)\not\equiv 0,\ \# f_{a,z}^{-1}(0) = 4\} \,.
  \]
     \end{itemize}

  Notice  $\partial_1\cY$ is an open subset of $\partial\cY$.
  This enables us to define a map
  \begin{align}
    \cM_\circ: \partial_1\cY \to \hConf_4(\bS^1), \quad 
    (a, z) \mapsto \left(f_{a,z}^{-1}(0); \mu(\{f_{a,z}>0\}), \mu(\{f_{a,z}<0\}) \right)\,,  \label{Equ_Def_Take Roots,Midpts}
  \end{align}
  where for a disjoint union $J$ of open intervals in $\bS^1$, $\mu(J)$ denotes the set of midpoints of the  connected components of $J$.
  \begin{lem} \label{Lem_cM:Y_1-to-hConf_4-D_24-equiva}
    The map $\cM_\circ$   is $G$-equivariant, where $G$ acts on $\hConf_4(\bS^1)$ by $\hat\varrho$ described in \eqref{Equ_D_24-action on hConf_4}, and acts on $\partial_1\cY$ by $\sigma$ defined in \eqref{Equ_D_24-action_Y}. As a corollary, it induces a well-defined map from $\partial_1\cX:=SO(4)\times_\sigma \partial_1\cY$ into $\hConf_4(\tau)$:
    \begin{align}
      \cM_\tau: \partial_1\cX \to \hConf_4(\tau), \quad (R, (a,z))\cdot G \mapsto (R, \cM_\circ(a,z))\cdot G \,.
    \end{align}
  \end{lem}
  \begin{proof} By  Lemma \ref{Lem_D_24Sym_F_a,z(p)}, for all $(a,z)\in\partial_1\cY$, for $i=1,2$, we have $F_{\sigma(g_i)(a,z)}(p) = -F_{a,z}(g_i^{-1}\cdot p)$  for every $p\in S^3$. Differentiate this equation with respect to the $(x_1,x_2)$. Then using  the definition \eqref{sect:bx} of ${\bf x}_a$ and the action  \eqref{eqn:g_act2} on $S^3$, we can check that \begin{equation}\label{eq:D24_action_xa}
   {\bf x}_{g_1a}(\alpha)=g_1{\bf x}_a(\alpha-\pi/3),\quad {\bf x}_{g_2a}(\alpha)=g_2{\bf x}_a(\pi-\alpha).   
  \end{equation} 
  Note here we abused notation: $g_ia$ denotes the $\RP^5$ component in \eqref{Equ_D_24-action_Y}, while $g_i(x_1,x_2)$ denotes the first two components in $g_i(x_1,x_2,\alpha)$ under \eqref{eqn:g_act2}. Thus, putting this into the definition \eqref{Equ_Def_f_a,z} of $f_{a,z}$, we see that 
  \begin{equation}\label{eq:D24_action_faz}
  f_{\sigma(g_1)(a,z)}=-f_{a,z}(\alpha-\pi/3),\quad f_{\sigma(g_2)(a,z)}=-f_{a,z}(\pi-\alpha).    
  \end{equation}
   This implies the zero set $f_{\sigma(g_i)(a,z)}^{-1}(0)$ is the image of $f_{ (a,z)}^{-1}(0)$ under $g_i$ under the action \eqref{Equ_D_24-action on S^1}, and that 
  $$\cM_\circ(\sigma(g_i)(a,z))=\left(\varrho(g_i)(f^{-1}_{a,z}(0));\varrho(g_i)(\mu(\{f_{a,z}{\color{red}<}0\})),\varrho(g_i)(\mu(\{f_{a,z}{\color{red}>}0\}))\right).$$
  But then according to \eqref{Equ_D_24-action on hConf_4} this is exactly $
    \hat\varrho(g_i)(f^{-1}_{a,z}(0);\mu(\{f_{a,z}>0\}),\mu(\{f_{a,z}<0\})).$
  \end{proof}

{
  Another important fact about $f_{a,z}$ is the following consequence of Lemma \ref{Lem_SmallDelta} \ref{Item_Smalldelta_ZerosBd}.   
  \begin{lem} \label{Lem_Theta:b_1X-to-tau^4}
    The  map $\Theta_\circ: \partial\cY \to \bS^1$ given by $ \Theta_\circ(a, e^{i\theta}) \mapsto -\theta$
    is $G$-equivarient, where $G$ acts on $\partial\cY$ by $\sigma$ and on $\bS^1$ by $\bar\varrho$. 
    Thus, it induces a well-defined map 
    \begin{align}
      \Theta_\tau: \partial\cX \to \bar\tau\,, \quad (R, (a,z))\cdot G \mapsto (R, \Theta_\circ(a,z))\cdot G \,.  \label{Equ_Theta_tau:b_+X-to-tau^4}
    \end{align}    
    Moreover, 
    \begin{align}
      \dist_{\bS^1}(\cP_\circ \circ \cM_\circ(a, e^{i\theta}), -\theta) \leq 1000^{-1} , \quad \forall\, (a, e^{i\theta})\in \partial_1\cY\,.\label{eq:1000^-1}
    \end{align}
    and hence the maps $\Theta_\tau\circ\iota$ and $\cP_\tau\circ\cM_\tau:\ \partial_1\cX\to \bar\tau$ are homotopic.
  \end{lem}
}
  \begin{proof}
For the $G$-equivariance, by definition (and abusing notation by writing $g_ia$) 
$$\Theta_\circ(g_1(a,(r,\theta)))=\Theta_\circ(g_1a,(r,\theta+\pi/3))= {-(\theta+\pi/3) = \bar\varrho(g_1)(-\theta)},$$
and 
$$\Theta_\circ(g_2(a,(r,\theta)))=\Theta_\circ(g_2a,(r,\pi-\theta))=  {-(\pi-\theta)\sim \pi-(-\theta)=\bar\varrho(g_2)(-\theta)},$$
as desired. 

    For any $(a,z)\in\partial_1\cY$ (so that $\Psi(a,z)$  has  genus 1), we have $\# f^{-1}_{a,z}(0)=4$ by Lemma \ref{Lem_SmallDelta} \ref{Item_Smalldelta_GenusBd}. Therefore by Lemma \ref{Lem_Constru Psi_GenusCal of (F=0)} \eqref{Item_GenusCalc_concl_genus}, every zero of $f_{a,z}$ is of order 1. Hence, the expression $\ord_{f_{a,z}}(\alpha)-\Neg_{f_{a,z}}(\alpha)$ must be zero in the definition  \eqref{Equ_Def_cZ^-} of $\cZ^-$. In other words, $\cZ^-(f_{a,z})=\sum _{\alpha\in \mu(\{f_{a,z}<0\})}\alpha$, which is equal to $\cP_\circ\circ\cM_\circ(a,z)$ by definitions  \eqref{Equ_Def_P:hConf-to-S^1} and \eqref{Equ_Def_Take Roots,Midpts}. Together with Lemma \ref{Lem_SmallDelta} \ref{Item_Smalldelta_ZerosBd}, we obtain \eqref{eq:1000^-1}, and the fact $\Theta_\tau\circ\iota\simeq\cP_\tau\circ\cM_\tau$   follows as the two maps are pointwise close by \eqref{eq:1000^-1}. 
  \end{proof}

\subsection{Hopf links through midpoints II: the map $\Upsilon$}
  Finally, we define a map $\Upsilon: \partial_1\cX\to \cE(L, S^3)$.  We will first construct a map $\check\Upsilon:\partial_1\cX\to\check\cE(L,S^3)$, and then smoothen the Hopf links to obtain $\Upsilon$. By \eqref{Item_GenusCalc_concl_innerLoop} of Lemma \ref{Lem_Constru Psi_GenusCal of (F=0)},
  \begin{align}\label{eq:betaBreve}
    \beta_{a,z}^+ := \left\{ \big(\bx_a(\alpha)+(t, t),\ \alpha\big): t\in \R,\ \alpha\in \mu(\{f_{a,z}>0\})\right\} \cap (\bD\times \bS^1)
  \end{align}
  is the union of two line segments contained in $\{F_{a,z}>0\}$. The picture here is similar to Figure \ref{fig:link}, except that now the four end points of $\beta_{a,z}^+$ may be distinct. However, the four end points naturally come in two pairs, in which each pair consists of two points close to each other. Then for each pair, we bridge the two points by  adding a segment on  the great circle $C\subset S^3$ (which is the image of $\partial \bD\times \bS^1$ under the parametrization \eqref{Equ_Constru Psi_Param S^3}), and form a loop   $\bar\beta_{a,z}^+\subset \{F_{a,z}>0\}\subset S^3$. Similarly, 
  \begin{align*}
    \beta_{a,z}^- := \left\{ \big(\bx_a(\alpha)+(t, -t),\alpha\big): t\in \R,\ \alpha\in \mu(\{f_{a,z}<0\})\right\} \cap (\bD\times \bS^1)  
  \end{align*}
  also extends to a loop $\bar\beta_{a,z}^-$ in $\{F_{a,z}<0\}\subset S^3$. As a result $\Psi(a,z)$ is unknotted and bounds the Hopf link  $\bar\beta^+_{a,z}\cup \bar\beta^-_{a,z}$ (we  introduced this notion before Proposition \ref{prop:wholeFamilyUnknotted}).

\begin{lem}
Under the notations above, the  map 
      $\check\Upsilon_\circ: \partial_1\cY \to \check\cE(L, S^3)$ given by $
      (a, z) \mapsto \bar\beta_{a,z}^+ \cup \bar\beta_{a,z}^-$
    is $G$-equivarient. Thus, it induces a well-defined map 
    \begin{align}
      \check\Upsilon: \partial_1\cX \to \cE(L, S^3)\,, (R, (a,z))\cdot G \mapsto R\big( \check\Upsilon_\circ(a,z) \big)\,.  \label{Equ_Upsilon:b_1X-to-HopfLinks}
    \end{align}
\end{lem}
\begin{proof}
 For the action by $g_1$, by \eqref{eq:D24_action_xa} and \eqref{eq:D24_action_faz} we have
  \begin{align*}
  \beta_{g_1(a,z)}^+&=\left\{ \big(\bx_{g_1(a)}(\alpha)+(t, t),\ \alpha\big): t\in \R,\ \alpha\in \mu(\{f_{g_1(a,z)}>0\})\right\} \cap (\bD\times \bS^1)  \\
  &=\left\{ \big(g_1{\bf x}_a(\alpha-\pi/3)+(t, t),\ \alpha\big): t\in \R,\ \alpha\in \mu(\{\alpha:-f_{a,z}(\alpha-\pi/3)>0\})\right\} \cap (\bD\times \bS^1)\\
  &=g_1\left\{ \big({\bf x}_a(\alpha)+(t, -t),\ \alpha\big): t\in \R,\ \alpha\in \mu(\{\alpha:f_{a,z}(\alpha)<0\})\right\} \cap (\bD\times \bS^1)
  \end{align*}
  and similarly
    \begin{align*}
  \beta_{g_1(a,z)}^-&=\left\{ \big(\bx_{g_1(a)}(\alpha)+(t, -t),\ \alpha\big): t\in \R,\ \alpha\in \mu(\{f_{g_1(a,z)}<0\})\right\} \cap (\bD\times \bS^1)  \\ 
  &=g_1\left\{ \big({\bf x}_a(\alpha)+(-t, -t),\ \alpha\big): t\in \R,\ \alpha\in \mu(\{\alpha:f_{a,z}(\alpha)>0\})\right\} \cap (\bD\times \bS^1).
  \end{align*}
As a result, the set $ \beta_{g_1(a,z)}^+\cup \beta_{g_1(a,z)}^-$ is equal to $g_1(\beta_{ a,z}^+\cup \beta_{ a,z}^-)$, so that $ \bar\beta_{g_1(a,z)}^+\cup \bar\beta_{g_1(a,z)}^-$ is equal to $g_1(\bar\beta_{ a,z}^+\cup \bar\beta_{ a,z}^-)$. The calculation for the action by $g_2$ is similar, so the $G$-equivariance follows.    
\end{proof}

Hence, we smoothen the Hopf links given by $\check\Upsilon$ in a $G$-equivarient way, to  define a map $\Upsilon:\partial_1\cX\to \cE(L,S^3)$. 
 
  \begin{lem} \label{Lem_Upsilon:b_1X-to-HopfLinks} Regarding the map $\Upsilon$, we can guarantee that:
    \begin{enumerate} [label={\normalfont(\roman*)}]
      \item\label{item:UpsilonBoundedByXi} for every $q\in \partial_1\cX$, $\Upsilon(q)$ is a Hopf link bounded by $\Xi(q)$;
      \item\label{item:UpsilonD0} $\Upsilon_*[\partial D_0]\neq 0$ in $H_1(\cE(L, S^3), \Z_2)$;
      \item\label{itme:Upsilon'MtauHomotopyEquiv} $\Upsilon'_\tau\circ \cM_\tau$ and $\Upsilon: \partial_1\cX\to \cE(L, S^3)$ are homotopic.
    \end{enumerate}
  \end{lem}
  \begin{proof}

As for the second paragraph, item \ref{item:UpsilonBoundedByXi} follows from the fact that $\bar\beta^+_{a,z}\cup \bar\beta^-_{a,z}$ is a Hopf link bounded by $\Psi(a,z)$, which is unknotted. To prove \ref{item:UpsilonD0}, it suffices to show that the loop $\Upsilon_\circ|_{\partial D_0}:\partial D_0\to\cE(L,S^3)$ is homologically trivial in $\Z_2$-coefficients. First, we label the two loops in $\Upsilon_\circ|_{\partial D_0}$ as $\beta_+$ and $\beta_-$, and orient them both. Now, viewing $\partial D_0$ as $[0,2\pi]$ with end points identified, we can obtain via $\Upsilon_\circ|_{\partial D_0}$ a $[0,2\pi]$-family of labeled, oriented Hopf links, $\beta_+(t)\cup\beta_-(t)$. Due to the explicit description of the topology of  $\cE(L,S^3)$   by  Boyd-Bregman \cite{BoydBregman2025_HopfLink} (see Lemma \ref{Lem_Hopflink_Topology}), to prove that $[\Upsilon_\circ|_{\partial D_0}]=0$ in $H_1(\cE(L,S^3);\Z_2)$, it suffices to show that $\beta_+(0)$ and $\beta_+(2\pi)$ are the same loop  {\it with opposite orientations}.

To show that $\beta_+(0)$ and $\beta_+(2\pi)$ have opposite orientations, we first recall that the family $\Psi|_{\partial D_0}$ is given by the zero sets 
$$\{x_1x_2+\rho([0:\dots:0:1],(1,\theta))\cos(\theta+2\alpha)=0\}\subset S^3,$$
where $\theta$ parametrizes $\partial D_0$. Now, this is just an isometric family of smooth tori in $S^3$. Let us fix a continuous choice of inside region for each torus such that they contain the loops $\beta_+(t)$. It can be easily checked that as $\theta$ varies from $0$ to $2\pi$, the inside  regions of $\Psi([0:...:0:1],(1,0))$ and $\Psi([0:...:0:1],(1,2\pi))$
(both are solid tori) are not exchanged, but the associated isometry map between the starting solid torus and the ending one in fact induces the $-\id$ map on $\pi_1$.   As a result, considering the core loops of the solid tori, we see that $\beta_+(0)$ and $\beta_+(2\pi)$ necessarily have different orientations, as desired.

Finally, for item \ref{itme:Upsilon'MtauHomotopyEquiv}, it suffices to construct a homotopy $H:[0,1]\x \partial_1\cY\to \check \cE(L,S^3)$ such that $H(0,\cdot)=\check\Upsilon'_\circ\circ\cM_\circ$, $H(1,\cdot)=\check\Upsilon_\circ$, and $H$ is $G$-equivariant: After that we can apply smoothening to conclude $\Upsilon'_\tau\circ \cM_\tau\simeq\Upsilon$. To this end, for each $s\in [0,1]$, we define 
    \begin{align*}
    \beta_{a,z}^+(s) &:= \left\{ \big(s\bx_a(\alpha)+(t, t),\ \alpha\big): t\in \R,\ \alpha\in \mu(\{f_{a,z}>0\})\right\} \cap (\bD\times \bS^1)\\
    \beta_{a,z}^-(s) &:= \left\{ \big(s\bx_a(\alpha)+(t, -t),\alpha\big): t\in \R,\ \alpha\in \mu(\{f_{a,z}<0\})\right\} \cap (\bD\times \bS^1),  
  \end{align*}
and then, as in the paragraph containing \eqref{eq:betaBreve}, close up $\beta_{a,z}^+(s)\cap \beta_{a,z}^-(s)$ to form a Hopf link $\bar\beta_{a,z}^+(s)\cap \bar\beta_{a,z}^-(s)$. Now, define $H(s,(a,z))$ as this Hopf link. It is clear  by definition that $H(0,\cdot)=\check\Upsilon'_\circ\circ\cM_\circ$, $H(1,\cdot)=\check\Upsilon_\circ$, and similar to the above we can check that $H$ is indeed $G$-equivariant, as desired.
\end{proof}

\subsection{Proof of Proposition \ref{Lem_NontrivialLoop_Upsilon}.}\label{sect:ProofLem_NontrivialLoop_Upsilon}
{
   Item \ref{Item_Upsilon_BdedHopfLinks} and \ref{Item_Upsilon_nontrivial-bD_0} were stated in Lemma \ref{Lem_Upsilon:b_1X-to-HopfLinks} already. 
   As for \ref{Item_Upsilon_Ker-iota_*-to-0}, let $\vartheta\in H_1(\partial_1\cX, \Z_2)$ be such that $\iota_*\vartheta=0$ in $H_1(\partial\cX;\Z_2)$. 
   Since $\Theta_\tau\circ \iota \simeq \cP_\tau\circ \cM_\tau$ by Lemma \ref{Lem_Theta:b_1X-to-tau^4}, we know $(\cP_\tau\circ \cM_\tau)_*\vartheta=0$.    
   Hence, since $(\cP_\tau)_*: H_1(\hConf_4(\tau); \Z_2) \to H_1(\bar\tau; \Z_2)$ is injective by Lemma \ref{Lem_P_tau:hConf-to-tau^4}, we in fact know $(\cM_\tau)_*\vartheta=0$ in $H_1(\hConf_4(\tau); \Z_2)$. Finally, since $\Upsilon'_\tau\circ \cM_\tau \simeq \Upsilon$ by Lemma \ref{Lem_Upsilon:b_1X-to-HopfLinks}, we know that $\Upsilon_* \vartheta = 0$ in $H_1(\cE(L, S^3), \Z_2)$, as desired.
}

\appendix

\section{Metric perturbation}\label{sect:proofPerturbMetric} In this section we prove  Proposition~\ref{prop:perturbMetric}. 
    Let $(M, \bg)$ be as in Proposition~\ref{prop:perturbMetric}. Fix an integer $g_1 > g_0$.  For each metric $\bg'$ on $M$, let $\tilde \cN_{\leq g_1, \leq L}(M, \bg')$ denote the set of immersed minimal surfaces obtained as the double cover of those in $\cN_{\leq g_1, \leq L}(M, \bg')$. And we extend~\eqref{eqn:def_K_g'} to
    \[
      \tilde \cK_{\bg'}:= \left(\cO_{\leq g_1, \leq L}(M, \bg') \cup \cN_{\leq g_1, \leq L}(M, \bg') \cup \tilde \cN_{\leq g_1, \leq L}(M, \bg')\right) \setminus \cO_{g_0,\leq L}(M,\bg')\,.
    \]

    \begin{lem}\label{lem:other_min_surf_far}
        There exist $\varepsilon > 0$, and some open set $U\subset S^3$ of the form $U = \bigcup_{\Sigma \in \cO_{g_0, \leq L}(M, \bg)} B^{\bg}_\delta(\Sigma)$  for some $\delta>0$, such that for any metric $\bg'$ with $\|\bg' - \bg\|_{C^\infty} < \varepsilon$, no $\bg'$-minimal surface in $\cK_{\bg'}$ (defined in~\eqref{eqn:def_K_g'}) is contained in $U$.
    \end{lem}
    \begin{proof}
        Suppose by contradiction that such $\epsilon$ and $U$ do  not exist. Then, by~\cite[Theorem~3]{Whi87}, there exist $\Sigma \in \cO_{g_0, \leq L}(M, \bg)$, a sequence $\bg_j \to \bg$ in $C^\infty$, and a sequence  $\{\Sigma_j \in \cK_{\bg_j}\}$ such that in the varifold sense, $|\Sigma_j| \to k |\Sigma|$. Moreover, there is a finite set $Z$ such that $\Sigma_j$ converges to $\Sigma$ multi-graphically in any compact subset of $M \setminus Z$. In addition, we may assume that $\Sigma_j \in \tilde \cK_{\bg_j}$, since each $\tilde \Sigma_j \in \tilde \cN_{\bg_j}$

        If $k = 1$, by Allard's regularity~\cite{All72}, $Z = \emptyset$ and for sufficiently large $j$, $\Sigma_j$ has the same diffeomorphic type as $\Sigma$, contradicting the definition of $\tilde \cK_{\bg_i}$. Otherwise, $k \geq 2$, and by~\cite[Claim~5~and~Claim~6]{Sha17}, $\Sigma$ is degenerate stable, contradicting the assumption on $(M, \bg)$.
    \end{proof}

    Clearly, the lemma also holds with $\cK_{\bg'}$ replaced by $\tilde \cK_{\bg'}$.

    \begin{lem}
        Fix $\varepsilon$ and $U$ from Lemma~\ref{lem:other_min_surf_far}. There exists $\varepsilon' \in (0, \varepsilon)$, such that for any metric $\bg'$ with $\|\bg' - \bg\|_{C^\infty} < \varepsilon$, if $\bg'|_U = \bg|_U$, then $\cO_{g_0, \leq L}(M, \bg') = \cO_{g_0, \leq L}(M, \bg)$.
    \end{lem}
    \begin{proof}
        Suppose by contradiction that such an $\varepsilon'$ does not exist, and then by~\cite[Theorem~3]{Whi87} again, there exist $\Sigma \in \cK_\bg$, a sequence $\bg_j \to \bg$ in $C^\infty$ with $\bg'|_U = \bg|_U$, and a sequence $\{\Sigma_j \in \cO_{g_0, \leq L}(M, \bg_j)\}$, such that $|\Sigma_j| \to k|\Sigma|$ in the varifold sense.

        Following the same argument in the previous proof, if $k = 1$, then for sufficiently large $j$, $\Sigma_j$ has the same diffeomorphic type as $\Sigma$, a contradiction; otherwise, $k \geq 2$ and $\Sigma$ is degenerate stable, also a contradiction.
    \end{proof}
    
    Fix $\varepsilon$, $U$, and $\varepsilon'$ from the previous two lemmas. For $3 \leq q \leq \infty$, we set
    \begin{align*}
        \Gamma^q &:= \{C^q\text{-metric }\bg'\text{ on M}\},\quad
        \Gamma^q_{U^c}:= \{\bg' \in \Gamma^q : \bg'|_U = \bg|_U\}\,,\\
        \mathcal{M}^q_{U^c, \tilde \cK} &:= \{(\bg', \Sigma') : \bg' \in \Gamma^q_{U^c},\ \Sigma' \in \tilde \cK_{\bg'}\}\,.
    \end{align*}
    The topology of $\mathcal{M}^q_{U^c, \tilde \cK}$ is induced by the product topology  of the $C^q$-topology and $C^{2, \alpha}$-topology for $\alpha \in (0, 1)$.

    Following the proof of~\cite[Proposition~6.1, Lemma~6.5]{chuLi2024fiveTori} verbatim, we can conclude that for $3 \leq q < \infty$, $\mathcal{M}^q_{U^c, \cK}$ is a separable $C^{q - 2}$ Banach manifold and $\Pi: \mathcal{M}^q_{U^c, \cK} \to \Gamma^q_{U^c}$ is a $C^{q - 2}$ Fredholm map with index $0$. 
    Let $\Gamma^q_{U^c, \text{bumpy}}\subset \Gamma^q_{U^c}$ be the subset of regular values of $\Pi$. Then $\bg'\in\Gamma^q_{U^c, \text{bumpy}}$ if and only if for every $(\bg', \Sigma') \in \mathcal{M}^q_{U^c, \tilde \cK}$, $\Sigma'$ is non-degenerate. In this case,  every $\Sigma' \in \cK_{\bg'}$ is strongly non-degenerate. In addition, $\Gamma^q_{U^c, \text{bumpy}} \cap \Gamma^\infty$ is a generic set in $\Gamma^\infty$. 

    Fix $\hat \bg \in \Gamma^q_{U^c, \text{bumpy}} \cap \Gamma^\infty$ with $\|\hat \bg - \bg\|_{C^\infty} < \varepsilon'$ and $\hat \bg \leq \bg$. Then by the previous arguments, conclusions (1) and (2) hold. Moreover, since $\# \cK_{\hat \bg} < \infty$, there is a neighborhood $\cU \subset \subset \Gamma^\infty$ of $\hat \bg$ such that (1) and (2) also hold.
    
    Finally, applying~\cite[Lemma~6.6]{chuLi2024fiveTori} to the area of minimal surfaces in $\cO_{g_0, \leq L}(M, \hat \bg)$ and following the proof~\cite[Proposition~6.1]{chuLi2024fiveTori}, we can find $\bg'$ near $\hat \bg$ such that $\bg' \leq \hat \bg$ and all the statements (1), (2) and (3) hold.

\section{A lemma on cap product}\label{sect:capProd}
In this section we prove   Lemma \ref{lem:existsThetaAlgTop}.
We consider the following diagram.
\begin{equation}\label{eq:bigCommDiag}
\begin{tikzcd}[cramped,column sep=scriptsize]
	{H_{p}(X,A\cup B)} & {H^{q}(X,A)} & {H_{p-q}(X,B)} \\
	{H_{p}(C,C\cap(A\cup B))} & {H^{q}(C,C\cap A)} & {H_{p-q}(C,C\cap B)} \\
	{H_p(C,C\cap(A\cup B))} & {H^q(C,W_1\cap(A \cup B))} & {H_{p-q}(C,W_2\cap(A \cup B))} \\
	{H_p(C,C)} & {H^q(C,W_1)} & {H_{p-q}(C,W_2)}
	\arrow[r,phantom,"\times" description, from=1-1, to=1-2]
	\arrow[from=1-2, to=1-3]
	\arrow[from=1-2, to=2-2]
	\arrow[from=2-1, to=1-1]
	\arrow[r,phantom,"\times" description,from=2-1, to=2-2]
	\arrow[from=2-2, to=2-3]
	\arrow[from=2-3, to=1-3]
	\arrow[r,phantom,"||" description, from=3-1, to=2-1]
	\arrow[r,phantom,"\times" description,from=3-1, to=3-2]
	\arrow[from=3-1, to=4-1]
	\arrow[from=3-2, to=2-2]
	\arrow[from=3-2, to=3-3]
	\arrow[from=3-3, to=2-3]
	\arrow[from=3-3, to=4-3]
	\arrow[r,phantom,"\times" description,from=4-1, to=4-2]
	\arrow[from=4-2, to=3-2]
	\arrow[from=4-2, to=4-3]
\end{tikzcd}\end{equation}
Here, all  four horizontal arrows are given by cap products, and all the  vertical arrows   are induced by inclusions: Note that the inclusion
$$(C,C\cap A)\hookrightarrow (C,W_1\cap (A\cup B))$$ 
(for the middle arrow in the middle column)
and 
$$(C,W_2\cap (A\cup B))\hookrightarrow (C,C\cap B)$$
(for the middle arrow in the right column)
are well-defined because $W_2\cap A=\emptyset$ by assumption. Moreover, this diagram satisfies the naturality property for relative cap products (see \cite[p.241]{hatcher2002book}).

To prove Lemma \ref{lem:existsThetaAlgTop}, we need to ``pull  $[C]\frown\omega$ backward" along this chain of arrows induced by inclusions
\begin{equation}\label{eq:3maps}
    H_{p-q}(W_2,W_2\cap B)\rightarrow H_{p-q}(C,W_2\cap(A\cup B))\rightarrow H_{p-q}(C,C\cap B)\rightarrow H_{p-q}(X,B).
\end{equation}
  More precisely, we need to find some element   $\theta\in H_{p-q}(W_2,W_2\cap B)$ whose pushforward under this three maps is $[C]\frown\omega$. Hence, there are three steps.

{\bf Step 1}: Observe that $[C]\frown\omega$ has $[C]\frown\omega|_{C}\in H_{p-q}(C,C\cap B)$ (in the latter expression we treat $[C]\in H_p(C,C\cap(A\cup B))$) as a preimage, using the naturality for the first two rows of (\ref{eq:bigCommDiag}).

{\bf Step 2}: Then, noting $C\cap A=W_1\cap A\subset W_1\subset C$, we consider the exact sequence for this triple:
    \begin{equation}\label{eq:exactSeq1}
        H^q(C,W_1)\to H^q(C,C\cap A)\to H^q(W_1,C\cap A).
    \end{equation}
    Since the pullback $\omega|_{W_1}$ of $\omega$ under the inclusion $(W_1,W_1\cap A)\hookrightarrow (X,A)$ is zero by assumption, we know that $\omega|_C$ gets mapped to zero under the map
    $H^q(C,C\cap A)\to H^q(W_1,C\cap A)$ (recall $C\cap A=W_1\cap A$).
    Hence, using the exact sequence (\ref{eq:exactSeq1}), we know that under the map $H^q(C,W_1)\to H^q(C,C\cap A)$, there exists some $\mu$ that gets mapped to $\omega|_C$.
     Thus, there exists some $\mu'$ that gets mapped to $\omega|_C$ under the map $H^q(C,W_1\cap (A\cup B))\to H^q(C,C\cap A)$  
     (namely, we define $\mu'$ as the image of $\mu$ under the bottom arrow in the middle column of (\ref{eq:bigCommDiag})).
     Hence, using the second and the third row of (\ref{eq:bigCommDiag}) and by naturality, we know $[C]\frown\omega|_C$ has  $[C]\frown \mu'$ as a preimage, under the map
     $$H_{p-q}(C,W_2\cap(A\cup B))\rightarrow H_{p-q}(C,C\cap B).$$

{\bf Step 3}: Finally, based on the inclusions $W_2\cap B\subset W_2\subset C$, we consider the exact sequence
$$H_{p-q}(W_2,W_2\cap B)\to H_{p-q}(C,W_2\cap B)\to H_{p-q}(C,W_2).$$
Note the middle term here is the same as $H_{p-q}(C,W_2\cap(A\cup B))$ as $W_2\cap A=\emptyset$.
    Thus, in order to obtain the desired $\theta\in H_{p-q}(W_2,W_2\cap B)$ that gets mapped to $[C]\frown \mu'$ under the first arrow of (\ref{eq:3maps}), it suffices to show that $[C]\frown \mu'$ gets mapped to $0$ under the map
    \begin{equation}\label{eq:mapHpq}
        H_{p-q}(C,W_2\cap B)\to H_{p-q}(C,W_2).
    \end{equation}
    But now note that this map (\ref{eq:mapHpq}) is precisely the map at the bottom of the right column of (\ref{eq:bigCommDiag}). So we can deduce that $[C]\frown \mu'$  would indeed get mapped to $0$ from these facts:
    \begin{itemize}
        \item By definition $\mu$ is sent to  $\mu'$   under the map $H^q(C,W_1)\to H^q(C,W_1\cap (A\cup B))$.
        \item $H_p(C,C)=0$.
        \item The last two rows of (\ref{eq:bigCommDiag}) satisfy the naturality for cup product.
    \end{itemize}
This completes the proof of Lemma \ref{lem:existsThetaAlgTop}.

\section{Poincar\'e dual of cohomology classes via group action} \label{Append_PoincareDual}
  In this appendix, let $\tilde X$ be a compact simply connected manifold (possibly with boundary), $H$ be a finite group acting freely on $\tilde X$ from the right, and $X:= \tilde X/H$. By Hurewicz's theorem, there is a natural isomorphism $H_{ab}\otimes \Z_2=\pi_1(X)_{ab}\otimes \Z_2 \to H_1(X, \Z_2)$, called the Hurewicz map, where $H_{ab}=H/[H, H]$ denotes the abelianization of $H$. 
  
  We recall the following construction of the Poincar\'e dual of any cohomology class in $H^1(X, \Z)$ via $H$-equivarient functions on $\tilde X$.
  \begin{lem} \label{Lem_Poincare dual via line bundle}
    Given any group homomorphism $\phi: H\to O(1)=\{\pm 1\}$, one can associate a real line bundle to $\phi$ over $X$: \[
      E_\phi:= \tilde X\times_\phi \R := (\tilde X\times \R)/\left((p, v)\sim (ph, \phi(h^{-1})v) \right)\,.
    \]
    A smooth function $u\in C^\infty(\tilde X, \R)$ descends to a section of $E_\phi\to X$ if and only if
    \begin{align}
        u(ph) = \phi(h^{-1})u(p), \qquad \forall\, p\in \tilde X,\ \forall\, h\in H\,. \label{Equ_AssoBundle_u H-equivar}      
    \end{align}
    Moreover, for any such $u$, if in addition $du\neq 0$ at every point in $u^{-1}(0)$ and $d(u|_{\partial \tilde X})\neq 0$ at every point in $u^{-1}(0)\cap \partial \tilde X$, then $Z_u:= u^{-1}(0)/H$ is a smooth hypersurface of $X$ transverse to $\partial X$, whose Poincar\'e dual $\Omega_u\in H^1(X, \Z_2) \cong\Hom( H_1(X, \Z_2);\Z_2)$ is determined by 
    \begin{align}
        (-1)^{\langle \Omega_u, [h] \rangle} = \phi(h), \quad \forall\, h\in H \,; \label{Equ_PoincareDual vs gp action}
    \end{align}
    here $[h]\in H_1(X, \Z_2)$ is the image of $h\in H=\pi_1(X)$ under the Hurewicz map.
  \end{lem}
  \begin{proof}
    We only need to prove \eqref{Equ_PoincareDual vs gp action}, which is a  consequence of standard theory of vector bundle and Euler class. For the sake of completeness, we include a quick proof here. 
    
    For every $h\in H$, let $\tilde\gamma=\tilde\gamma_h: [0,1]\to \tilde X$ be a smooth path such that $\tilde\gamma(1)=\tilde\gamma(0)\,h$. Hence $\tilde\gamma$ descends to a loop $\gamma$ in $X$, and the induced homology $[\gamma]=[h]$ in $H_1(X, \Z_2)$. By a small perturbation, we may further assume that $\gamma$ is transverse to $Z_u$ and $\gamma(0)=\gamma(1)\notin Z_u$. Hence, $u\circ\tilde\gamma\in C^\infty([0, 1], \R)$ satisfies: 
    \begin{itemize}
      \item $(u\circ\tilde\gamma)'(t)\neq 0$ for every $t\in (u\circ\tilde\gamma)^{-1}(0)$;
      \item $u\circ\tilde\gamma(1) = u(\tilde\gamma(0)h)=\phi(h)\cdot u\circ\tilde\gamma(0)\neq 0$, where $\phi(h)\in \{\pm 1\}$;
      \item by Poincar\'e duality, $\langle\Omega_u, [h]\rangle=\#(Z_u\cap \gamma) = \# (u\circ\tilde\gamma)^{-1}(0) \mod2 \,.$
    \end{itemize}
    Considering the number of end points of the disjoint intervals $\{u\circ\tilde\gamma>0\}$, we see that \[
      \phi(h) = (-1)^{\# (u\circ\tilde\gamma)^{-1}(0)} = (-1)^{\langle\Omega_u, [h]\rangle} \,.
    \]
  \end{proof}

  \begin{cor} \label{Cor_Poincare dual via line bundles}
    For $\ell=1, \dots, m$, let $\phi_\ell: H\to O(1)=\{\pm 1\}$ be a group homomorphism, $\Omega_\ell\in H^1(X, \Z_2)=H_1(X, \Z_2)^*$ be the cohomology class determined by \[
      (-1)^{\langle\Omega_\ell, [h]\rangle} = \phi_\ell(h), \qquad \forall\, h\in H\,,
    \]
    and let $u_\ell\in C^\infty(\tilde X, \R)$ be satisfying \[
      u_\ell(ph) = \phi_\ell(h^{-1})u_\ell(p), \qquad \forall\, p\in \tilde X,\ \forall\, h\in H\,.  
    \]
    Suppose that
    \begin{itemize}
      \item $\{d u_\ell(p)\}_{\ell=1}^m$ are linearly independent in $T_p\tilde X$ for every $p\in \{u_1=\dots =u_m=0\}$;
      \item $\{d (u_\ell|_{\partial \tilde X})(p)\}_{\ell=1}^m$ are linearly independent in $T_p\tilde \partial X$ for every $p\in \{u_1=\dots =u_m=0\}\cap \partial \tilde X$.
    \end{itemize}    
    Then $Z:= \{u_1=\dots =u_m=0\}/H\subset X$ is a smooth submanifold of codimension $m$ transverse to $\partial X$, and $[(Z, \partial Z)]\in H^{\dim \cX - m}(X, \partial X; \Z_2)$ is the Poincar\'e dual of $\Omega_1\smile\dots \smile\Omega_m$.  
  \end{cor}
  \begin{proof}
    Since by Lemma \ref{Lem_Poincare dual via line bundle}, each $u_\ell$ descends to sections of a line bundle, we can perturb them to $\breve u_\ell$ which are still $H$-equivariant and such that
    \begin{itemize}
     \item $\{\breve u_1=\dots=\breve u_m=0\} = \{u_1=\dots=u_m=0\}$;
     \item for every $\ell$, $d \breve u_\ell\neq 0$ along $\{\breve u_\ell=0\}$, $d (\breve u_\ell|_{\partial \tilde X})\neq 0$ along $\{\breve u_\ell=0\}\cap \partial \tilde X$;
      \item $\{d \breve u_\ell(p)\}_{\ell=1}^m$ are linearly independent in $T_p\tilde X$ for every $p\in \{u_1=\dots =u_m=0\}$;
      \item $\{d (\breve u_\ell|_{\partial \tilde X})(p)\}_{\ell=1}^m$ are linearly independent in $T_p\tilde \partial X$ for every $p\in \{u_1=\dots =u_m=0\}\cap \partial \tilde X$.
    \end{itemize}
    Thus we can apply Lemma \ref{Lem_Poincare dual via line bundle} to get smooth hypersurfaces $Z_\ell:= \{\breve u_\ell=0\}/H \subset X$ transverse to $\partial X$, whose Poincar\'e duals are $\Omega_\ell$. Hence $Z=Z_1\cap \dots \cap Z_m$ has Poincar\'e dual $\Omega_1\smile\dots \smile\Omega_m$.
  \end{proof}

\section{Differential of real and imaginary part of complex polynomials} \label{Append_Diff Cplx polyn}
  In \S \ref{sect:topologyX}, we introduced some functions on $\mathbb S^3\subset \C^2$ under complex coordinates. The following lemma helps justify the surjectivity of their differentials.
  \begin{lem} \label{Lem_AssoBundle_nabla Re(f) vs partial_z f}
    For $n\geq 1$, we parametrize $\C^n=\R^{2n}$ by $(z_1, \dots, z_n)=(x_1, y_1, \dots,  x_n, y_n)$. For $f\in C^1(\C^n, \C)$, its differential is a $\C^{2n}$-valued function,
    \begin{align}
     \begin{split}
        \nabla \Re (f) + \mathbf{i} \nabla \Imag(f)  \  = \left(
         (\partial_{z_1}+\bar\partial_{z_1})f,\ \mathbf{i}(\partial_{z_1}-\bar\partial_{z_1})f,\ \dots,\ 
         (\partial_{z_n}+\bar\partial_{z_n})f,\ \mathbf{i}(\partial_{z_n}-\bar\partial_{z_n})f  
       \right) \,.
     \end{split} \label{Equ_AssoBundle_grad (Re f + Im f)}
    \end{align}
    In particular, if $1\leq k\leq n$ and $f_1, \dots, f_k\in C^1(\C^n, \C)$ satisfy $\bar\partial_z f_{\ell}=0$ for $\ell=1, \dots, k$, then for every $p\in \C^n$, the following are equivalent:
    \begin{itemize}
      \item $\{\nabla \Re(f_\ell)|_p, \nabla \Imag(f_\ell)|_p\}_{\ell=1}^k$ are $\R$-linearly independent in $\R^{2n}$;
      \item $\{(\partial_zf_{\ell})|_p\}_{\ell=1}^k$ are $\C$-linearly independent in $\C^n$, where $\partial_z f:= (\partial_{z_1}f, \dots, \partial_{z_n}f)$.
    \end{itemize}
  \end{lem}
  \begin{proof}
    \eqref{Equ_AssoBundle_grad (Re f + Im f)} follows from the fact that $\partial_{x_\ell}=\partial_{z_\ell}+\bar\partial_{z_\ell}$, $\partial_{y_\ell}=\mathbf{i}(\partial_{z_\ell}-\bar\partial_{z_\ell})$.

    When $\bar\partial_z f=0$, by taking real and imaginary parts of \eqref{Equ_AssoBundle_grad (Re f + Im f)}, we have 
    \begin{align*}
      \nabla \Re(f) & = \big( 
        \Re(\partial_{z_1}f),\ -\Imag(\partial_{z_1}f),\ \dots,\  
        \Re(\partial_{z_n}f),\ -\Imag(\partial_{z_n}f) 
      \big) \,, \\
      \nabla \Imag(f) & = \big( 
        \Imag(\partial_{z_1}f),\ \Re(\partial_{z_1}f),\ \dots,\
        \Imag(\partial_{z_n}f),\ \Re(\partial_{z_n}f) 
      \big) \,.
    \end{align*} 
    Hence we see that $J\nabla \Re(f) = \nabla \Imag(f)$, where $J:\R^{2n}\to \R^{2n}$ is  the standard almost complex structure given by $J(x_1, y_1, \dots, x_n, y_n):= (-y_1, x_1, \dots, -y_n, x_n)$.
    
    When $\bar\partial_z f_\ell=1$ for all $\ell \in \{1, \dots, k\}$, the equivalence of the two bullet points then follows from \eqref{Equ_AssoBundle_grad (Re f + Im f)} and the standard fact that for vectors $u_1, \dots u_k\in\R^{2n}$, the following are equivalent:
    \begin{itemize}
      \item $\{u_\ell, Ju_\ell\}_{\ell=1}^k$ are $\R$-linearly independent in $\R^{2n}$;
      \item $\{u_\ell + \mathbf{i}Ju_\ell\}_{\ell=1}^k$ are $\C$-linearly independent in $\C^{2n}$.
    \end{itemize}
  \end{proof}

\section{Unknotted genus 1 surfaces and Hopf links}

\subsection{An extension lemma of Hopf links} \label{Append_HopfLink}

  For a closed manifold $M$, we denote 
  \begin{align*}
    \Gamma(M) & := \{\text{unoriented embedded loops in }M\}\,; \\
    \tilde\Gamma(M) & := \{\text{oriented embedded loops in }M\}\,.
  \end{align*}
  And let $\mathrm p: \tilde\Gamma(M)\to \Gamma(M)$ be the orientation forgetting map (which is a double cover). 
  Also, let $L:= \beta_+^H\sqcup \beta_-^H\subset S^3$ be the standard Hopf link in $S^3$, which consists of two great circles \[
    \beta^H_+:= \{(x,y,0,0): x^2+y^2=1\}, \qquad
    \beta^H_-:= \{(0,0,x,y): x^2+y^2=1\}\,.    
  \]
  And let
  \begin{align*}
    \cE(L, S^3) & := \{\text{unparametrized Hopf links in }S^3\} = \Diff(S^3)/\Diff(S^3, L) \,, \\
    \cE(\hat L, S^3) & := \{\text{unparametrized labeled Hopf links in }S^3\}
     = \Diff(S^3)/\Diff(S^3, \hat L)\,,
  \end{align*}
  where $\Diff(S^3, L)$ denotes the subgroup of all diffeomorphisms $\phi\in \Diff(S^3)$ such that $\phi(L)=L$, and $\Diff(S^3,\hat L)$ denotes that of all $\phi\in \Diff(S^3)$ such that $\phi(\beta^H_\pm)=\beta^H_\pm$, which constitutes a subgroup of $\Diff(S^3, L)$. Thus, every element of $\cE(\hat L, S^3)$ can be represented by an ordered pair $(\beta_+, \beta_-)$ whose union forms a Hopf link in $S^3$. 
  Let $\mathrm{q}: \cE(\hat L, S^3)\to \cE(L, S^3)$ be the label-forgetting map (which is a double cover); and $\mathrm{r}_\pm: \cE(\hat L, S^3) \to \Gamma(S^3)$ be the map given by $(\beta_+, \beta_-)\mapsto \beta_\pm$.

  The topology of $\cE(L, S^3)$ has been studied by Boyd-Bregman  in \cite{BoydBregman2025_HopfLink}.  We collect a key consequence of it:
  \begin{lem} \label{Lem_Hopflink_Topology}
    The fundamental group $\pi_1(\cE(L, S^3)) = \Z_2\oplus \Z_2$. Moreover, a loop $\gamma: \bS^1\to \cE(L,S^3)$ is trivial in $\pi_1(\cE(L, S^3))$ if and only if it lifts to a loop $\hat \gamma: \bS^1\to \cE(\hat L, S^3)$ and $\mathrm{r}_+\circ \hat\gamma$ lifts to a loop $\tilde\gamma: \bS^1\to \tilde\Gamma(S^3)$. 
  \end{lem}
  \begin{proof}
    If $[\gamma]=0$ in $\pi_1(\cE(L, S^3))$, let $f: D^2\to \cE(L, S^3)$ be a continuous extension of $\gamma$. Then by the contractibility of $D^2$, there is a lift $\hat f: D^2\to \cE(\hat L, S^3)$ of $f$ and a lift $\tilde f: D^2\to \tilde\Gamma(S^3)$ of $\mathrm{r}_+\circ \hat f$. Restricting to $\partial D^2 = \bS^1$ gives the desired lifts $\hat\gamma$ and $\tilde\gamma$.
    
    Conversely suppose the lifts $\hat\gamma$ and $\tilde\gamma$ exist. Let $\Gamma_\circ(S^3)\subset \Gamma(S^3)$ be the subspace of unoriented {\it great circles} in $S^3$, and
    \begin{align*}
      \cR(L, S^3) & := \{R(L): R\in O(4)\} \subset \cE(L, S^3) \,, \\
      \cR(\hat L, S^3) & := \mathrm{q}^{-1} (\cR(L, S^3)) \subset \cE(\hat L, S^3) \,.
    \end{align*} 
    By \cite[Theorem 5.2]{BoydBregman2025_HopfLink}, $\cR(L, S^3)\hookrightarrow \cE(L, S^3)$ is a homotopy equivalence; by \cite[Appendix (7)]{Hat83}, $\Gamma_\circ(S^3)\hookrightarrow \Gamma(S^3)$ is also a homotopic equivalence. Then by chasing the  diagram 
    \begin{equation*}
     \begin{tikzcd}[cramped,column sep=scriptsize]
       {\Gamma_\circ(S^3)} & {\cR(\hat L, S^3)} & {\cR(L, S^3)} \\
       {\Gamma(S^3)} & {\cE(\hat L, S^3)} & {\cE(L, S^3)} 
       \arrow[from=1-2, to=1-1, "\mathrm{r}_+"]
       \arrow[from=1-2, to=1-3, "\mathrm{q}"]
       \arrow[from=2-2, to=2-1, "\mathrm{r}_+"]
       \arrow[from=2-2, to=2-3, "\mathrm{q}"]
       \arrow[from=1-1, to=2-1]
       \arrow[from=1-2, to=2-2]
       \arrow[from=1-3, to=2-3]
     \end{tikzcd}
    \end{equation*}
    there exists $\hat\gamma_\circ: \bS^1\to \cR(\hat L, S^3)$ such that $\mathrm{q}\circ \hat\gamma_\circ$ is homotopic to $\gamma$ and $\mathrm{r}_+\circ \hat\gamma_\circ: \bS^1\to \Gamma_\circ(S^3)$ admits a lift into $\mathrm{p}^{-1}(\Gamma_\circ(S^3))\subset \tilde\Gamma(S^3)$. Notice that $\mathrm{p}^{-1}(\Gamma_\circ(S^3))$ is isomorphic to the space of oriented $2$-planes in $\R^4$, which is diffeomorphic to $S^2\times S^2$ and hence is simply connected. Hence $\mathrm{r}_+\circ \hat\gamma_\circ: \bS^1\to \Gamma_\circ(S^3)$ is homotopically trivial in $\Gamma_\circ(S^3)$. Noticing that $\mathrm{r}_+|_{\cR(\hat L, S^3)}$ is an isomorphism between $\cR(\hat L, S^3)$ and $\Gamma_\circ(S^3)$, we conclude that $\hat\gamma_\circ $ is homotopically trivial in $\cR(\hat L, S^3)$, so $\gamma\simeq \mathrm{q}\circ\hat\gamma_\circ$ is also homotopically trivial in $\cE(L, S^3)$.
  \end{proof}
  
   In the following, we study the relation between unknotted punctate surfaces of genus $1$ and Hopf links. We recall that the notion of a Hopf link {\it bounded by} an unknotted element $S\in \cS_1(S^3)$ was described in \S \ref{sect:XiProper} (before Proposition \ref{prop:wholeFamilyUnknotted}).
   The following lemma shows that the space of Hopf links bounded by any given such {\it smooth} $S$ is path-connected. 
   \begin{lem}\label{Lem_IsotopyHopfLink}
     Let $T\in \cS_1(S^3)$ be a smooth  surface of genus $1$ that is unknotted, and  $U_\pm$ be the two open regions bounded by $T$ (i.e. $S^3\setminus T = U_+\sqcup U_-$ such that $\partial U_\pm = T$). Suppose for $i=0, 1$, $\beta_i^\pm \subset U_\pm$ are embedded loops such that $\beta_i^+\cup \beta_i^-$ forms a Hopf link. Then there exists an isotopy $t\mapsto \beta^\pm(t)$ of embedded loops over $t\in [0, 1]$ such that
     \begin{itemize}
        \item $\beta^\pm(i) = \beta^\pm_i$, $i=0, 1$;
        \item for every $t_\pm\in [0, 1]$, $\beta^\pm(t_\pm)\subset U_\pm$ and $\beta^+(t_+)\cup\beta^-(t_-)$ is also a Hopf link.
     \end{itemize}
   \end{lem}
   \begin{proof}
     By the Schoenflies theorem in $S^3$, every connected component of $U_\pm$ whose boundary is a $2$-sphere must be contractible, and hence cannot contain either of $\beta_i^\pm$.  Also, any isotopy of closed curves in a region $U$ can be  deformed to avoid any given $3$-dimensional ball in $U$. Hence by inductively subtracting all the inner-most or outer-most spherical components (which still does not affect the assumptions on $\beta_i^\pm$), we may assume without loss of generality that $T$ is diffeomorphic to a smooth torus and $U_\pm$ are connected. 
     
     By the van Kampen theorem, one of the homomorphisms $\iota^\pm: \pi_1(T)\to \pi_1(U_\pm)$ induced by natural inclusion is not injective:  Say $\iota^+$ is not injective. Then by Dehn's Lemma, there exists an embedded nontrivial loop $\gamma\subset T$ which bounds an embedded disc $D\subset \overline{U_+}$. When $\epsilon\ll 1$, after smoothing the corners, $\overline{U_+}\setminus B_\epsilon(D)$ is a smooth domain with boundary diffeomorphic to $S^2$, hence by Schoenflies theorem in $S^3$, $\bar U_+\setminus B_\epsilon(D)$ (after smoothing out corners) is diffeomorphic to $B^3$. Therefore, $\overline{U_+}$ is diffeomorphic to $B^3$ attached with a $1$-handle, which is a solid torus $D^2\times S^1$. In summary, $U_+$ is diffeomorphic to the tubular neighborhood of some knot $K\subset U_+\subset S^3$.

     \paragraph{\textbf{Claim}} For every $i=0,1$ and every map from the disc $f:D\to \overline{U_+}\setminus \beta_i^+$ such that $f(\partial D)\subset \partial U_+ = T$, we have $f_*[\partial D] = 0$ in $H_1(T)$.

     \begin{proof}[Proof of Claim.]
       Since $U_+$ is diffeomorphic to a tubular neighborhood of $K$, $H_1(U_+)$ is generated by $[K]$.
       Set $[\beta_i^+] = m_i[K]$ in $H_1(U_+)$ for some integer $m_i$. By the Mayer–Vietoris sequence \[
         0 = H_2(S^3)\to H_1(T) \to H_1(U_-)\oplus H_1(U_+) \to H_1(S^3) = 0\,,
       \] 
       so there exists a loop $\lambda\subset T$ such that $[\lambda]\in H_1(T)$ is mapped to $(0, [K])\in H_1(U_-)\oplus H_1(U_+)$. Hence, $m_i\lambda$ is homologous to $\beta_i^+$ in $U_+$, and there is some map $\Sigma\to U_-$ from some surface $\Sigma$ such that $\partial \Sigma$ is mapped to $\lambda$. Therefore, \[
         \pm1 = \link(\beta_i^-, \beta_i^+) = \text{ algebraic intersection of }\beta_i^- \text{ and }m_i\Sigma\,.
       \]
       Thus $m_i=\pm 1$. 

       Now since $f(D)\cap \beta^+_i = \emptyset$, the algebraic intersection of $f(D)$ and $K$ is also $0$. This forces $f_*[\partial D] = 0$ in $H_1(T)$ by Poincar\'e duality.
     \end{proof}

     Now that with this claim, by the Hurewicz theorem, the natural inclusion $T\hookrightarrow \overline{U_+}\setminus \beta_i^+$ induces an injective map $\pi_1(T)\to \pi_1(U_+\setminus \beta_i^+)$, hence by the van Kampen theorem, the homomorphism induced by inclusion \[
       \pi_1(U_-\setminus \beta_i^-) \to \pi_1(S^3\setminus (\beta_i^+\cup \beta_i^-))
     \]
     is also injective. Since $\beta_i^+\cup \beta_i^-$ forms a Hopf link, the right hand side is an abelian group, and so is the left hand side. And since $U_+$ is diffeomorphic to a tubular neighborhood of $K$, we know \[
       \pi_1(S^3\setminus (K\cup \beta_i^-)) = \pi_1(U_-\setminus \beta_i^-)
     \]
     is also an abelian group. By \cite[Remark 1.7]{XieZhang23_LinkGp} (see also \cite[Section 6.3]{Kawauchi96_SurveyKnot}), $K\cup\beta_i^-$ is also a Hopf link, and then $K$ is an unknot, forcing $U^- = S^3\setminus \overline{U_+}$ to be also diffeomorphic to the tubular neighborhood of some (unknotted) embedded loop. Therefore, for each $i$, there exist $\varepsilon>0$ and an embedded disc $D_i\subset S^3$ with boundary $\beta_i^-$ such that $D_i\cap B_{\varepsilon}(K)$ is an $\varepsilon$-disc $D_i'$ whose boundary is a meridian of $B_{\varepsilon}(K)$. We can then construct isotopy from $\beta_0^-$ to $\partial D_0'$ in $D_0$, from $\partial D_0'$ to $\partial D_{1}'$ in $\partial B_{\varepsilon}(K)$ and from $\partial D_{1}'$ to $\partial D_{1} = \beta_{1}^-$ in $D_{1}$. Since all these isotopies are not intersecting $B_{\varepsilon/2}(K)$, they can be concatenated and deformed to an isotopy from $\beta_0^-$ to $\beta_{1}^-$ in $U_-$. In summary, there exists an isotopy $[0, 1]\ni t\mapsto \beta^-(t)\subset U_-$ such that $\beta^-(i) = \beta^-_i$ for every $i=0,1$. The same process also produces an isotopy $[0, 1]\mapsto \beta^+(t)\subset U_+$ such that $\beta^+(i) = \beta^+_i$ for every $i=0,1$. These isotopies satisfy the desired properties.
   \end{proof}

   When $S$ is not smooth, the lemma above may not be true. The following lemma allows us to perturb a punctate surface to a smooth one that still bounds a compact family of Hopf links.

   \begin{lem}\label{Lem_Hopflink_SmoothingSurface}
{
     Let $T\in \cS_1(S^3)$ be a punctate surface of genus $1$ that is unknotted, and  $U_\pm$ be the two open regions bounded by $T$ (i.e. $S^3\setminus T = U_+\sqcup U_-$ such that $\partial U_\pm = T$). 
     Suppose $K$ is a compact metric space and $\beta^\pm:K\to \Gamma(S^3)$ are continuous maps such that $\beta^\pm(q) \subset U_\pm$ and $\beta^+(q)\cup \beta^-(q)$ forms a Hopf link for every $q\in K$.  
     Then there exists a smooth surface $\tilde T\in \cS_1(S^3)$ which bounds two open regions $\tilde U_\pm$ such that $\beta^\pm(q) \subset \tilde U_\pm$ for every $q\in K$. 
}
   \end{lem}
   \begin{proof}
     Since $K$ is compact and punctate set of $T$ is finite, sufficiently small balls $B_\epsilon$ centered at the punctate set $P$ do not intersect any $\beta^\pm(q)$. Hence by subtracting these small balls and filling in the boundary by discs in these balls, we obtain a smooth genus $1$ surface $\tilde T\subset S^3$ that bounds open regions $\tilde U_\pm$ which agree with $U_\pm$ outside $B_\epsilon(P)$. Therefore for every $q\in K$, \[
       \beta^\pm(q) \subset U_\pm \setminus B_\epsilon(P) = \tilde U_\pm \setminus B_\epsilon(P) \subset \tilde U_\pm \,.
     \]
   \end{proof}

   The following extension lemma of Hopf links is crucial in the proof of Theorem \ref{thm:XiProper}.  

   \begin{lem} \label{Lem_Hopflink_ExtLoops}
     Suppose $W$ is a pure simplicial $2$-complex, $\Phi: W\to \cU\subset \cS_1(S^3)$ is a Simon--Smith family of unknotted punctate surfaces of genus $1$.
     Suppose $\Upsilon: \partial W\to \cE(L, S^3)$ is a continuous map such that for every $q\in \partial W$, $\Upsilon(q)$ is a Hopf link bounded by $\Phi(q)$.
     Then there exists a continuous map $\bar\Upsilon: W\to \cE(L,S^3)$ such that $\bar\Upsilon|_{\partial W} = \Upsilon$.
   \end{lem}
   \begin{proof}
     The extension will be constructed cell by cell. 
     \paragraph{\textbf{Claim}} After refining $W$, there exists  a map that associates to each $2$-cell $\sigma$ of $W$ a Hopf link $\mathfrak H(\sigma)$ such that,
     \begin{enumerate} [label=\normalfont(\roman*)]
       \item\label{Item_ExtLem_BdyConti} for every $q\in \partial W$ and $q'\in W$ in the same $2$-cell of $q$, we have $\Upsilon(q)$ is bounded by $\Phi(q')$;
       \item\label{Item_ExtLem_BdyCptble} for every $2$-cell $\sigma$ such that $\sigma\cap \partial W \neq \emptyset$, there exists $q_\sigma\in \sigma\cap \partial W$ such that $
         \mathfrak H(\sigma) = \Upsilon(q_\sigma)$.
       \item\label{Item_ExtLem_Conti} for every $2$-cell $\sigma$, and for every point $q'$ in $\sigma$ or some adjacent $2$-cell of $\sigma$, we have  $\mathfrak H(\sigma)$ is bounded by $\Phi(q')$. 
     \end{enumerate}
     \begin{proof}[Proof of Claim.]
       By the definition of $\cU$, for every point $q\in W$, there exists a Hopf link $\breve\Upsilon(q)\subset S^3$ bounded by $\Phi(q)$ (of course, $q\mapsto \breve\Upsilon(q)$ need not be a continuous map), and when $q\in \partial W$, we can just choose $\breve\Upsilon(q) = \Upsilon(q)$. 
       By the closedness of $\Phi$ (in Definition \ref{def:Simon_Smith_family}), there exists $\delta_q>0$ such that $\breve\Upsilon(q)$ is bounded by $\Phi(q')$ for every $q'\in B_{\delta_q}(q)\subset W$. Since $q\mapsto \breve\Upsilon(q)$ is continuous on $\partial W$, we can even ask $\delta_{\partial W}:=\inf_{q\in\partial W}\delta_q >0$.

       We can then choose a finite subset $J\subset W$ such that $\{B_{\delta_q/10}(q)\}_{q\in J}$ forms a cover of $W$. 
       Hence, possibly after refining $W$, 
       \begin{itemize}
         \item every $2$-cell $\sigma$ of $W$ is contained in $B_{\delta_{q(\sigma)}/4}(q(\sigma))$ for some $q(\sigma)\in J$;
         \item every $2$-cell $\sigma$ that intersects $\partial W$ is contained in $B_{\delta_{\partial W}/4}(\partial W)$.
       \end{itemize}
       For every $2$-cell $\sigma$ of $W$, define
       \begin{align*}
         \mathfrak H(\sigma) := \begin{cases}
            \Upsilon (q) &\ \text{ if }\sigma\cap \partial W\neq \emptyset,\text{ where }q\text{ is some point in }\sigma\cap \partial W\,; \\
            \breve\Upsilon(q(\sigma)) &\ \text{ if } \sigma\cap \partial W = \emptyset\,.
         \end{cases}
       \end{align*}
       It is easy to verify from the definition that \ref{Item_ExtLem_BdyConti} - \ref{Item_ExtLem_Conti} in Claim hold. 
     \end{proof}

     Now we are ready to construct the extension $\bar\Upsilon$.
     \paragraph{\textbf{Step 1 ($0$-cells)}} For every $0$-cell $q\in W$, define\footnote{The choice is not unique, but any such choice works for later purpose.} 
     \begin{align*}
       \bar\Upsilon(q):= \begin{cases}
          \Upsilon(q), &\ \text{ if }q\in \partial W\,; \\
          \mathfrak H(\sigma), &\ \text{ if }q\notin \partial W, \ \text{ where }\sigma \text{ is a 2-cell that contains }q\,.
       \end{cases}
     \end{align*}
{
     And for each $q\in W$, since by \ref{Item_ExtLem_BdyConti} and \ref{Item_ExtLem_Conti} \[
       \cL(q):= \{\bar\Upsilon(q'): q'\in \partial W\cap \sigma, \text{ 2-cell }\sigma\ni q\}\cup \{\mathfrak H(\sigma'): \sigma' \text{ or its adjacent 2-cell contains }q\}
     \]
     is a compact family of Hopf links bounded by the punctate surface $\Phi(q)$, by Lemma \ref{Lem_Hopflink_SmoothingSurface} there exists a smooth genus $1$ surface $\tilde\Phi(q)$ that also bounds every Hopf links in $\cL(q)$\footnote{We do not ask $\tilde\Phi(q)$ to vary continuously in $q$.}. 
}
     \paragraph{\textbf{Step 2 ($1$-cells)}} For every $1$-cell $e\subset \partial W$, define $\bar\Upsilon|_e = \Upsilon_+|_e$. 
     
     If $e\not\subset \partial W$, let $q_1, q_2\in W$ be the end points of $e$ (which are $0$-cells of $W$) and $q\in W$ be the midpoint of $e$. By connectedness of the space of Hopf links bounded by  {$\tilde\Phi(q)$} (see Lemma \ref{Lem_IsotopyHopfLink}), there exists a continuous isotopy $\bar\Upsilon|_e: e\to \cE(L, S^3)$ extending $\bar\Upsilon|_{q_i}$, such that $\bar \Upsilon(q')$ is bounded by  {$\tilde\Phi(q)$} for every $q'\in e$. 
\medskip

     \paragraph{\textbf{Step 3 ($2$-cells)}} For every fixed $2$-cell $\sigma$ of $W$, Step 2 produces a map $\bar\Upsilon|_{\partial \sigma}: \partial \sigma\to \cE(L, S^3)$ which agrees with $\Upsilon$ on $\partial W$ when $\partial \sigma\cap \partial W \neq \emptyset$.
     Moreover, if we fix a lift $(\beta_+(\sigma), \beta_-(\sigma))\in \cE(\hat L, S^3)$ of $\mathfrak H(\sigma)$, then by Claim 1 \ref{Item_ExtLem_Conti}, there is a unique pair of maps $U_\pm: \sigma \to \{\text{open subsets of }S^3\}$ such that for every $q'\in \sigma$,
     \begin{align*}
       S^3\setminus \Phi(q') = U_+(q') \sqcup U_-(q'), \quad
       \partial U_\pm(q') = \Phi(q')\,; \quad
       \beta_\pm(\sigma)\subset U_\pm(q') \,.
     \end{align*}
     For every $0$-cell $q_0\in \sigma$, this then determines a unique lift $(\bar\Upsilon_+(q_0), \bar\Upsilon_-(q_0))\in \cE(\hat L, S^3)$ of $\bar\Upsilon(q_0)$ such that 
     \begin{align*}
       \bar\Upsilon_\pm(q_0)\subset U_\pm(q')\,, \qquad \forall\, q'\in \sigma \,.  
     \end{align*}
 {
     In particular for every $1$-cell $e\subset \sigma$ with midpoint $q_e$, there exist two regions $\tilde U_\pm(q_e)$ bounded by $\tilde\Phi(q_e)$ such that 
     \begin{align}
        \beta_\pm(\sigma) \subset \tilde U_\pm(q_e)\,, \quad
        \bar\Upsilon_\pm(q_0) \subset \tilde U_\pm(q_e), \quad \forall\text{ 0-cell }q_0\in \sigma\,. \label{Equ_Hopflink_Choice of label}
     \end{align}
}
     Since $\bar\Upsilon|_e$ are bounded by  {$\tilde\Phi(q_e)$}, we see that \[
       \bar\Upsilon_\pm(q_1):= \bar \Upsilon(q_1)\cap {\tilde U_\pm(q_e)}, \quad q_1\in e
     \]
     defines a continuous lift $(\bar\Upsilon_+, \bar\Upsilon_-): e\to \cE(\hat L, S^3)$ of $\bar \Upsilon|_e$, which agrees with \eqref{Equ_Hopflink_Choice of label} on the end points.
     
     Putting them together, we conclude that there exists a lift $(\bar\Upsilon_+, \bar\Upsilon_-): \partial \sigma\to \cE(\hat L, S^3)$ of $\bar \Upsilon|_{\partial\sigma}$, and by \eqref{Equ_Hopflink_Choice of label} and Lemma \ref{Lem_IsotopyHopfLink}, $\bar\Upsilon_+(q_1) \cup \beta_-(\sigma)$ forms a Hopf link for every $q_1\in \partial \sigma$. Since the linking number is never zero, this forces that $\bar\Upsilon_+|_{\partial \sigma}: \partial \sigma\to \Gamma(S^3)$ lifts to a map $\partial \sigma\to \tilde\Gamma(S^3)$. Then by Lemma \ref{Lem_Hopflink_Topology}, $\bar\Upsilon|_{\partial \sigma}$ is trivial in $\pi_1(\cE(L, S^3))$, and hence there exists a continous map $\bar\Upsilon: \sigma\to \cE(L, S^3)$ extending $\bar\Upsilon|_{\partial \sigma}$. This finishes the construction of $\bar\Upsilon$.
   \end{proof}

\subsection{Unknotted  elements in $\cS_1(S^3)$}\label{sect:UnknottedTlements}
In this section, we prove Proposition \ref{prop:wholeFamilyUnknotted}.

Let us recall the notion of ``neck-pinch surgery" first. Let $S\in\cS(M)$ and $\gamma$ be a smooth embedded loop that lies on the smooth part of $S$. Suppose $\gamma$ bounds a smooth embedded disc $D$ whose interior lies entirely in $S^3\backslash S$. We can {\it perform a neck-pinch surgery along $\gamma$, using the disc $D$}: Namely, we remove a short cylinder on $S$ around $\gamma$ (whose geometry is like $\gamma\x[-\epsilon,\epsilon]$), glue back two discs, each being close to $D$ graphically, and finally smoothen it near the two gluing curves to obtain an element in $\cS(S^3)$. 

\begin{lem}\label{lem:neckPinchPreservesU}
    Let $S\in \cS_1(S^3)$  and let $S'$ be obtained from $S$ by performing a neck-pinch surgery, such that $S'$ still has genus $1$. Then   $S\in\cU$ if and only if $S'\in\cU$. 
\end{lem}

\begin{proof}For  the ``if" part, by definition there exist  some  loops $\alpha\subset \ins(S')$ and $\beta\subset\out(S')$ that form a Hopf link. When we undo the neck-pinch surgery, we can first deform the two loops to avoid the surgery region, so that after gluing back  the 1-handle to obtain $S$, we still have $\alpha\subset\ins(S)$   and $\beta\subset\out(S)$. So $S\in\cU$. 

To prove the ``only if" part, we first choose some  loops $\alpha\subset \ins(S)$ and $\beta\subset\out(S)$ that form a Hopf link.  Let $\gamma\subset S$ be the smooth embedded loop   along which the surgery is performed, and suppose $\gamma$ bounds a disc $D$ in $S^3\backslash S$. Without loss  of generality, we choose an inward direction for $S$, and assume that the interior of $D$ lies in $\ins(S)$. Let $S'$ be obtained from $S$ by performing a neck-pinch surgery along $\gamma$, using the disc $D$.

{\bf Step 1:} First, by definition, $S$ has at most finitely many singularities. We let $B_1, \dots,B_n\subset S^3$ be sufficiently small balls centered at these singularities such that:
\begin{itemize}
    \item Each $B_i$ avoids $\alpha$ and $\beta$.
    \item $S\backslash \cup^n_{i=1}B_i$ is still a genus 1 surface with boundary.
  \item Each $\partial B_i$ intersects $S$ on the smooth part, and   transversely.
    \item Each $B_i$ avoids some neighborhood of $D$ such that the surgery process would leave each $B_i$ intact.
\end{itemize} 
For each ball $B_i$, the intersection $S\cap \partial B_i$ consists of a finite union of smooth loops $\{\gamma_k\}_k$. We then perform neck-pinch surgery along these loops one-by-one: Namely, whenever such a loop $\gamma_k$ bounds a disc on $\partial B_i$ that does not contain any other loop $\gamma_{k'}$, we perform a surgery along $\gamma_k$ using the disc it bounds on $\partial B_i$, and then repeat the process for all $\gamma_k$, and  for all balls $B_i$. Then we {\it remove all components lying within $\cup_i B_i$}, and denote   the resulting surface at the end by $\tilde S$. Note that by assuming the balls $B_i$ to be sufficiently small, $\tilde S$ is a smooth genus one surface: This  follows from  $\fg(S)=1$ and the definition of genus for elements in $\cS(M)$ in  \cite[Definition 2.5]{chuLiWang2025optimalFamily}.

{\bf Step 2:} Now, by the way we chose the balls $B_i$, we would still have $\alpha\subset\ins(\tilde S)$ and $\beta\subset\out(\tilde S)$. Hence, since $\alpha\cup\beta$ is a Hopf link, $\tilde S$ is a genus one Heegaard surface. Then, using the assumption $\fg(S')=1$, we know that the surface  obtained \underline{from $\tilde S$} by performing a surgery along $\gamma$, using the disc $D$, has genus 1. Hence, the disc $D$ can be ``pushed"  (i.e. deformed via smooth embeddings) within $\overline{\ins(\tilde S)}$ to become a disc $\tilde D\subset \tilde S$, while keeping the boundary $\gamma$ fixed. Then, we can use the reverse of  this pushing procedure to deform $\alpha$, and obtain a loop $\tilde\alpha\subset\ins(\tilde S)$ that does not intersect $D$.  As a result, when we  perform  surgery on $\tilde S$ along $\gamma$ using the disc $D$, the loops  $\tilde\alpha $ and $\beta $ would be left intact. Call this surface $\tilde S'$

{\bf Step 3:} Finally, on the new surface $\tilde S'$, we {\it undo} the procedure described in Step 1. Namely, for $\tilde S'$, we put back the connected components discarded, and undo all the neck-pinch surgeries (near the balls $B_i$), by gluing back the 1-handles, so that we  obtain  the   surface $S'$ with singularities: This  $S'$ is precisely the surface  of interest $S'$ in Lemma \ref{lem:neckPinchPreservesU}. It is easy to see that $\alpha\subset\ins(S')$ and $\beta\subset\out(S')$, and  $\alpha\cup\beta$ is a Hopf link, so we have $S'\in\cU$, as desired. \end{proof}

We now prove Proposition \ref{prop:wholeFamilyUnknotted}. Suppose by contradiction there exists some $\Phi(t)\notin\cU$.
Let $$T:=\inf\{t\in[0,1]:\Phi(t)\notin\cU\}.$$
First, we claim that $\Phi(T)\notin \cU$. Indeed, if $\Phi(T)\in \cU$, then we can find embedded loops $\alpha\subset \ins(\Phi(T))$ and $\beta\subset\out(\Phi(T))$ such that they form a Hopf link: Note that we can consistently choose an inward direction for each $\Phi(t)$. Then using the definition of Simon--Smith family, for any $t$ sufficiently close to $T$, we have $\alpha\subset \ins(\Phi(t))$ and $\beta\subset\out(\Phi(t))$, and thus $\Phi(t)\in \cU$ for such $t$. This contradicts the definition of $T$. 

As a result, $T>0$, since $\Phi(0)\in\cU$ by assumption. Take $t_1,t_2, \dots\uparrow T.$ By the definition of Simon--Smith family, we can choose finitely many small open balls $B_j\subset S^3$ such that $\Phi(t_i)\to \Phi(T)$ smoothly outside $  B:=\cup_j B_j$. Without loss of generality, we can assume also that:
\begin{itemize}
    \item The boundary of each $B_j$ intersects each $\Phi(t_i)$ on the smooth part only, and such intersection is transverse, for each $t_i$. 
    \item $\Phi(T)\backslash B$ is a genus 1 surface with boundary.
\end{itemize}
Now, to derive a contradiction, we just need to show that for any sufficiently large $i$, we can choose some   embedded loops $\alpha_{t_i}\subset\ins(\Phi(t_i))\backslash   B$ and $\beta_{t_i}\subset\out(\Phi(t_i))\backslash  B$ that form a Hopf link. Indeed, once we have that, for some large $i$, we can ensure $\alpha_{t_i}\subset\ins(\Phi(T))\backslash   B$ and $\beta_{t_i}\subset\out(\Phi(T))\backslash   B$ through slight deformation, using the fact that $\Phi(t_i)\to \Phi(T)$ smoothly outside $  B=\cup_j B_j$. This shows that $\Phi(T)\in\cU$, which contradicts the claim $\Phi(T)\notin \cU$ we obtained from the previous paragraph.

Now, we fix $t_i$. We will modify the surface $\Phi(t_i)$ such that it avoids $\partial  B$ as follows. For each ball $B_j$, the intersection $\Phi(t_i)\cap \partial B_j$ consists of a finite union of smooth loops $\{\gamma_k\}_k$. We then perform neck-pinch surgery along these loops one-by-one: Namely, whenever such a loop $\gamma_k$ bounds a disc on $\partial B_j$ that does not contain any other loop $\gamma_{k'}$, we perform a surgery along $\gamma_k$, using the disc it bounds on $\partial B_j$, and then repeat this for all $\gamma_k$,  for all ball $B_j$. We denote   the resulting surface at the end by ${\Phi(t_i)'}$. As a result,   $ {\Phi(t_i)'}$ avoids $\partial   B$, and  is smooth outside $  B$.

By Lemma \ref{lem:neckPinchPreservesU} and $\Phi(t_i)\in\cU$, we know $ {\Phi(t_i)'}\in\cU$. Thus there exist some    embedded loops $\alpha_{t_i}\subset\ins( \Phi(t_i)')$ and $\beta_{t_i}\subset\out( {\Phi(t_i)'})$  that form a Hopf link. Now, by the fact that $\Phi(T)\backslash B$ has genus 1, and the smooth convergence $\Phi(t_i)\to\Phi(T)$ outside $B$, we know that for $i$  large, $\Phi(t_i)'\backslash B$ is a genus 1 surface with boundary and all the components of $ {\Phi(t_i)'}$ inside $  B$ have genus 0. Thus,  we can  further assume that $\alpha_{t_i}\subset\ins( {\Phi(t_i)'})\backslash  B$ and $\beta_{t_i}\subset\out( {\Phi(t_i)'})\backslash  B$: In each ball $B_i$, we for a moment ignore the genus 0 components, and then either push $\alpha_{t_i}$ to avoid $B_i$ (in the case $\partial B_i\subset \ins(\Phi(t_i)')$) or push $\beta_{t_i}$  to avoid $B_i$ (in the case $\partial B_i\subset \out(\Phi(t_i)')$). As a result, undoing all the neck-pinch surgeries by gluing back the 1-handles, we in fact have $\alpha_{t_i}\subset\ins( {\Phi(t_i)})\backslash  B$ and $\beta_{t_i}\subset\out( {\Phi(t_i)})\backslash  B$.  This is exactly what we need for a contradiction, as explained before.

This finishes the proof of Proposition \ref{prop:wholeFamilyUnknotted}.

 
\printbibliography

@article{bettiolPiccione2024nonplanarMinSpheres,
  title={Nonplanar minimal spheres in ellipsoids of revolution},
  author={Bettiol, Renato G. and Piccione, Paolo},
  journal={Annali della Scuola Normale Superiore di Pisa - Classe di Scienze},
  year={2024},
  publisher={Scuola Normale Superiore di Pisa}
}

@inproceedings{morton1967symmetricProduct,
  title={Symmetric products of the circle},
  author={Morton, Hugh R},
  booktitle={Mathematical Proceedings of the Cambridge Philosophical Society},
  volume={63},
  number={2},
  pages={349--352},
  year={1967},
  organization={Cambridge University Press}
}

@article{wangZhou2025improvedMembranes,
  title={Improved C 1, 1 Regularity for Multiple Membranes Problem},
  author={Wang, Zhichao and Zhou, Xin},
  journal={Peking Mathematical Journal},
  pages={1--57},
  year={2025},
  publisher={Springer}
}

@article{liWangYao2025minimalLensSpace,
  title={Minimal surfaces with low genus in lens spaces},
  author={Li, Xingzhe and Wang, Tongrui and Yao, Xuan},
  journal={Journal f{\"u}r die reine und angewandte Mathematik (Crelles Journal)},
  volume={2025},
  number={828},
  pages={175--218},
  year={2025},
  publisher={De Gruyter}
}

@article{karcherPinkallSterling1988new,
  title={New minimal surfaces in $S^3$},
  author={Karcher, Hermann and Pinkall, Ulrich and Sterling, Ivan},
  journal={Journal of Differential Geometry},
  volume={28},
  number={2},
  pages={169--185},
  year={1988},
  publisher={Lehigh University}
}

@article{bettiolPiccione2023bifurcationsCliffordTorus,
  title={Bifurcations of Clifford tori in ellipsoids},
  author={Bettiol, Renato G and Piccione, Paolo},
  journal={arXiv preprint arXiv:2309.13758},
  year={2023}
}

@article{BoydBregman2025_HopfLink,
  title={The embedding space of a Hopf link},
  author={Boyd, Rachael and Bregman, Corey},
  journal={arXiv preprint arXiv:2504.21806},
  year={2025}
}

@article{chuLiWang2025genus2PartI,
  title={Existence of genus 2 minimal surfaces in 3-spheres. I},
  author={Chu, Adrian Chun-Pong and Li, Yangyang and Wang, Zhihan},
  journal={arXiv preprint arXiv:2507.23239v1},
  year={2025}
}

@article{chuLiWang2025optimalFamily,
  title={Min-max theory and minimal surfaces with prescribed genus},
  author={Chu, Adrian Chun-Pong and Li, Yangyang and Wang, Zhihan},
  journal={arXiv preprint arXiv:2507.23239v2},
  year={2026}
}

@book {JamesPaul2001AlgebraicTopology,
    AUTHOR = {Davis, James F. and Kirk, Paul},
     TITLE = {Lecture notes in algebraic topology},
    SERIES = {Graduate Studies in Mathematics},
    VOLUME = {35},
 PUBLISHER = {American Mathematical Society, Providence, RI},
      YEAR = {2001},
     PAGES = {xvi+367},
      ISBN = {0-8218-2160-1},
   MRCLASS = {55-01 (57-01)},
  MRNUMBER = {1841974},
MRREVIEWER = {Don\ Shimamoto},
       DOI = {10.1090/gsm/035},
       URL = {https://doi.org/10.1090/gsm/035},
}

@article {Whi91,
    AUTHOR = {White, Brian},
     TITLE = {The space of minimal submanifolds for varying {R}iemannian
              metrics},
   JOURNAL = {Indiana Univ. Math. J.},
  FJOURNAL = {Indiana University Mathematics Journal},
    VOLUME = {40},
      YEAR = {1991},
    NUMBER = {1},
     PAGES = {161--200},
      ISSN = {0022-2518,1943-5258},
   MRCLASS = {58D10 (53C42)},
  MRNUMBER = {1101226},
MRREVIEWER = {Jo\~{a}o\ Lucas Marques Barbosa},
       DOI = {10.1512/iumj.1991.40.40008},
       URL = {https://doi.org/10.1512/iumj.1991.40.40008},
}

@misc{SS23,
      title={Optimal regularity for minimizers of the prescribed mean curvature functional over isotopies}, 
      author={Lorenzo Sarnataro and Douglas Stryker},
      year={2023},
      eprint={2304.02722},
      archivePrefix={arXiv},
      primaryClass={math.DG},
      url={https://arxiv.org/abs/2304.02722}, 
}

@article{LiWang2024NineTori,
  title={Existence of embedded minimal tori in three-spheres with positive Ricci curvature},
  author={Li, Xingzhe and Wang, Zhichao},
  journal={arXiv preprint arXiv:2409.10391},
  year={2024}
}

@article {Jost3geodesic1989,
    AUTHOR = {Jost, J\"urgen},
     TITLE = {A nonparametric proof of the theorem of {L}usternik and
              {S}chnirelman},
   JOURNAL = {Arch. Math. (Basel)},
  FJOURNAL = {Archiv der Mathematik},
    VOLUME = {53},
      YEAR = {1989},
    NUMBER = {5},
     PAGES = {497--509},
      ISSN = {0003-889X,1420-8938},
   MRCLASS = {53C22 (58E10)},
  MRNUMBER = {1019164},
MRREVIEWER = {Gudlaugur\ Thorbergsson},
       DOI = {10.1007/BF01324725},
       URL = {https://doi-org.proxy.library.cornell.edu/10.1007/BF01324725},
}

@article {Taimanov3geodesic1992,
    AUTHOR = {Ta\u imanov, I. A.},
     TITLE = {On the existence of three nonintersecting closed geodesics on
              manifolds that are homeomorphic to the two-dimensional sphere},
   JOURNAL = {Izv. Ross. Akad. Nauk Ser. Mat.},
  FJOURNAL = {Izvestiya Rossiiskoi Akademii Nauk. Seriya Matematicheskaya},
    VOLUME = {56},
      YEAR = {1992},
    NUMBER = {3},
     PAGES = {605--635},
      ISSN = {1607-0046,2587-5906},
   MRCLASS = {58E10},
  MRNUMBER = {1188331},
MRREVIEWER = {Dorin\ Andrica},
       DOI = {10.1070/IM1993v040n03ABEH002177},
       URL = {https://doi-org.proxy.library.cornell.edu/10.1070/IM1993v040n03ABEH002177},
}

@article {Lusternik1947,
    AUTHOR = {Lusternik, L.},
     TITLE = {Topology of functional spaces and calculus of variations in
              the large},
   JOURNAL = {Trav. Inst. Math. Stekloff},
  FJOURNAL = {Trav. Inst. Math. Stekloff},
    VOLUME = {19},
      YEAR = {1947},
     PAGES = {100},
   MRCLASS = {46.0X},
  MRNUMBER = {25083},
MRREVIEWER = {E.\ J.\ McShane},
}

@book {Klingenberg1977,
    AUTHOR = {Klingenberg, Wilhelm},
     TITLE = {Lectures on closed geodesics},
   EDITION = {Third},
 PUBLISHER = {Universit\"at Bonn, Mathematisches Institut, Bonn},
      YEAR = {1977},
     PAGES = {210 pp. (not consecutively paged)},
   MRCLASS = {53C20 (58E10)},
  MRNUMBER = {461361},
MRREVIEWER = {Y.\ Mut\^o},
}

@incollection {Ballmann3geodesic1978,
    AUTHOR = {Ballmann, Werner},
     TITLE = {Der {S}atz von {L}usternik und {S}chnirelmann},
 BOOKTITLE = {Beitr\"age zur {D}ifferentialgeometrie, {H}eft 1},
    SERIES = {Bonner Math. Schriften},
    VOLUME = {102},
     PAGES = {1--25},
 PUBLISHER = {Univ. Bonn, Bonn},
      YEAR = {1978},
   MRCLASS = {58E10 (53C22)},
  MRNUMBER = {520178},
MRREVIEWER = {Y.\ Mut\^o},
}

@article{Str84,
    AUTHOR = {Struwe, M.},
     TITLE = {On a free boundary problem for minimal surfaces},
   JOURNAL = {Invent. Math.},
  FJOURNAL = {Inventiones Mathematicae},
    VOLUME = {75},
      YEAR = {1984},
    NUMBER = {3},
     PAGES = {547--560},
      ISSN = {0020-9910,1432-1297},
   MRCLASS = {58E12 (35R35 49F15 53A10)},
  MRNUMBER = {735340},
MRREVIEWER = {Helmut\ Kaul},
       DOI = {10.1007/BF01388643},
       URL = {https://doi.org/10.1007/BF01388643},
}

@article {Whi87,
    AUTHOR = {White, Brian},
     TITLE = {Curvature estimates and compactness theorems in
              {$3$}-manifolds for surfaces that are stationary for
              parametric elliptic functionals},
   JOURNAL = {Invent. Math.},
  FJOURNAL = {Inventiones Mathematicae},
    VOLUME = {88},
      YEAR = {1987},
    NUMBER = {2},
     PAGES = {243--256},
      ISSN = {0020-9910,1432-1297},
   MRCLASS = {58E12 (49F10 53C20 58D10)},
  MRNUMBER = {880951},
MRREVIEWER = {Helmut\ Kaul},
       DOI = {10.1007/BF01388908},
       URL = {https://doi.org/10.1007/BF01388908},
}

@article{Zho16,
    AUTHOR = {Zhou, Xin},
     TITLE = {On the free boundary min-max geodesics},
   JOURNAL = {Int. Math. Res. Not. IMRN},
  FJOURNAL = {International Mathematics Research Notices. IMRN},
      YEAR = {2016},
    NUMBER = {5},
     PAGES = {1447--1466},
      ISSN = {1073-7928,1687-0247},
   MRCLASS = {53C22 (49Q05 58E10)},
  MRNUMBER = {3509932},
MRREVIEWER = {Anna\ Maria\ Candela},
       DOI = {10.1093/imrn/rnv184},
       URL = {https://doi.org/10.1093/imrn/rnv184},
}

@misc{Ko23b,
      title={Morse Index bound of simple closed geodesics on 2-spheres and strong Morse Inequalities}, 
      author={Dongyeong Ko},
      year={2023},
      eprint={2303.00644},
      archivePrefix={arXiv},
      primaryClass={math.DG},
      url={https://arxiv.org/abs/2303.00644}, 
}

@article {KW22Lawson_symmetry,
    AUTHOR = {Kapouleas, Nikolaos and Wiygul, David},
     TITLE = {The {L}awson surfaces are determined by their symmetries and
              topology},
   JOURNAL = {J. Reine Angew. Math.},
  FJOURNAL = {Journal f\"ur die Reine und Angewandte Mathematik. [Crelle's
              Journal]},
    VOLUME = {786},
      YEAR = {2022},
     PAGES = {155--173},
      ISSN = {0075-4102,1435-5345},
   MRCLASS = {53A10},
  MRNUMBER = {4434746},
MRREVIEWER = {Fei-Tsen\ Liang},
       DOI = {10.1515/crelle-2021-0087},
       URL = {https://doi.org/10.1515/crelle-2021-0087},
}

@article {GJ86,
    AUTHOR = {Gr\"uter, M. and Jost, J.},
     TITLE = {On embedded minimal disks in convex bodies},
   JOURNAL = {Ann. Inst. H. Poincar\'e{} Anal. Non Lin\'eaire},
  FJOURNAL = {Annales de l'Institut Henri Poincar\'e. Analyse Non
              Lin\'eaire},
    VOLUME = {3},
      YEAR = {1986},
    NUMBER = {5},
     PAGES = {345--390},
      ISSN = {0294-1449},
   MRCLASS = {49F10 (49F20 53A10 58E12)},
  MRNUMBER = {868522},
MRREVIEWER = {Harold\ Parks},
       URL = {http://www.numdam.org/item?id=AIHPC_1986__3_5_345_0},
}

@misc{Ko23a,
      title={Existence and Morse Index of two free boundary embedded geodesics on Riemannian 2-disks with convex boundary}, 
      author={Dongyeong Ko},
      year={2023},
      eprint={2309.09896},
      archivePrefix={arXiv},
      primaryClass={math.DG},
      url={https://arxiv.org/abs/2309.09896}, 
}

@misc{HK23,
      title={Free boundary minimal disks in convex balls}, 
      author={Robert Haslhofer and Daniel Ketover},
      year={2023},
      eprint={2307.01828},
      archivePrefix={arXiv},
      primaryClass={math.DG},
      url={https://arxiv.org/abs/2307.01828}, 
}

@article {HK19,
    AUTHOR = {Haslhofer, Robert and Ketover, Daniel},
     TITLE = {Minimal 2-spheres in 3-spheres},
   JOURNAL = {Duke Math. J.},
  FJOURNAL = {Duke Mathematical Journal},
    VOLUME = {168},
      YEAR = {2019},
    NUMBER = {10},
     PAGES = {1929--1975},
      ISSN = {0012-7094,1547-7398},
   MRCLASS = {49Q05 (49J35 53E10 58E12)},
  MRNUMBER = {3983295},
MRREVIEWER = {Hung\ Thanh\ Tran},
       DOI = {10.1215/00127094-2019-0009},
       URL = {https://doi.org/10.1215/00127094-2019-0009},
}

@article{Grayson89,
    AUTHOR = {Grayson, Matthew A.},
     TITLE = {Shortening embedded curves},
   JOURNAL = {Ann. of Math. (2)},
  FJOURNAL = {Annals of Mathematics. Second Series},
    VOLUME = {129},
      YEAR = {1989},
    NUMBER = {1},
     PAGES = {71--111},
      ISSN = {0003-486X,1939-8980},
   MRCLASS = {53C22 (58E10)},
  MRNUMBER = {979601},
MRREVIEWER = {Gudlaugur\ Thorbergsson},
       DOI = {10.2307/1971486},
       URL = {https://doi.org/10.2307/1971486},
}

@article {White89,
    AUTHOR = {White, Brian},
     TITLE = {Every three-sphere of positive {R}icci curvature contains a
              minimal embedded torus},
   JOURNAL = {Bull. Amer. Math. Soc. (N.S.)},
  FJOURNAL = {American Mathematical Society. Bulletin. New Series},
    VOLUME = {21},
      YEAR = {1989},
    NUMBER = {1},
     PAGES = {71--75},
      ISSN = {0273-0979,1088-9485},
   MRCLASS = {58E12 (53C42)},
  MRNUMBER = {994891},
MRREVIEWER = {A.\ J.\ Tromba},
       DOI = {10.1090/S0273-0979-1989-15765-0},
       URL = {https://doi.org/10.1090/S0273-0979-1989-15765-0},
}

@article{All72,
	title = {On the {First} {Variation} of a {Varifold}},
	volume = {95},
	issn = {0003-486X},
	url = {http://www.jstor.org/stable/1970868},
	doi = {10.2307/1970868},
	number = {3},
	urldate = {2018-04-21},
	journal = {Annals of Mathematics},
	author = {Allard, William K.},
	year = {1972},
    mrnumber = {0307015},
    zmnumber = {0252.49028},
	pages = {417--491},
}

@article{Sha17,
  title={Compactness of minimal hypersurfaces with bounded index},
  author={Sharp, Ben},
  journal={Journal of Differential Geometry},
  volume={106},
  number={2},
  pages={317--339},
  year={2017},
  publisher={Lehigh University}
}

@article{CD03,
  title={The min-max construction of minimal surfaces},
  author={Colding, Tobias  and De Lellis, Camillo },
  journal={Surveys in Differential Geometry},
  volume={8},
  number={1},
  pages={75--107},
  year={2003},
  publisher={International Press of Boston}
}

@article{MN21,
  title={Morse index of multiplicity one min-max minimal hypersurfaces},
  author={Marques, Fernando Cod\'a and Neves, Andr{\'e}},
  journal={Advances in Mathematics},
  volume={378},
  pages={107527},
  year={2021},
  publisher={Elsevier}
}

@book{hatcher2002book,
  title={Algebraic Topology},
  author={Hatcher, Allen},
  year={2002},
  publisher={Cambridge University Press}
}

@article{Hat83,
  title={A proof of the Smale conjecture},
  author={Hatcher, Allen },
  journal={Annals of Mathematics},
  pages={553--607},
  year={1983},
  publisher={JSTOR}
}

@article{LS47,
  title={Topological methods in variational problems and their application to the differential geometry of surfaces},
  author={Lyusternik, Lazar Aronovich and Shnirel'man, Lev Genrikhovich},
  journal={Uspekhi Matematicheskikh Nauk},
  volume={2},
  number={1},
  pages={166--217},
  year={1947},
  publisher={Russian Academy of Sciences, Steklov Mathematical Institute of Russian~…}
}

@article{DP10,
  title={Genus bounds for minimal surfaces arising from min-max constructions},
  author={De Lellis, Camillo and Pellandini, Filippo},
  journal={Journal f{\"u}r die reine und angewandte Mathematik (Crelles Journal)},
  volume={2010},
  number={644},
  pages={47--99},
  year={2010},
  publisher={Walter de Gruyter GmbH \& Co. KG}
}

@article{Ket19,
  title={Genus bounds for min-max minimal surfaces},
  author={Ketover, Daniel},
  journal={Journal of Differential Geometry},
  volume={112},
  number={3},
  pages={555--590},
  year={2019},
  publisher={Lehigh University}
}

@article{Smith82,
  title={On the existence of embedded minimal 2-spheres in the 3-sphere, endowed with an arbitrary Riemannian metric},
  author={Smith, Francis},
  journal={Doctoral dissertation},
  year={1982},
  publisher={University of Melbourne}
}

@article{Zho20,
  title={On the multiplicity one conjecture in min-max theory},
  author={Zhou, Xin},
  journal={Annals of Mathematics},
  volume={192},
  number={3},
  pages={767--820},
  year={2020},
  publisher={JSTOR}
}

@article{SarnataroStryker2023Optimal,
  title={Optimal regularity for minimizers of the prescribed mean curvature functional over isotopies},
  author={Sarnataro, Lorenzo and Stryker, Douglas},
  journal={arXiv preprint arXiv:2304.02722},
  year={2023}
}

@article{WangZhou23FourMinimalSpheres,
  title={Existence of four minimal spheres in {$S^3$} with a bumpy metric},
  author={Wang, Zhichao and Zhou, Xin},
  journal={arXiv preprint arXiv:2305.08755},
  year={2023}
}

@article{Law70,
  title={Complete minimal surfaces in S3},
  author={Lawson, H Blaine},
  journal={Annals of Mathematics},
  pages={335--374},
  year={1970},
  publisher={JSTOR}
}

@article{chu2025arbitraryGenus,
  title={Minimal surfaces with arbitrary genus in 3-spheres of positive Ricci curvature},
  author={Chu, Adrian Chun-Pong},
  journal={arXiv preprint arXiv:2508.06019},
  year={2025}
}

@book{Yau82,
  title={Seminar on differential geometry},
  author={Yau, Shing-Tung},
  number={102},
  year={1982},
  publisher={Princeton University Press}
}

@article{chuLi2024fiveTori,
  title={Existence of 5 minimal tori in 3-spheres of positive Ricci curvature},
  author={Chu, Adrian Chun-Pong and Li, Yangyang},
  journal={arXiv preprint arXiv:2409.09315},
  year={2024}
}

@article{XieZhang23_LinkGp,
  title={On meridian-traceless SU (2)--representations of link groups},
  author={Xie, Yi and Zhang, Boyu},
  journal={Advances in Mathematics},
  volume={418},
  pages={108947},
  year={2023},
  publisher={Elsevier}
}

@book{Kawauchi96_SurveyKnot,
  title={Survey on knot theory},
  author={Kawauchi, Akio},
  year={1996},
  publisher={Springer Science \& Business Media}
}
 
\end{document}